\documentclass[11pt, eqno]{article}
\usepackage{bbm}
\usepackage{mathrsfs}
\usepackage{amsfonts}
\usepackage{amssymb}
\usepackage{graphicx}
\usepackage[all]{xy}

\usepackage{amsthm}
\usepackage{amsmath}
\usepackage{amsmath,amssymb,latexsym,color}
\usepackage[mathscr]{eucal}
\usepackage{CJK}
\usepackage{cases}
\usepackage{graphics}

\newenvironment{dedication}
        {\vspace{3ex}\begin{quotation}\begin{center}\begin{em}}
        {\par\end{em}\end{center}\end{quotation}}

\usepackage{psfrag}
\usepackage{subfigure}
\usepackage{color}
\usepackage{amssymb,latexsym}
\usepackage{amsmath,latexsym}
\usepackage{amscd}

\newcommand{\N}{{\mathbb N}}

\newcommand{\R}{{\mathbb R}}

\newtheorem{theorem}{Theorem}[section]
\newtheorem{cor}[theorem]{Corollary}
\newtheorem{corollary}[theorem]{Corollary}
\newtheorem{defn}[theorem]{Definition}
\newtheorem{thm}[theorem]{Theorem}
\newtheorem{definition}[theorem]{Definition}
\newtheorem{example}[theorem]{Example}
\newtheorem{remark}[theorem]{Remark}

\newtheorem{lemma}[theorem]{Lemma}

\newtheorem{guess}[theorem]{Conjecture}
\newtheorem{prop}[theorem]{Proposition}

\begin{document}
\title{
Coisotropic Hofer-Zehnder capacities
of convex domains\\ and related results}
\date{September 18, 2019\\
First revision July 30, 2021\\
Second revision July 12, 2022\\
Final version March 25, 2023
}

\author{Rongrong Jin\footnote{Partially supported
by Scientific Research Foundation of CAUC (No:2020KYQD107).}\; and\; Guangcun Lu
\thanks{Corresponding author
\endgraf \hspace{2mm} Partially supported
by the NNSF  11271044 of China and the Fundamental Research Funds for Central Universities.
\endgraf\hspace{2mm} 2010 {\it Mathematics Subject Classification.}
 53D35, 53C23 (primary), 70H05, 37J05, 57R17 (secondary).}}
 \maketitle \vspace{-0.3in}

\vspace{0.3in}
\begin{dedication}
\hfill{Dedicated to Professor Leonid Polterovich on the occasion of his sixtieth birthday}
\end{dedication}
\vspace{0.1in}

\abstract{
We prove representation formulas for the  coisotropic Hofer-Zehnder capacities of bounded convex domains with special coisotropic
submanifolds and the leaf relation (introduced by Lisi and Rieser recently), study their estimates and relations with the Hofer-Zehnder capacity,
give some interesting corollaries, and also obtain corresponding versions of a Brunn-Minkowski type inequality by Artstein-Avidan and Ostrover
 and a theorem by Evgeni Neduv.
} \vspace{-0.1in}


\medskip\vspace{12mm}
\tableofcontents

\section{Introduction}

\subsection{Coisotropic capacity}\label{sec:coCap}
We begin with brief review of the notion of a  \textsf{coisotropic capacity} (i.e., a symplectic capacity relative to a coisotropic
submanifold of a symplectic manifold) introduced by   Lisi and  Rieser \cite{LiRi13} recently.
Consider a tuple $(M,N,\omega, \sim)$ consisting of a symplectic manifold
$(M,\omega)$, a coisotropic submanifold $N\hookrightarrow M$,  and
a coisotropic equivalence relation $\sim$.
(\textsf{In this paper all manifolds are assumed to be connected without special statements}!) Recall that the symplectic complementary distribution $TN^\omega\subset TN$ is integrable on $N$ and so $N$ is foliated by leaves of this foliation(\cite{McSal98}).
 By \cite[Definition 1.4]{LiRi13}, a  \textsf{coisotropic equivalence relation} on $N$ is an equivalence relation $\sim$ with the
property that if $x$ and $y$ are on the same leaf then $x\sim y$. A special example
is the so-called  \textsf{leaf relation} $\sim$: \;
 $x\sim y$ if and only if $x$ and $y$ are on the same leaf.
We do not assume $\dim N<\dim M$. If $\dim N=\dim M$ then the leaf relation on $N$ means that
that $x,y\in N$ satisfies $x\sim y$ if and only if $x=y$.

A class of important examples of $(M,N,\omega, \sim)$ consist of
   the standard symplectic space $(\mathbb{R}^{2n},  \omega_0)$,
 its coisotropic linear subspaces
\begin{equation}\label{e:stand-cois}
\mathbb{R}^{n,k}=\{x\in\mathbb{R}^{2n}\,|\,x=(q_1,\cdots,q_n,p_1,\cdots,p_k,0,\cdots,0)\}
\end{equation}
for $k=0,\cdots,n$, and the leaf relation $\sim$ on $\mathbb{R}^{n,k}${\color{red}.}
(Here  we understand
$\mathbb{R}^{n,0}=\{x\in\mathbb{R}^{2n}\,|\,x=(q_1,\cdots,q_n,0,\cdots,0)\}$.)
Denote by
\begin{eqnarray}\label{e:V0}
&&\hspace{-6mm}V_0^{n,k}=\{x\in\mathbb{R}^{2n}\;|\;x=(0,\cdots,0,q_{k+1},\cdots,q_n,0,\cdots,0)\},\\
&&\hspace{-6mm}V^{n,k}_1=\{x\in\mathbb{R}^{2n}\,|\,x=(q_1,\cdots,q_k,0,\cdots,0,p_1,\cdots,p_k,0,\cdots,0)\}.\label{e:V1}
\end{eqnarray}
When $k=0$, $L_0^n:=V_0^{n,0}=\{x\in\mathbb{R}^{2n}\;|\;x=(q_{1},\cdots,q_n,0,\cdots,0)\}=\mathbb{R}^{n,0}$ is a Lagrangian subspace, and $V^{n,k}_1=\{0\}$. Moreover, $V^{n,n}_0=\{0\}$ and $V^{n,n}_1=\mathbb{R}^{2n}$.
Clearly, two points $x, y\in\mathbb{R}^{n,k}$  satisfy
  $x\sim y$ if and only if $y\in x+V_0^{n,k}$.
Let  $J_{2n}$ denote the standard complex structure  on $\mathbb{R}^{2n}$ given by
\begin{equation}\label{e:standComplex}
 (q_1,\cdots,q_n, p_1,\cdots, p_n)\mapsto (-p_1,\cdots,-p_n, q_1,\cdots, q_n).
\end{equation}
Then $\mathbb{R}^{2n}$ has the orthogonal decompositions
\begin{equation}\label{lagmultipier1}
\mathbb{R}^{2n}=J_{2n}V^{n,k}_0\oplus \mathbb{R}^{n,k}=J_{2n}\mathbb{R}^{n,k}\oplus V^{n,k}_0
\end{equation}
 with respect to the standard inner product.

Let $(M_0, N_0, \omega_0, \sim_0)$ and $(M_1, N_1, \omega_1, \sim_1)$ be
two tuples as above.  By \cite[Definition 1.5]{LiRi13}
a \textsf{relative symplectic embedding} from
$(M_0, N_0, \omega_0)$ and $(M_1, N_1, \omega_1)$
is a symplectic embedding $\psi: (M_0, \omega_0)\to (M_1, \omega_1)$
satisfying $\psi^{-1}(N_1) = N_0$. We say such an embedding
 $\psi$ to \textsf{respect the pair of coisotropic equivalence relations}
    $(\sim_0,\sim_1)$
        if for every $x, y \in N_0$,
    $$
    \psi(x) \sim_1 \psi(y)\quad \Longrightarrow \quad x \sim_0 y.
    $$

For $a\in\mathbb{R}$ we write ${\bf a}:=(0,\cdots,0,a)\in\mathbb{R}^{2n}$. Denote by
\begin{equation}\label{e:Ball1}
B^{2n}({\bf a}, r)\quad\hbox{and}\quad B^{2n}(r)
\end{equation}
the open balls of radius $r$ centered at ${\bf a}$ and the origin in $\R^{2n}$ respectively, and by
\begin{eqnarray}\label{e:Ball2}
&&\hspace{-10mm}W^{2n}(R) := \left \{ (x_1, \dots, x_n, y_1, \dots, y_n) \in \R^{2n} \, | \, x_n^2 + y_n^2  < R^2\,\text{or}\, y_n < 0
\right\},\\
&&\hspace{-10mm}  W^{n,k}(R):=  W^{2n}(R) \cap \R^{n,k}\quad\hbox{and}\quad B^{n,k}(r):= B^{2n}(r) \cap \R^{n,k}.\label{e:Ball3}
\end{eqnarray}
($W^{2n}(R)$ was written as $W(R)$ in \cite[Definition~1.1]{LiRi13}).

\begin{definition}[\hbox{\cite[Definition~1.7]{LiRi13}}]\label{def:coCap}
{\rm
A \textsf{coisotropic capacity} is a map which associates a tuple $(M,N,\omega, \sim)$ as above
to a non-negative number or infinity $c(M,N,\omega, \sim)$ with the following properties.
  \begin{enumerate}
         \item [(i)] {\bf Monotonicity}. If there exists a relative
symplectic embedding $\psi$ from $(M_0, N_0, \omega_0, \sim_0)$ to $(M_1, N_1, \omega_1, \sim_1)$
respecting the coisotropic equivalence relations, and $\dim M_0=\dim M_1$, then
$$c(M_0,N_0,\omega_0, \sim_0) \leq c(M_1,N_1,\omega_1, \sim_1).$$

    \item [(ii)]{\bf Conformality}.  
$c(M,N,\alpha\omega, \sim)=|\alpha|c(M,N,\omega, \sim),\;\forall\alpha \in \mathbb{R}\backslash \{0\}$.

    \item [(iii)]{\bf Non-triviality}. With the leaf
    relation $\sim$  it holds that for $k=0,\cdots,n-1$,
    \begin{eqnarray}\label{e:Ball4}
    c(B^{2n}(1),B^{n,k}(1),\omega_0, \sim ) =\frac{\pi}{2}=
    c(W^{2n}(1),W^{n,k}(1),\omega_0, \sim ).
    \end{eqnarray}
    \end{enumerate}}
\end{definition}

As remarked in  \cite[Remark~1.9]{LiRi13}, it was because of
the non-triviality (iii) that $c$ cannot be any symplectic capacity.

\textsf{From now on, we abbreviate $c(M,N,\omega,\sim)$ as $c(M,N,\omega)$
if $\sim$ is the leaf relation on $N$.  
In particular, for domains $D\subset\mathbb{R}^{2n}$ we also abbreviate $c\left(D, D\cap\mathbb{R}^{n,k},\omega_0\right)$ as
   $c\left(D, D\cap\mathbb{R}^{n,k}\right)$ for simplicity.}

\begin{definition}\label{def:coWidth}
{\rm Let $N$ be a $n+k$-dimensional coisotropic submanifold  in a  symplectic manifold $(M, \omega)$ of dimension $2n$.
We define the \textsf{relative Gromov width} of it to be
$$
{\it w}_G(N;M,\omega):=\sup\left\{\pi r^2\,\Bigg|\,\begin{array}{ll}
&\exists\;\hbox{a relative symplectic embedding}\\
&\hbox{$(B^{2n}(r), B^{n,k}(r))\to (M,N)$ respecting}\\
&\hbox{ the leaf relations on $B^{n,k}(r)$ and $N$}
\end{array}\right\}.
$$
}
\end{definition}

\textsf{We always assume $k\in\{0,1\cdots,n-1\}$ without special statements.}
When $k=0$, $N$ is a Lagrangian submanifold and this relative Gromov width was first
introduced  (without the terminology) by  Barraud-Cornea \cite[\S 1.3.3]{BaCor07}
and formally defined by Biran-Cornea \cite[\S6.2]{BirCor08}.

It is easily seen that ${\it w}_G$ satisfies monotonicity, conformality and
\begin{equation}\label{e:Rwidth-c}
{\it w}_G(B^{2n}(r)\cap\mathbb{R}^{n,k}; B^{2n}(r),\omega_0) =\pi r^2,\quad\forall r>0.
\end{equation}
By the non-squeezing for relative embeddings \cite{LiRi13}, 
${\it w}_G(W^{n,k}(1);W^{2n}(1),\omega_0)\le \pi$.
Hence
$$
\pi={\it w}_G(B^{2n}(r)\cap\mathbb{R}^{n,k}; B^{2n}(r),\omega_0)\le{\it w}_G(W^{n,k}(1);W^{2n}(1),\omega_0)\le\pi
$$
and it follows that ${\it w}_G/2$ is a coisotropic capacity. Indeed, ${\it w}_G/2$ is the smallest coisotropic capacity in the sense that
\begin{equation}\label{e:width-c}
{\it w}_G(N;M,\omega)/2\le c(M, N,\omega).
\end{equation}


 Lisi and  Rieser \cite{LiRi13} gave an example of the coisotropic capacities
by constructing  an analogue of the Hofer-Zehnder capacity relative to a coisotropic submanifold,
called the \textsf{coisotropic Hofer-Zehnder
capacity}. Using it they also studied symplectic embeddings relative to coisotropic constraints and got some corresponding dynamical results.
This coisotropic capacity also played a key role in the proof of
Humili\'ere-Leclercq-Seyfaddini's  important rigidity result  that symplectic homeomorphisms preserve coisotropic submanifolds
 and their characteristic foliations (\cite{HuLeSe15}).

\begin{defn}[\hbox{\cite[Definition 1.10]{LiRi13}}]\label{def:simple}
{\rm Given a coisotropic submanifold $N$ of a symplectic manifold $(M,\omega)$,
a smooth Hamiltonian $H:M\to\mathbb{R}$ is called  \textsf{simple} if
\begin{enumerate}
  \item [(i)] there exists a compact set $K\subset M$ (depending on $H$) and a constant $m(H)$ such that $K\subset M\setminus\partial M$, $\emptyset\ne K\cap N\subset N$, and
      $H|_{M\setminus K}\equiv m(H)$;

  \item [(ii)] there exists an open set $U\subset M$ (depending on $H$) intersecting with $N$ such that
      $H|_U\equiv 0$;

  \item [(iii)] $0\le H(x)\le m(H)$ for all $x\in M$.
  \end{enumerate}
Denote the set of simple Hamiltonians by $\mathcal{H}(M,N)$. ({\it Note}: $\mathcal{H}(M,N)=\mathcal{H}(M)$ if $N=M$.)}
\end{defn}

Let $\sim$ be a coisotropic equivalence relation on $N$ and $X_H$ the Hamiltonian
vector field of  $H\in \mathcal{H}(M,N)$.
Then any solution $\gamma(t)$ of $\dot{\gamma}(t)=X_H(\gamma(t))$ with
$\gamma(0)\in N$ is well-defined on $\mathbb{R}$ (since $X_H$ has compact support).
In \cite[Definition 1.11]{LiRi13} the \textsf{return time} of $\gamma$ relative to $N$ and $\sim$ was defined
by
$$
T_\gamma=\inf\{t>0\;|\; \gamma(t)\in N\;\hbox{and}\;\gamma(0)\sim\gamma(t)\},
$$
where the infimum of the empty set is understood as $+\infty$.
Clearly, for the trivial equivalence relation $\sim$, $T_\gamma$ is a return time
to the submanifold $N$ itself. If $\dim N=\dim M$, and the
coisotropic equivalence relation on $N$ is the leaf relation $\sim$, then
 $T_\gamma$ is the minimal period of $\gamma$, hereafter we understand $T_\gamma=+\infty$
if $\gamma$ is nonconstant and has no finite period.
For the leaf relation $\sim$, $T_\gamma$ measures the
shortest non-trivial leafwise chord. Here a \textsf{leafwise chord} for $N$ in $(M,\omega)$ is
 a Hamiltonian trajectory that starts and ends on  the same leaf of the  coisotropic foliation (\cite{LiRi13}).
By \cite[Definition 1.12]{LiRi13},
a function $H\in \mathcal{H}(M, N)$ is called
\textsf{admissible} for the coisotropic equivalence relation $\sim$,
if any solution of
$$
\dot{\gamma}=X_{H}(\gamma)\quad\hbox{with}\quad
\gamma(0)\in N
$$
 is {either} constant {or} such that $T_\gamma > 1$, i.e.,  the return time of the orbit $\gamma$ relative to $(N,\sim)$
  is greater than $1$.  Denote the collection of all such admissible functions by $\mathcal{H}_{\rm ad}(M,N,\omega, \sim)$.
It should be noted that for $N=M$ the above admissible condition  is equivalent to the condition  that
  any solution of $\dot{\gamma}=X_{H}(\gamma)$  is either constant {or} has minimal period $T_\gamma > 1$,
  namely, it becomes the admissible condition in the definition of the Hofer-Zehnder capacity \cite{HoZe90}.
In view of \cite[Definition 1.13]{LiRi13} we call
$$
c_{\rm LR}(M,N,\omega,\sim)=\sup\{m(H)\;|\;H\in\mathcal{H}_{\rm ad}(M,N,\omega,\sim)\}
$$
the \textsf{coisotropic Hofer-Zehnder capacity} of the tuple $(M,N,\omega,\sim)$.
It is a coiso-tropic  capacity (\cite[Theorem~1.14]{LiRi13}).
  Without special statements, for domains $D\subset\mathbb{R}^{2n}$ we abbreviate $c_{\rm LR}\left(D, D\cap\mathbb{R}^{n,k},\omega_0\right)$ as $c_{\rm LR}\left(D, D\cap\mathbb{R}^{n,k}\right)$ for simplicity.

\begin{remark}\label{rem:cocap}
{\rm It is not hard to check that $c_{\rm LR}(M, M,\omega)$ is equal to the Hofer-Zehnder capacity $c_{\rm HZ}(M, \omega)$ in \cite{HoZe90} since $\mathcal{H}(M,N)=\mathcal{H}(M)$ and $x\sim y$ if and only if $x=y$ for leaf relation $\sim$ on $M$.
If $\dim N=\dim M$, $N\ne M$ and $c_{\rm LR}(M, N,\omega)<+\infty$, then the Hamiltonian vector field $X_H$
of any $H\in\mathcal{H}(M,N)$ with $\max H>c_{\rm LR}(M, N,\omega,\sim)$ has a nonconstant periodic trajectory through $N$;
thus $c_{\rm LR}(M, N,\omega,\sim)$ is completely different from  relative Hofer-Zehnder capacities  introduced in \cite{GiGu04} and \cite{Lu06IJM}.
If $\dim N<\dim M$, $\sim$ is the leaf relation on $N$ and $c_{\rm LR}(M,N,\omega,\sim)<\infty$,
for a compact hypersurface $S \hookrightarrow M$  bounding a compact symplectic manifold,
  and a parametrized family  $\{S_\epsilon\,|\,\epsilon \in I\}$ of hypersurfaces modelled on it and
   transverse to $N$, \cite[Theorem~4.6]{LiRi13} showed that there exists a  leafwise chord for $N$ on $S_\epsilon$
   for almost each $\epsilon \in I$.
 }
\end{remark}

\subsection{Background and main results}\label{sec:main}

Symplectic capacities of convex bodies (i.e., compact convex subsets containing interior points)
in $(\mathbb{R}^{2n},\omega_0)$ play important roles in studies
of symplectic topology and other subjects such as billiard dynamics and convex geometry (cf. \cite{AAO08, AAO14, AAKO14}).
For example, in \cite{AAKO14}  Artstein-Avidan,  Karasev and Ostrover proved that
Viterbo's symplectic { isoperimetric} conjecture for symplectic capacities of convex domains  (\cite{Vit00}) implies the famous Mahler conjecture regarding the volume product of symmetric convex bodies in convex geometry.
Their proof is based on the
relation between symplectic capacities and the length of shortest billiard trajectories (\cite{AAO14}).
In \cite{AAO08} Artstein-Avidan and Ostrover proved a Brunn-Minkowski type inequality for the Ekeland-Hofer-Zehnder capacity of convex domains based on the representation formula for the Ekeland-Hofer-Zehnder capacity of convex domains (\cite{EH89, EH90, Sik90} and {\cite[Proposition~4]{HoZe90}}.
We  generalized these results to the symmetric Ekeland-Hofer-Zehnder  capacity
 and the generalized Ekeland-Hofer-Zehnder one of convex domains (\cite{JinLu1915, JinLu1916}).
 Generally speaking, it is more difficult to compute the coisotropic Hofer-Zehnder capacity than
 to compute the  Hofer-Zehnder capacity. In this paper, for the coisotropic Hofer-Zehnder capacity of convex domains
  we shall prove that  there exists a corresponding representation formula (Theorem~\ref{th:represention})
  and a corresponding Brunn-Minkowski type inequality  (Theorem~\ref{th:Brun}).
The basic proof ideas are following \cite{HoZe90, AAO08} and \cite{JinLu1915, JinLu1916}. The main difficulty realizing
the goal is looking for  suitable Banach spaces on which variational methods are carried out.\\

\noindent{1.2.1}. {\bf Representation formulas for coisotropic Hofer-Zehnder capacities of convex domains}.\;
Let $\mathcal{S}$ be  the boundary  of a bounded convex domain $D$ in $(\mathbb{R}^{2n},\omega_0)$.
   We defined in \cite{JinLu1916}  a nonconstant  absolutely continuous curve $z:[0,T]\to \mathbb{R}^{2n}$ (for some $T>0$)
  to be a  \textsf{generalized characteristic} on $\mathcal{S}$  if $z([0,T])\subset \mathcal{S}$ and
    $\dot{z}(t)\in JN_\mathcal{S}(z(t))\;\hbox{a.e.}$, where
    $$
    N_\mathcal{S}(x)=\{y\in\mathbb{R}^{2n}\,|\, \langle u-x, y\rangle\le 0\;\forall u\in D\}
    $$
    is the normal cone to $D$ at $x\in\mathcal{S}$.
\textsf{Fix an integer $0\le k<n$.} When $D\cap \mathbb{R}^{n,k}\ne\emptyset$, we call a generalized characteristic $z:[0,T]\to\mathcal{S}$
a \textsf{generalized leafwise chord} (abbreviated GLC) on $\mathcal{S}$ for $\mathbb{R}^{n,k}$ if
   $z(0), z(T)\in \mathbb{R}^{n,k}$ and $z(0)-z(T)\in V_0^{n,k}$.
(Such a chord becomes a leafwise chord on $\mathcal{S}$ for $\mathbb{R}^{n,k}$ if $\mathcal{S}$ is of class $C^1$.)
We define the action of a GLC $z:[0,T]\to\mathcal{S}$ by
$$
A(z)=\frac{1}{2}\int_0^T\langle -J_{2n}\dot{z},z\rangle dt.
$$

\begin{theorem}\label{th:represention}
  Let $D\subset\mathbb{R}^{2n}$ be a bounded convex  domain such that $D\cap \mathbb{R}^{n,k}\ne\emptyset$. Then there exists a generalized leafwise chord
  $x^\ast$ on $\partial D$ for $\mathbb{R}^{n,k}$ such that
  \begin{eqnarray}
  &&c_{\rm LR}(D,D\cap \mathbb{R}^{n,k})=A(x^\ast)\label{capacity}\\
  &&\qquad=\min\{A(x)>0 \,|\,x\;\hbox{is a GLC on\;$\partial D$\;for \;$\mathbb{R}^{n,k}$}\}\label{minaction}.
  \end{eqnarray}
   \end{theorem}

This result is a core of this paper. Seemingly, the proof of (\ref{minaction}) is following
\cite[\S1.5]{HoZe94} within the analytic framework in  \cite{LiRi13}.
 However, our main novelty is to find a substitute of the space $\mathcal{F}$ on page 26 of \cite{HoZe94},
  the  Banach subspace $\mathscr{F}_p^{n,k}$ of $W^{1,p}([0,1],\mathbb{R}^{2n})$
 given in (\ref{e:BanachSpace}), which does not appear in any references. 
 It is by no means natural to choose such a space. After carefully checking the proof of periodic case (for example
 \cite{Sik90}), trying and testing  many times we gradually realized that this space was suitable.
  Another key for the proof of (\ref{capacity}) is Lemma~\ref{nointerior}, which is a corresponding
  result about set of actions of leafwise chords on $\partial D$ for $\mathbb{R}^{n,k}$
of a proposition in \cite[\S~7.4]{Sik90}.

 \begin{theorem}\label{th:represention*}
  Let $D\subset\mathbb{R}^{2n}$ be a bounded convex  domain such that  $D\cap \mathbb{R}^{n,k}\ne\emptyset$.
  If $0\le k<n-1$ then
  \begin{eqnarray}\label{e:add1}
  c_{\rm LR}(D,D\cap \mathbb{R}^{n,k})\le c_{\rm LR}(D,D\cap \mathbb{R}^{n,k+1}).
  \end{eqnarray}
 If $k=n-1$ then
 \begin{eqnarray}\label{e:add2}
  c_{\rm LR}(D,D\cap \mathbb{R}^{n,n-1})\le c_{\rm HZ}(D).
  \end{eqnarray}
  \end{theorem}

These two theorems are proved in \S~\ref{sec:3} and \S~\ref{sec:4}, respectively.

\begin{definition}
{\rm For a bounded convex domain $D\subset \mathbb{R}^{2n}$ such that $D\cap \mathbb{R}^{n,k}\ne\emptyset$,
a generalized leafwise chord on $\partial D$ for $\mathbb{R}^{n,k}$ with action $c_{\rm LR}(D, D\cap\mathbb{R}^{n,k})$
is called a $c_{\rm LR}(D, D\cap\mathbb{R}^{n,k})$ carrier.}
\end{definition}

\begin{remark}\label{rem:CLRcarrier}
{\rm Suppose that $D$ is a convex domain in $\mathbb{R}^{2n}$ satisfying $0\in D$. \textsf{ We always denote by $j_D$ the Minkowski functional associated to $D$ defined by}
\begin{equation}\label{e:Mink}
j_D(x):=
\left\{
\begin{array}{ll}
\inf\{\lambda>0\,|\,\lambda^{-1}x\in D\},& \hbox{for}\quad x\ne 0,\\
0,& \hbox{for}\quad x=0.
\end{array}
\right.
\end{equation}
Let $H_D:=j^2_D$.
{Then a standard computation (at the beginning of Step~4 in the proof of Proposition~\ref{prop:Brun.0})  leads to the fact that
 a generalized characteristic  on $\partial D$ is a $c_{\rm LR}(D, D\cap\mathbb{R}^{n,k})$ carrier if and only if
it may be reparametrized as a solution $x:[0,\mu]\to \partial D$ of
\begin{equation}\label{e:BDY}
-J_{2n}\dot{x}(t)\in  \partial H_D(x(t)),\quad x(0), x(\mu)\in \mathbb{R}^{n,k},\quad x(0)-x(\mu)\in V_0^{n,k}
\end{equation}
where $\mu=c_{\rm LR}(D,D\cap \mathbb{R}^{n,k})=A(x)$ is the return time of $x$ for $\mathbb{R}^{n,k}$ and the leaf relation $\sim$.
Moreover, }since $\{\partial H(x)|x\in\partial D\}$ is a bounded set in $\mathbb{R}^{2n}$ (see \cite[(4.3)]{JinLu1916}), by Arzela-Ascoli theorem we deduce that
all $c_{\rm LR}(D, D\cap\mathbb{R}^{n,k})$ carriers  form a compact subset in $C^0([0,\mu],\partial D)$
(and $C^1([0,\mu],\partial D)$ if $\partial D$ is $C^1$), where $\mu=c_{\rm LR}(D,D\cap \mathbb{R}^{n,k})$.}
\end{remark}

\begin{remark}\label{rem:R10case}
{\rm If $n=1$, $k=0$ and $D\subset\mathbb{R}^2$ is a bounded convex domain such that
$D\cap\mathbb{R}^{1,0}\ne\emptyset$, for a non-periodic leafwise chord $x:[0,T]\to \partial D$,  the line segment $D\cap\mathbb{R}^{1,0}$
 and $x$ form a loop $\gamma$. Note that  $\langle -J_{2n}\dot{z},z\rangle$ vanishes
  along line segment $D\cap\mathbb{R}^{1,0}$. Then
  $$
  A(x)=\int_x qdp=\int_\gamma qdp
  $$
 is equal to the symplectic area of the bounded domain surrounded by $\gamma$ in view of Stokes theorem (by a smooth approximation if necessary). Hence by Theorem ~\ref{th:represention}~ $c_{\rm LR}(D, D\cap\mathbb{R}^{1,0})$ is equal to the smaller symplectic area of
$D$ above and below the the line segment $D\cap\mathbb{R}^{1,0}$.
 In particular, we obtain
 \begin{eqnarray}\label{e:BDY.1}
 &&c_{\rm LR}(B^2(r), B^2(r)\cap\mathbb{R}^{1,0})=\pi r^2/2,\\
  &&c_{\rm LR}((a, b)\times (-c, d), (a, b)\times (-c, d)\cap\mathbb{R}^{1,0})=(b-a)\min\{c,d\}\label{e:BDY.2}
  \end{eqnarray}
for any positive numbers $r, c, d$ and real numbers $a<b$. (\ref{e:BDY.1}) and (\ref{e:BDY.2}) are basic for computation of coisotropic
 Hofer-Zehnder capacities of polydiscs and cuboids
in higher dimension. }
\end{remark}

In the following we give more corollaries of Theorem~\ref{th:represention}; in particular, we can derive (\ref{e:Ball4}).

\begin{corollary}\label{cor:ellipsoid}
For numbers $r_j>0$, $j=1,\cdots,n$, define an ellipsoid
\begin{equation}\label{e:ellipsoid-}
E(r_1,\cdots,r_n):=\left\{(q_1,\cdots,q_n,p_1,\cdots,p_n)\in\mathbb{R}^{2n}\,\bigg|\,
\sum^n_{j=1}(q^2_j+p^2_j)/r_j^2<1\right\}.
\end{equation}
Then there holds
  \begin{equation}\label{e:ellipsoid}
  c_{\rm LR}\left(E(r_1,\cdots,r_n), E(r_1,\cdots,r_n)\cap\mathbb{R}^{n,k}\right) =
  \frac{\pi}{2}\min\{2\min_{i\le k}r_i^2, \min_{i>k}r^2_i\}.
    \end{equation}
\end{corollary}

Clearly, (\ref{e:ellipsoid}) implies the first equality in (\ref{e:Ball4}).

\begin{corollary}[\hbox{\cite[Proposition~2.7, Corollary~1.15]{LiRi13}}]\label{cor:ellipsoid+}
Let $B^{2n}({\bf a}, 1)$ and $B^{n,k}(r)$ be as in (\ref{e:Ball1}) and (\ref{e:Ball3}), respectively, where
$a\in (-1,0]\cup[0,1)$ and $r=\sqrt{1-a^2}$.      Then
\begin{equation}\label{e:1ellipsoid+}
c_{\rm LR}\left(B^{2n}({\bf a}, 1), B^{n,k}(r)\right) =\arcsin(r)-r\sqrt{1-r^2}.
\end{equation}
 This and the conformality (due to $B^{2n}({\bf a}, R)=RB^{2n}({\bf a}/R, 1)$) imply
\begin{equation}\label{e:ellipsoid++}
c_{\rm LR}\left(B^{2n}({\bf a}, R), B^{2n}({\bf a}, R)\cap\mathbb{R}^{n,k}\right) =\left(
\arcsin(r_R)-r_R\sqrt{1-r_R^2}\right)R^2
\end{equation}
for any $R>0$ and $|a|<R$, where $r_R=\sqrt{1-a^2/R^2}$.
\end{corollary}

\cite[Proposition~2.7]{LiRi13} showed
$$
c_{\rm LR}\left(B^{2n}({\bf a}, 1), B^{n,k}(r)\right) \ge\arcsin(r)-r\sqrt{1-r^2}.
$$
 The converse inequality was contained in the proof of \cite[Corollary~1.15]{LiRi13}.

Since $B^{n,k}(r)=B^{2n}({\bf a}, 1)\cap\mathbb{R}^{n,k}$, taking $a=0$ we recover
the first equality in (\ref{e:Ball4}) again.

Define $U^{n,k}(1)=U^{2n}(1)\cap\mathbb{R}^{n,k}$, where
$$
U^{2n}(1)=\mathbb{R}^{2n-2}\times\{(x,y)\in\mathbb{R}^2\,|\,x^2+y^2<1\;\hbox{or
$-1<x<1$ and $y<0$}\}.
$$
 Let $W^{2n}(R)$ and  $W^{n,k}(R)$ be as in (\ref{e:Ball2}) and (\ref{e:Ball3}), respectively.
It was proved in \cite[Section~3]{LiRi13} that
\begin{equation}\label{e:ellipsoid+++}
c_{\rm LR}\left(W^{2n}(1), W^{n,k}(1)\right) =c_{\rm LR}\left(U^{2n}(1), U^{n,k}(1)\right).
\end{equation}

From Theorem~\ref{th:represention} we can also derive

\begin{corollary}[\hbox{\cite[Propositions~2.7,3.1]{LiRi13}}]\label{cor:ellipsoid++}
\begin{equation}\label{e:ellipsoid4+}
c_{\rm LR}\left(U^{2n}(1), U^{n,k}(1)\right) =\frac{\pi}{2}.
\end{equation}
\end{corollary}

\begin{corollary}\label{cor:polyDisk}
For numbers $r_j>0$, $j=1,\cdots,n$, consider polydiscs
$$
P^{2n}(r_1,\cdots,r_n):=B^2(r_1)\times\cdots\times B^2(r_n)\subset (\mathbb{R}^{2})^n\equiv\mathbb{R}^{2n}
$$
(and so $\mathbb{R}^{n,k}$ and $V^{n,k}_0$ are identified with
$$
(\mathbb{R}^2)^k\times(\mathbb{R}\times\{0\})^{n-k}\quad\hbox{and}\quad(\{0\}\times\{0\})^k\times(\mathbb{R}\times\{0\})^{n-k},
$$
respectively.) Then
\begin{eqnarray}\label{e:polyD.1}
&&c_{\rm LR}\left(P^{2n}(r_1,\cdots,r_n), P^{2n}(r_1,\cdots,r_n)\cap\mathbb{R}^{n,k}\right)\le\frac{\pi}{2}\min_{i>k}r^2_i,\\
&&c_{\rm LR}\left(P^{2n}(r_1,\cdots,r_n), P^{2n}(r_1,\cdots,r_n)\cap\mathbb{R}^{n,k}\right)
\ge\frac{\pi}{2}\min\{2\min_{i\le k}r_i^2, \min_{i>k}r^2_i\},
\label{e:polyD.2}
\end{eqnarray}
and hence
\begin{eqnarray}\label{e:polyD.3}
c_{\rm LR}\left(P^{2n}(r_1,\cdots,r_n), P^{2n}(r_1,\cdots,r_n)\cap\mathbb{R}^{n,k}\right)=\frac{\pi}{2}\min_{i>k}r^2_i
\end{eqnarray}
if $\sqrt{2}\min\{r_1,\cdots,r_k\}\ge\min\{r_{k+1},\cdots, r_n\}$.
\end{corollary}

(\ref{e:polyD.3}) may be viewed as a corresponding version of a result for the Hofer-Zehnder capacity
$$
c_{\rm HZ}(P^{2n}(r_1,\cdots,r_n))=\pi\min\{r^2_1,\cdots, r^2_n\}.
$$
A special case of equation (\ref{e:polyD.3}) is that for any $a>0$,
\begin{eqnarray}\label{e:polyD.4}
&&c_{\rm LR}\left((B^2(a))^k\times (B^2(\sqrt{2}a))^{n-k},  ((B^2(a))^k\times (B^2(\sqrt{2}a))^{n-k})\cap\mathbb{R}^{n,k}\right)\nonumber\\
&&=\pi a^2.
\end{eqnarray}

With results in \cite{JinLu1918},
we can prove that the inequality in (\ref{e:polyD.2}) is actually an equality,
see Remark~\ref{rem:EH-product}(ii) below.

\begin{remark}\label{rem:EH-product}
{\rm
\begin{enumerate}
\item[(i)] For suitable subsets in $\mathbb{R}^{2n}$, we constructed another
coiso-tropic capacity in \cite{JinLu1918}, the coisotropic Ekeland-Hofer capacity.
For each bounded convex  domain $D\subset \mathbb{R}^{2n}$ such that  $D\cap\mathbb{R}^{n,k}\ne\emptyset$,
we proved that its  coisotropic Ekeland-Hofer capacity $c^{n,k}(D)$ is equal to
the right side of (\ref{minaction}). Thus it follows from the definition of $c^{n,k}$ and the inner regularity of $c_{\rm LR}$  that
\begin{eqnarray}\label{e:fixpt.3}
c^{n,k}(D)=c_{\rm LR}(D,D\cap \mathbb{R}^{n,k})
\end{eqnarray}
for any  convex  domain $D\subset\mathbb{R}^{2n}$ such that $D\cap \mathbb{R}^{n,k}\ne\emptyset$.
For the coisotropic Ekeland-Hofer capacity we proved a product formula in \cite{JinLu1918}. It and equation (\ref{e:fixpt.3}) lead to:\\
{\it Claim}. For   convex domains $D_i\subset\mathbb{R}^{2n_i}$ containing the origins, $i=1,\cdots,m\ge 2$,
 and integers $0\le l_0\le n:=n_1+\cdots+n_m$, $l_j=\max\{l_{j-1}-n_j,0\}$, $j=1,\cdots, m-1$,
    it holds that
 \begin{eqnarray}\label{e:product1}
&&c_{\rm LR}(D_1\times\cdots\times D_m,  (D_1\times\cdots\times D_m)\cap \mathbb{R}^{n,l_0} )\nonumber\\
&&=\min_ic_{\rm LR}(D_i, D_i\cap\mathbb{R}^{n_i,\min\{n_i,l_{i-1}\}}).
 \end{eqnarray}
Hereafter  $\mathbb{R}^{2n_1}\times \mathbb{R}^{2n_2}\times\cdots\times \mathbb{R}^{2n_m}$ is identified with $\mathbb{R}^{2(n_1+\cdots+n_m)}$
via
$$
((q^{(1)},p^{(1)}),\cdots, (q^{(m)},p^{(m)}))\mapsto (q^{(1)},\cdots, q^{(m)},p^{(1)},\cdots, p^{(m)}).
$$
By Remark~\ref{rem:cocap} we understand $c_{\rm LR}(D_i, D_i\cap\mathbb{R}^{n_i,n_i})$ as $c_{\rm HZ}(D_i)$.

\item[(ii)] Applying (\ref{e:product1}) to $P^{2n}(r_1,\cdots,r_n)$ we get an improvement of
Corollary~\ref{cor:polyDisk} as follows:
 \begin{eqnarray}\label{e:product3}
 &&c_{\rm LR}(P^{2n}(r_1,\cdots,r_n), P^{2n}(r_1,\cdots,r_n)\cap\mathbb{R}^{n,k})\nonumber\\
 &=&\min_ic_{\rm LR}(B^2(r_i), B^2(r_i)\cap\mathbb{R}^{1,\min\{1,l_{i-1}\}})\\
 &=&\frac{\pi}{2}\min\{2\min_{i\le k}r_i^2, \min_{i>k}r^2_i\},\label{e:product4}
 \end{eqnarray}
 where $0\le l_0=k< n$, $l_j=\max\{l_{j-1}-1,0\}$, $j=1,\cdots, n-1$. (Here we use (\ref{e:BDY.1}) and the fact that
$c_{\rm LR}(B^2(r_i), B^2(r_i)\cap\mathbb{R}^{1,1})=c_{\rm HZ}(B^2(r_i))=\pi r_i^2$.)

\item[(iii)] For real numbers $a_i<b_i$ and positive numbers $c_i, d_i$, $i=1,\cdots,n$, since
\begin{eqnarray}\label{e:product5}
&& c_{\rm LR}((a_i, b_i)\times (-c_i, d_i), (a_i, b_i)\times (-c_i, d_i)\cap\mathbb{R}^{1,1})\\
&& =c_{\rm HZ}((a_i, b_i)\times (-c_i, d_i))  =(b_i-a_i)(c_i+d_i),
  \end{eqnarray}
as above we may derive from equations (\ref{e:BDY.2}) and (\ref{e:product1}) that for each integer $0\le k<n$,
\begin{eqnarray}\label{e:product7}
&&c_{\rm LR}(\prod^n_{i=1}(a_i, b_i)\times (-c_i, d_i), \prod^n_{i=1}(a_i, b_i)\times (-c_i, d_i)\cap \mathbb{R}^{n,k} )\nonumber\\
&&=\min\left\{\min_{i\le k}(b_i-a_i)(c_i+d_i), \min_{i>k}((b_i-a_i)\min\{c_i,d_i\})\right\}.
 \end{eqnarray}
 \end{enumerate}
}
\end{remark}

  Suppose that a bounded convex  domain $D\subset\mathbb{R}^{2n}$ is centrally symmetric, i.e., $D=-D$.
  Let $\overline{D}^\circ=\{y\in\mathbb{R}^{2n}\;|\;\langle x, y\rangle\le 1\;\forall x\in\overline{D}\}$ be the polar of
the closure $\overline{D}$ of $D$, and let $\|\cdot\|_{\overline{D}^\circ}$ and $\|\cdot\|_{\overline{D}}$ be
 the norms given by the Minkowski functionals  $j_{\overline{D}}$ and $j_{\overline{D}^\circ}$ associated to $\overline{D}$ and $\overline{D}^\circ$
 as in (\ref{e:Mink}), respectively.
  Gluskin and  Ostrover \cite{GlOs} gave the following estimate for
the cylindrical capacity $c^Z(D)$,
\begin{eqnarray}\label{e:J-estimate1}
c_{\rm EHZ}(D)\le c^Z(D)\le \frac{4}{\|J_{2n}\|_{\overline{D}^\circ\to\overline{D}}},
\end{eqnarray}
where
\begin{eqnarray}\label{e:J-estimate2}
\|J_{2n}\|_{\overline{D}^\circ\to\overline{D}}:&=&\sup\{\langle J_{2n}v,u\rangle\;|\;v,u\in\overline{D}^\circ\}\nonumber\\
&=&\sup\{\|J_{2n}v\|_{\overline{D}}\;|\;\|v\|_{\overline{D}^\circ}\le 1\}
\end{eqnarray}
is the operator norm of $J_{2n}$ as a linear map between the normed
spaces $(\mathbb{R}^{2n}, \|\cdot\|_{\overline{D}^\circ})$ and $(\mathbb{R}^{2n}, \|\cdot\|_{\overline{D}})$.
 Correspondingly,  we have the following the estimate for  $c_{\rm LR}(D,D\cap \mathbb{R}^{n,k})$, whose proof will be given
in Section~\ref{sec:add}.

\begin{theorem}\label{th:J-estimate}
  Let $D\subset\mathbb{R}^{2n}$ be a centrally symmetric bounded convex domain
   such that $D\cap \mathbb{R}^{n,k}\ne\emptyset$ for some integer $0\le k<n$. Then
  \begin{eqnarray}\label{e:J-estimate3}
  c_{\rm LR}(D,D\cap \mathbb{R}^{n,k})\le \frac{2}{\|J_{2n}|_{J_{2n}V^{n,k}_0}\|_{\overline{D}^\circ\to\overline{D}}}.
  \end{eqnarray}
  Here $V_0^{n,k}$ is as in (\ref{e:V0}) and
   \begin{eqnarray}\label{e:J-estimate4}
  \|J_{2n}|_{J_{2n}V^{n,k}_0}\|_{\overline{D}^\circ\to\overline{D}}=\sup\{\langle J_{2n}v,u\rangle\;|\;
  u\in \overline{D}^\circ,\;
  v\in (J_{2n}V^{n,k}_0)\cap\overline{D}^\circ\}
  \end{eqnarray}
 is the operator norm of $J_{2n}|_{J_{2n}V^{n,k}_0}$ as a linear map between the normed
spaces $(J_{2n}V^{n,k}_0, \|\cdot\|_{\overline{D}^\circ})$ and $(\mathbb{R}^{2n}, \|\cdot\|_{\overline{D}})$.
   \end{theorem}

If the domain $D$ in Theorem~\ref{th:J-estimate} is not necessarily
 centrally symmetric, then we can apply Theorem~\ref{th:J-estimate} to $\mathbb{D}:=D-D$.
 Since $D\subset\mathbb{D}$ we obtain
 \begin{eqnarray}\label{e:J-estimate5}
 c_{\rm LR}(D,D\cap \mathbb{R}^{n,k})\le  c_{\rm LR}(\mathbb{D}, \mathbb{D}\cap \mathbb{R}^{n,k})
 \le \frac{2}{\|J_{2n}|_{J_{2n}V^{n,k}_0}\|_{\overline{\mathbb{D}}^\circ\to\overline{\mathbb{D}}}}.
 \end{eqnarray}
Clearly, $\|J_{2n}|_{J_{2n}V^{n,k}_0}\|_{\overline{\mathbb{D}}^\circ\to\overline{\mathbb{D}}}$
 is much easier to compute than the right side of (\ref{minaction}).

For a real symplectic manifold $(M,\omega,\tau)$ with nonempty real part $N:={\rm Fix}(\tau)$,
$N$ is a Lagrangian submanifold of $(M,\omega)$.
 (Indeed, for any $x\in N$ it follows from  $\tau^\ast\omega=-\omega$ that both
$T_xN={\rm Ker}(D\tau(x)-id_{T_xM})$ and ${\rm Ker}(D\tau(x)+id_{T_xM})$
are isotropic subspaces in $(T_xM,\omega_x)$. Moreover, there holds $T_xM=T_xN\oplus{\rm Ker}(D\tau(x)+id_{T_xM})$
   since $\tau^2=id_M$. Hence $T_xN$ and ${\rm Ker}(D\tau(x)+id_{T_xM})$
are Lagrange subspaces in $(T_xM,\omega_x)$.)
  Clearly, the leaf relation in $N$ is trivial, i.e, $x\sim y$ for all $x$, $y\in N$. The coisotropic Hofer-Zehnder capacity $c_{\rm LR}(M, N,\omega)$ is closely related to the symmetrical Hofer-Zehnder capacity of $(M,\omega,\tau)$ which is given by
$$
c_{\rm HZ,\tau}(M,\omega):=\sup\{m(H)\,|\,H\in\mathcal{H}_{ad}(M,\omega,\tau)\}
$$
(\cite{LiuW12}). Here $\mathcal{H}_{ad}(M,\omega,\tau)$ denotes the set of Hamiltonians $H:M\rightarrow \mathbb{R}$ such that
\begin{enumerate}
  \item [(i)]$H$ is $\tau$-invariant, i.e., $H(\tau x)=H(x)$, $\forall x\in M$;
  \item [(ii)]there exists a nonempty open subset $U=U(H)\subset M$ with $N\cap U\ne\emptyset$ such that $H|_U\equiv0$;
  \item [(iii)]there exists a compact subset
$K=K(H)\subset M\setminus\partial M$ such that $H|_{M\setminus K}\equiv m(H):=\max H$;
\item [(iv)]$0\leq H\leq m(H)$;
\item [(v)] every periodic trajectory $x(t)$ of $X_H$ satisfying
$$
 x(T-t)=x(-t)=\tau(x(t)),\;\;\forall t\in\mathbb{R}
$$
is either constant or has period $T>1$.
\end{enumerate}
By the definitions of the two capacities we deduce that
\begin{equation}\label{e:LargP0}
 c_{\rm HZ,\tau}(M,\omega)\le  2c_{\rm LR}(M, N,\omega)
 \end{equation}
(see also \cite[Remark~1.1]{JinLu1915}).

In this paper the canonical antisymplectic involution on $(\mathbb{R}^{2n},\omega_0)$ means
\begin{equation}\label{e:1.can-inv}
\tau_0:\mathbb{R}^{2n}\to\mathbb{R}^{2n},\;(q,p)\mapsto (q,-p).
\end{equation}
If a convex body $K\subset \mathbb{R}^{n}$ is centrally symmetric, i.e., $-K=K$,
its mean-width  is defined by
  \begin{equation}\label{e:mean-width}
M^\ast(K)=\int_{S^{n-1}}h_K(x)d\sigma_n(x)=
\int_{O(n)}h_K(Ax)d\mu_n(A)
 \end{equation}
 for any $x\in S^{n-1}$, where $\sigma_n$ is the normalized rotation invariant measure on $S^{n-1}$
and $\mu_n$ is the Haar measure (cf. \cite{AAO08} and \cite[(17)]{BoLiMi88}).
For any convex body $K\subset \mathbb{R}^{n}$ containing the origin in its interior, define
\begin{eqnarray}\label{e:BrunJ}
r_K=M^\ast(\widehat{K})\quad\hbox{with}\quad\widehat{K}:=\frac{1}{2}(K+(-K)).
\end{eqnarray}
Clearly,  $r_K=M^\ast(K)$ if $K$  is centrally symmetric.

\begin{corollary}\label{cor:LagrP}
\begin{enumerate}
  \item [\rm (1)]
For a $\tau_0$-invariant convex bounded domain $D\subset\mathbb{R}^{2n}$ it holds that
 \begin{equation}\label{e:LargP1}
  c_{\rm HZ}(D)\le c_{\rm HZ,\tau_0}(D,\omega_0)=2c_{\rm LR}(D, D\cap\mathbb{R}^{n,0})\le 2c_{\rm HZ}(D).
 \end{equation}
 \item[\rm (2)]
If a convex body $\Delta\subset\mathbb{R}^n_q$ contains the origin in its interior
and $\Lambda\subset\mathbb{R}^n_p$ is centrally symmetric,
 \begin{equation}\label{e:LargP2}
 c_{\rm LR}(\Delta\times\Lambda, (\Delta\times \Lambda)\cap\mathbb{R}^{n,0})\le 2r_\Delta r_\Lambda.
 \end{equation}
In particular, for any centrally symmetric convex body $\Delta\subset\mathbb{R}^n_q$,  there hold equalities
 \begin{equation}\label{e:LargP3}
 c_{\rm LR}(\Delta\times \Delta^\circ, (\Delta\times \Delta^\circ)\cap\mathbb{R}^{n,k})=2,\quad k=0,\cdots,n-1,
 \end{equation}
where $\Delta^{\circ}=\{p\in\mathbb{R}^{n}_p\,|\, \langle q,p\rangle\le 1\;\forall q\in \Delta\}$
is the polar body of $\Delta$.
\end{enumerate}
\end{corollary}
\begin{proof}
Using the representation formula for the symmetrical Hofer-Zehnder capacity
(\cite[Theorem~1.3]{JinLu1915}), we can obtain a generalized closed characteristic
 on $\partial D$, $x^{\ast}:[0,T]\to \partial D$, such that $x(T-t)=\tau_0x(t)\;\forall t$ and $A(x^{\ast})=c_{\rm HZ,\tau_0}(D)$.
These and  the representation formula for $c_{\rm HZ}(D)$ in \cite[Proposition~4]{HoZe90}
 yield the first inequality in (\ref{e:LargP1}). The second comes from
  (\ref{e:add2}).

Let us prove the equality in (\ref{e:LargP1}). By (\ref{e:LargP0}) we get
\begin{equation}\label{e:left}
c_{\rm HZ,\tau_0}(D, \omega_0)\le 2c_{\rm LR}(D, D\cap\mathbb{R}^{n,0}).
\end{equation}
Note that both $x^\ast(T/2)$ and $x^\ast(0)=x^\ast(T)$ belong to
$$
{\rm Fix}(\tau_0)=\{(q,p)\in\mathbb{R}^{2n}\,|\,p=0\}=
\mathbb{R}^{n,0}.
$$
 The restriction of $x^\ast$ to $[0, T/2]$, denoted by $y^\ast$,
is a generalized leafwise chord on $\partial D$ for $\mathbb{R}^{n,0}$.
It follows from Theorem~\ref{th:represention} that
\begin{eqnarray*}
c_{\rm LR}(D, D\cap\mathbb{R}^{n,0})\le A(y^\ast)=\frac{1}{2}A(x^\ast)=\frac{1}{2}c_{\rm HZ,\tau_0}(D,\omega_0).
\end{eqnarray*}
This and (\ref{e:left}) lead to  the equality in (\ref{e:LargP1}).

(\ref{e:LargP2}) follows from (\ref{e:LargP1}), and \cite[(1.62)]{JinLu1915} which claims
  $c_{\rm HZ,\tau_0}(\Delta\times\Lambda)\le 4r_\Delta r_\Lambda$.

By  \cite[(1.58)]{JinLu1915}, i.e., $c_{\rm HZ,\tau_0}(\Delta\times \Delta^\circ)=4$,
and (\ref{e:LargP1})  we get
$$
 c_{\rm LR}(\Delta\times \Delta^\circ, (\Delta\times \Delta^\circ)\cap\mathbb{R}^{n,0})=2,
$$
and hence
\begin{equation}\label{e:LargP3-Add1}
 c_{\rm LR}(\Delta\times \Delta^\circ, (\Delta\times \Delta^\circ)\cap\mathbb{R}^{n,k})\ge 2,\quad k=0,\cdots,n-1
 \end{equation}
by (\ref{e:add1}).
On the other hand, for each $k=0,\cdots,n-1$, by considering the leafwise chord $x$ with respect to $\mathbb{R}^{n,k}$ staring from $((0,\cdots,0,q_n),(0,\cdots,0))\in \partial (\Delta\times \Delta^\circ)$ with action $A(x)=2$, we obtain
\begin{equation}\label{e:LargP3-Add2}
c_{\rm LR}(\Delta\times \Delta^\circ, (\Delta\times \Delta^\circ)\cap\mathbb{R}^{n,k})\le 2.
\end{equation}
(\ref{e:LargP3}) follows from (\ref{e:LargP3-Add1}) and (\ref{e:LargP3-Add2}).
\end{proof}

For a $\tau_0$-invariant convex bounded domain $D\subset\mathbb{R}^{2n}$, Theorem~\ref{th:represention*}
and (\ref{e:LargP1})  lead to
 \begin{equation}\label{e:LargP1-Add1}
\frac{1}{2}c_{\rm HZ}(D)\le  c_{\rm LR}(D, D\cap\mathbb{R}^{n,k})\le c_{\rm HZ}(D),\quad k=0,\cdots,n-1.
 \end{equation}
 The first equality case is possible because
 $c_{\rm HZ}(B^{2n}(1))=\pi$, and
 $$
 c_{\rm LR}(B^{2n}(1), B^{2n}(1)\cap\mathbb{R}^{n,k})=\frac{\pi}{2}
 $$
 for each $k=0,\cdots,n-1$ by (\ref{e:ellipsoid}).
 The following example shows that  the first inequality in (\ref{e:LargP1-Add1}) can be strict in some cases.

 \begin{example}\label{ex:add}
 {\rm For the convex domain  $D:=E^2(1,2)\times D^2(1)\subset\mathbb{R}^4(q,p)=\mathbb{R}^2_q\times \mathbb{R}^2_p$,
   where
\begin{eqnarray*}
&&E^2(1,2)=\{(x_1,x_2)\,|\,x_1^2+\frac{x_2^2}{4}\le 1\}\subset\mathbb{R}^2_q\quad\hbox{and}\quad\\
&& D^2(1)=\{(y_1,y_2)\,|\,y_1^2+y_2^2\le 1\}\subset\mathbb{R}^2_p,
\end{eqnarray*}
 there hold
 \begin{eqnarray}\label{e:LargP1-Add2}
&&c_{\rm LR}(E^2(1,2)\times D^2(1),(E^2(1,2)\times D^2(1))\cap\mathbb{R}^{2,0})=2.\\
&&c_{\rm HZ}(E^2(1,2)\times D^2(1))=4.\label{e:LargP1-Add3}\\
&&c_{\rm LR}(E^2(1,2)\times D^2(1),E^2(1,2)\times D^2(1)\cap\mathbb{R}^{2,1})\ge\pi.\label{e:LargP1-Add4}
\end{eqnarray}}
  \end{example}
These will be proved in Section~\ref{sec:add}.

Let $D\subset\mathbb{R}^{2n}$ be a bounded convex  domain which is invariant under
the anti-symplectic involution $\tau$ on $(\mathbb{R}^{2n},\omega_0)$. Assume
 $\partial D\cap {\rm Fix}(\tau)\ne\emptyset$, and define
$\mathfrak{R} (\partial K, \tau)$ by
$$
\frac{\min\{A(x)>0\,|\,x\;\text{is a generalized $\tau$-brake closed characteristic on}\;\partial D\}}
{\min\{A(x)>0\,|\,x\;\text{is a generalized  closed characteristic on}\;\partial D\}}.
$$
%
%
When $\partial D$ is smooth, $\mathfrak{R} (\partial K, \tau)$ was called the \textsf{symmetric ratio} of the symmetric convex
 hypersurface $(\partial D, \tau)$ in \cite{KiKiKw20}.
Suppose that the involution $\tau$ is linear. By Lemma 2.29 in \cite{R12},  there exists a linear symplectic isomorphism $\Psi$
of $(\mathbb{R}^{2n},\omega_0)$ such that $\Psi \tau_0=\tau\Psi$.
It follows from \cite[(1.21) \& ((1.25)]{JinLu1915} and (\ref{e:LargP1}) that
$$
c_{\rm HZ,\tau}(D)=c_{\rm EH,\tau}(D)=c_{\rm EH,\tau_0}(\Psi D)=c_{\rm HZ,\tau_0}(\Psi D)\le 2c_{\rm HZ}(\Psi D)=2c_{\rm HZ}(D).
$$
Since $c_{\rm HZ}(D)\le c_{\rm HZ,\tau}(D)$ by definitions of $c_{\rm HZ,\tau}$ and $c_{\rm HZ}$,
from these, { \cite[Proposition~4]{HoZe90}} and \cite[Theorem~1.3]{JinLu1915}
we derive:

  \begin{corollary}\label{cor:S-ratio}
Let $\tau$ be a linear anti-symplectic involution  on $(\mathbb{R}^{2n},\omega_0)$ and let
$D\subset\mathbb{R}^{2n}$ be a $\tau$-invariant bounded convex  domain such that
 $\partial D\cap {\rm Fix}(\tau)\ne\emptyset$. Then
\begin{equation}\label{e:S-ratio}
1 \leq \mathfrak{R}(\partial D, \tau) \leq 2.
\end{equation}
\end{corollary}
This result is a special case of \cite[Theorem~1.3]{KiKiKw20}
where the involution $\tau$ is not required to be linear
though  $\partial D$ is assumed to be smooth (which can be
removed as in \cite[\S4.3]{JinLu1915}). Different from our methods
\cite[Theorem~1.3]{KiKiKw20} was proved  with Floer theory.

By (\ref{e:LargP1}) and (\ref{e:LargP1-Add2})-(\ref{e:LargP1-Add3}) (resp. two lines under (\ref{e:LargP1-Add1})) we have
$$
\mathfrak{R} (\partial (E^2(1,2)\times D^2(1)), \tau_0)=1\quad\hbox{(resp. $\mathfrak{R} (\partial B^{2n}(1), \tau_0)=1$).}
$$
Hence the first equality  in (\ref{e:S-ratio}) is possible. Moreover, if $n=1$ in Corollary~\ref{cor:S-ratio},
 by  {\cite[Proposition~4]{HoZe90}} and \cite[Theorem~1.3]{JinLu1915}
both $c_{\rm HZ}(D)$ and $c_{\rm HZ,\tau}(D)$ are equal to the area of $D$ and so $\mathfrak{R}(\partial D, \tau)=1$.

From properties of $c_{\rm LR}$, Theorem~\ref{th:represention} and Corollary~\ref{cor:ellipsoid+},
we deduce the following estimates for the action of leafwise chord on the boundary of convex domains which
generalize the well-known results for closed characteristics \cite[Theorem~C]{CrWe81} by Croke-Weinstein and
    \cite[Chapter~5, \S1, Proposition~5]{Ek90} by Ekeland. We also generalized the two theorems to closed brake
     characteristics (\cite[Corollary~1.17]{JinLu1915}) and to
    characteristics satisfying special boundary conditions (\cite[Corollary~1.16]{JinLu1916}).

\begin{corollary}\label{cor:CrokeW1}
Let  $D\subset \mathbb{R}^{2n}$ be a convex bounded domain with
 boundary $\mathcal{S}=\partial D$.
\begin{description}
\item[(i)] If $D$ contains a ball $B^{2n}({\bf a}, R)$ as in (\ref{e:ellipsoid++}) then for any
 generalized leafwise chord  $x$ on $\mathcal{S}$ for $D\cap\mathbb{R}^{n,k}$
  it holds with $r_R=\sqrt{1-a^2/R^2}$ that
\begin{equation}\label{e:croke1}
A(x)\ge \left(
\arcsin(r_R)-r_R\sqrt{1-r_R^2}\right)R^2.
\end{equation}
\item[(ii)] If $D$ is contained in a ball $B^{2n}({\bf a}, R)$ as in (\ref{e:ellipsoid++}),  there exists a
 generalized leafwise chord   $x^\star$ on $\mathcal{S}$ for  $B^{2n}({\bf a}, R)\cap\mathbb{R}^{n,k}$   such that
  with $r_R=\sqrt{1-a^2/R^2}$,
 \begin{equation}\label{e:croke2}
 A(x^\star)\le \left(
\arcsin(r_R)-r_R\sqrt{1-r_R^2}\right)R^2.
\end{equation}
\end{description}
 \end{corollary}

\begin{remark}\label{rem:chord}
{\rm Theorem~\ref{th:represention} can also yield some special cases of
the Arnold's chord conjecture.
For a bounded  star-shaped domain $D\subset\mathbb{R}^{2n}$ (with respect to the origin) with smooth boundary $\mathcal{S}$,
the canonical one-form  on $\mathbb{R}^{2n}$,
$$
\lambda_0=\frac{1}{2}\sum^{n}_{j=1}(q_jdp_j-p_jdq_j),
$$
 restricts to a contact form $\alpha$ on $\mathcal{S}$.  A \textsf{Reeb chord of length} $T$ to a Legendrian submanifold $\Lambda$ in $(\mathcal{S}, \alpha)$
is an orbit $\gamma:[0, T]\to \mathcal{S}$ of the Reeb vector field $R_\alpha$ with $\gamma(0),\gamma(T)\in\Lambda$.
\textsf{Arnold's chord conjecture} \cite{Ar86} stated that every closed Legendrian submanifold in $(\mathcal{S}, \alpha)$
has a Reeb chord. This was proved by  Mohnke \cite{Moh01}.

Choose a Hamiltonian function $H$ on $\mathbb{R}^{2n}$ such that $\mathcal{S}=H^{-1}(c)$ for some $c\in\mathbb{R}$.
Then by \cite[Lemma~1.4.10]{Ge08}
the Reeb flow of $\alpha$ is a reparametrisation of the Hamiltonian flow of $X_H$ on $\mathcal{S}$
because the Reeb vector field $R_\alpha$ of $\alpha$ is equal to $f X_H|_{\mathcal{S}}$, where $f$ is the restriction of $1/\lambda_0(X_H)$
to $\mathcal{S}$. In particular, suppose that the domain $D$ is convex and contains the origin in its interior.
We choose $H:=(j_D)^2$ and hence $\mathcal{S}=H^{-1}(1)$.
Let $\Lambda:=\mathcal{S}\cap\mathbb{R}^{n,0}$, which is a Legendrian submanifold in $\mathcal{S}$.
By {Remark~\ref{rem:CLRcarrier}}
we have a smooth path
 $x:[0,\mu]\to \mathcal{S}$ such that $\dot{x}(t)\in  X_H(x(t))$, and $x(0),\; x(\mu)\in \mathbb{R}^{n,0}$,
 where  $\mu=c_{\rm LR}(D,D\cap \mathbb{R}^{n,k})=A(x)$ is the return time of $x$ for $\mathbb{R}^{n,k}$.
Let $h(s)=\int^s_0\lambda_0(X_H(x(\tau))d\tau$ for $s\in [0, \mu]$ and $T=h(\mu)$. Then $h:[0,\mu]\to [0, T]$
has a smooth inverse $g:[0, T]\to [0,\mu]$. Define $\gamma:[0, T]\to\mathcal{S}$ by $\gamma(t)=x(g(t))$.
It is  a Reeb chord of length $h(\mu)$ to a Legendrian submanifold $\mathcal{S}\cap\mathbb{R}^{n,0}$ in $(\mathcal{S}, \alpha)$.
By the reverse reasoning it is easily seen that such a Reeb chord has the shortest length.}
\end{remark}

\noindent{1.2.2}. {\bf A Brunn-Minkowski type inequality for coisotropic Hofer-Zehnder capacities}.\;
 Let $h_K$ be the  support  function of a nonempty   convex subset $K\subset \mathbb{R}^{2n}$  defined by
 \begin{equation}\label{e:supp.1}
h_K(w)=\sup\{\langle x,w\rangle\,|\,x\in K\},\quad\forall w\in\mathbb{R}^{2n}.
\end{equation}
For a compact convex subset $K\subset \mathbb{R}^{2n}$  containing $0$ in its interior,
and its polar body $K^{\circ}$ defined below (\ref{e:LargP3}),
by  \cite[Theorem~1.7.6]{Sch93} and (\ref{e:Brunn.0}) we get respectively
 \begin{equation}\label{e:supp.1.1}
h_K=j_{K^{\circ}}\quad\hbox{and}\quad (h_K)^2= 4H^\ast_K,
\end{equation}
 where $H^\ast_K$ is the Legendre transform of $H_K:=(j_K)^2$ defined by
 $$
 H^\ast_K(z):=\max_{\xi\in\mathbb{R}^{2n}}(\langle\xi,z\rangle-H_K(\xi)).
 $$

For two bounded convex  domains $D, K\subset\mathbb{R}^{2n}$ containing  $0$
 and a real number $p\ge 1$,  there exists a unique bounded convex domain
 $D+_pK\subset\mathbb{R}^{2n}$ containing $0$ with support function
   \begin{equation}\label{e:supp.2}
\mathbb{R}^{2n}\ni w\mapsto h_{D+_pK}(w)=(h^p_{D}(w)+h_{K}^p(w))^{\frac{1}{p}}
 \end{equation}
  (\cite[Theorem~1.7.1]{Sch93}), which is called
   the $p$-sum of $D$ and $K$ by Firey (cf. \cite[(6.8.2)]{Sch93}).
   For convex domains $D$, $K\subset\mathbb{R}^{2n}$ not necessarily containing $0$ and p=1, define
   $$
   D+_1K:=D+K=\{x+y\,|\, x\in D,\,y\in K\},
   $$
which coincides with one defined above if both $D$ and $K$ contain $0$.

Corresponding to the Brunn-Minkowski type inequality for the Hofer-Zehnder capacities of convex domains given
 by \cite[Theorem~1.5]{AAO08}, we have the following inequality for the coisotropic
Hofer-Zehnder capacities of convex domains.

\begin{thm}\label{th:Brun}
   Let $D, K\subset \mathbb{R}^{2n}$ be two bounded convex  domains  containing  $0$.
      Then for any  real $p\ge 1$ it holds that
   \begin{eqnarray}\label{e:BrunA}
   &&\left(c_{\rm LR}(D+_pK, (D+_pK)\cap\mathbb{R}^{n,k} )\right)^{\frac{p}{2}}\nonumber\\
   &&\ge
    \left(c_{\rm LR}(D, D\cap\mathbb{R}^{n,k})\right)^{\frac{p}{2}}+ \left(c_{\rm LR}(K, K\cap\mathbb{R}^{n,k})\right)^{\frac{p}{2}}.
   \end{eqnarray}
   Moreover,  the equality holds if  there exist
    a $c_{\rm LR}(D, D\cap\mathbb{R}^{n,k})$ carrier,
$\gamma_D:[0, T]\to\partial D$,  and
a $c_{\rm LR}(K, K\cap\mathbb{R}^{n,k})$ carrier,
 $\gamma_K:[0, T]\to\partial K$,
such that they coincide up to  dilations and translations in $\mathbb{R}^{n,k}$, i.e., $\gamma_D=\alpha\gamma_K+{\bf b}$
for some $\alpha\in\mathbb{R}\setminus\{0\}$ and
some ${\bf b}\in \mathbb{R}^{n,k}$;
   and in the case $p>1$ the latter condition is also necessary for the equality in (\ref{e:BrunA}) holding.
  \end{thm}

\begin{corollary}\label{cor:Brun.1}
 For two bounded convex  domains $D, K\subset \mathbb{R}^{2n}$,
  there holds
  \begin{eqnarray*}
&& \left(c_{\rm LR}(D+K, (D+K)\cap(\mathbb{R}^{n,k}+{\bf a}+{\bf b}))\right)^{\frac{1}{2}}\\
&\ge& \left(c_{\rm LR}(D, D\cap(\mathbb{R}^{n,k}+{\bf a}))\right)^{\frac{1}{2}}+
\left(c_{\rm LR}(K, K\cap(\mathbb{R}^{n,k}+{\bf b}))\right)^{\frac{1}{2}}
 \end{eqnarray*}
  for any ${\bf a}\in D$ and ${\bf b}\in K$.
\end{corollary}

\begin{proof}
For any ${\bf a}\in D$ and ${\bf b}\in K$, the monotonicity of $c_{\rm LR}$  and (\ref{e:BrunA}) lead to
\begin{eqnarray*}
&& \left(c_{\rm LR}(D+K, (D+K)\cap(\mathbb{R}^{n,k}+{\bf a}+{\bf b}))\right)^{\frac{1}{2}}\\
&=& \left(c_{\rm LR}(D+K-{\bf a}-{\bf b}, (D+K-{\bf a}-{\bf b})\cap\mathbb{R}^{n,k})\right)^{\frac{1}{2}}\\
&\ge& \left(c_{\rm LR}(D-{\bf a}, (D-{\bf a})\cap\mathbb{R}^{n,k})\right)^{\frac{1}{2}}+
\left(c_{\rm LR}(K-{\bf b}, (K-{\bf b})\cap\mathbb{R}^{n,k})\right)^{\frac{1}{2}}\\
&=& \left(c_{\rm LR}(D, D\cap(\mathbb{R}^{n,k}+{\bf a}))\right)^{\frac{1}{2}}+
\left(c_{\rm LR}(K, K\cap(\mathbb{R}^{n,k}+{\bf b}))\right)^{\frac{1}{2}}.
 \end{eqnarray*}
\end{proof}

As in \cite{AAO08, AAO14} Corollary~\ref{cor:Brun.1} has

\begin{corollary}\label{cor:Brun.2}
 Let $D$ and $K$ be as in Corollary~\ref{cor:Brun.1}.
 \begin{description}
\item[(i)] If $0\in D$ and $\{-x,-y\}\subset K$,  then for any $0\le \lambda\le 1$, there holds
 \begin{eqnarray*}
&&\lambda\left(c_{\rm LR}(D\cap (x+K),  D\cap (x+K)\cap\mathbb{R}^{n,k})\right)^{1/2}\nonumber\\
&&+(1-\lambda)\left(c_{\rm LR}(D\cap (y+K),  D\cap (y+K)\cap\mathbb{R}^{n,k})\right)^{1/2}\nonumber\\
&\le&\left(c_{\rm LR}(D\cap(\lambda x+(1-\lambda)y+K),  D\cap(\lambda x+(1-\lambda)y+K)\cap\mathbb{R}^{n,k})\right)^{1/2}.
\end{eqnarray*}
  In particular, if $D$ and $K$ are centrally symmetric, i.e., $-D=D$ and $-K=K$, then
     \begin{eqnarray}\label{e:BrunB}
   c_{\rm LR}(D\cap(x+K),  D\cap(x+K)\cap\mathbb{R}^{n,k})\le c_{\rm LR}(D\cap K, D\cap K\cap\mathbb{R}^{n,k}).
   \end{eqnarray}

\item[(ii)] Suppose that $0\in D\cap K$. Then the limit
\begin{equation}\label{e:BrunC}
\lim_{\varepsilon\to 0+}
\frac{c_{\rm LR}(D+ \varepsilon K, (D+\varepsilon K)\cap\mathbb{R}^{n,k})- c_{\rm LR}(D, D\cap\mathbb{R}^{n,k})}{\varepsilon}
\end{equation}
exists, denoted by $d_K(D, D\cap\mathbb{R}^{n,k})$, and it holds that
\begin{eqnarray}\label{e:BrunD}
2(c_{\rm LR}(D, D\cap\mathbb{R}^{n,k}))^{1/2}(c_{\rm LR}(K, K\cap\mathbb{R}^{n,k}))^{1/2}\nonumber\\
\le d_K(D, D\cap\mathbb{R}^{n,k})\le
\inf_{z_D}\int_0^1h_K(-J\dot{z}_D(t))dt,
\end{eqnarray}
where the infimum is over all $c_{\rm LR}(D, D\cap\mathbb{R}^{n,k})$-carriers $z_D:[0,1]\to\partial D$.
\end{description}
\end{corollary}

 Following  \cite{AAO08, AAO14} we define  the length of $z_D$ with respect to the convex body $JK^\circ$ by
$$
{\rm length}_{JK^\circ}(z_D):=\int_0^1j_{JK^\circ}(\dot{z}_D(t))dt.
$$
Since $H^\ast_K(-Jv)=(j_{JK^\circ}(v))^2/4$ by (\ref{e:supp.1.1}),  (\ref{e:BrunD}) implies
$$
d_K(D, D\cap\mathbb{R}^{n,k})\le 2\inf_{z_D}\int_0^1(H_K^{\ast}(-J\dot{z}_D(t)))^{\frac{1}{2}}dt=
\inf_{z_D}\int_0^1j_{JK^\circ}(\dot{z}_D(t))dt
$$
and hence $4c_{\rm LR}(D, D\cap\mathbb{R}^{n,k})c_{\rm LR}(K, K\cap\mathbb{R}^{n,k})\le
\inf_{z_D}({\rm length}_{JK^\circ}(z_D))^2$.\\

\noindent{1.2.3}. {\bf  An extension of a theorem by Evgeni Neduv}.\;
Let $\mathscr{H}\in C^2(\mathbb{R}^{2n},\mathbb{R}_+)$ be a proper and strictly convex Hamiltonian
such that $\mathscr{H}(0)=0$ and $\mathscr{H}''>0$ (i.e. $\mathscr{H}'' $ is positive definite at each point),  which imply
$\mathscr{H}\ge 0$ by Taylor's formula.
 If $e_0> 0$ is a regular value of $\mathscr{H}$ with
$\mathscr{H}^{-1}(e_0)\ne\emptyset$, then for each number $e$ near $e_0$
the set $D(e):=\{\mathscr{H}<e\}$
is a bounded strictly convex  domain in $\mathbb{R}^{2n}$ containing  $0$ and
with $C^2$-boundary $\mathcal{S}(e)=\mathscr{H}^{-1}(e)$.
Based on the representation formula for Hofer-Zehnder capacities of convex domains (\cite[Proposition~4]{HoZe90}),
Evgeni Neduv \cite{Ned01} studied  the differentiability of  the Hofer-Zehnder capacity $c_{\rm HZ}(D(e))$
 at $e=e_0$ and obtained some applications to the prescribed periodic problems.
The main result \cite[Theorem~4.4]{Ned01} was generalized by us to brake periodic orbits (\cite{JinLu1915})
and non-periodic orbits satisfying certain boundary conditions (\cite{JinLu1916}) recently.
 In the following we consider corresponding generalizations to leafwise chord using the representation
 formula for coisotropic Hofer-Zehnder capacities Theorem \ref{th:represention}.

For any $e$ near $e_0$ let $\mathscr{C}_k(e):=c_{\rm LR}(D(e),D(e)\cap \mathbb{R}^{n,k})$.
As remarked below Theorem~\ref{th:represention} all $c_{\rm LR}(D(e),D(e)\cap \mathbb{R}^{n,k})$-carriers
form a compact subset in $C^1([0, \mathscr{C}_k(e)], \mathcal{S}(e))$. Hence
\begin{equation}\label{e:convexDiff}
\mathscr{I}_k(e):=\left\{T_x=2\int^{\mathscr{C}(e)}_0\frac{dt}{\langle\nabla\mathscr{H}(x(t)), x(t)\rangle}\,\bigg|\hspace{-4mm}
 \begin{array}{ll}
 &\hbox{$x$ is a}\; c_{\rm LR}(D(e),D(e)\cap \mathbb{R}^{n,k})\\
 & \hbox{-carrier}
 \end{array}\right\}
\end{equation}
is a compact subset in $\mathbb{R}$. Denote by $T_k^{\max}(e)$ and $T_k^{\min}(e)$
 the largest and smallest numbers in $\mathscr{I}_k(e)$.
By the reparameterization every $c_{\rm LR}(D(e),D(e)\cap \mathbb{R}^{n,k})$-carrier $x$ yields a solution of
\begin{equation}\label{e:convexDiff+}
-J\dot{y}(t)=\nabla\mathscr{H}(y(t)),
\end{equation}
$y:[0, T_x]\to \mathcal{S}(e)=\mathscr{H}^{-1}(e)$, such that
\begin{equation}\label{e:convexDiff++}
\left.\begin{array}{ll}
&\hbox{$y(0), y(T_x)\in \mathbb{R}^{n,k}$, $y(0)-y(T_x)\in V_0^{n,k}$,   and for each $t\in (0, T_{x})$}\\
&\hbox{there holds either  $y(t)\notin \mathbb{R}^{n,k}$  or $y(0)-y(t)\notin V_0^{n,k}$},
 \end{array}\right\}
\end{equation}
 namely,
$T_{x}$ is the return time of $y$ for $\mathbb{R}^{n,k}$ and the leaf relation $\sim$.
 Corresponding to \cite[Theorem~4.4]{Ned01} we have:

\begin{thm}\label{th:convexDiff}
Let  $\mathscr{H}\in C^2(\mathbb{R}^{2n},\mathbb{R}_+)$ be as above.
Then $\mathscr{C}_k(e)$ has left and right derivatives
at $e_0$, $\mathscr{C}'_{k,-}(e_0)$ and $\mathscr{C}'_{k,+}(e_0)$, and they satisfy
\begin{eqnarray*}
&&\mathscr{C}'_{k,-}(e_0)=\lim_{\epsilon\to0-}T_k^{\max}(e_0+\epsilon)=T_k^{\max}(e_0)\quad\hbox{and}\\
&&\mathscr{C}'_{k,+}(e_0)=\lim_{\epsilon\to0+}T_k^{\min}(e_0+\epsilon)=T_k^{\min}(e_0).
\end{eqnarray*}
Moreover, if  $[a,b]\subset (0,\sup\mathscr{H})$ is a regular interval of $\mathscr{H}$ such that
$\mathscr{C}'_{k,+}(a)<\mathscr{C}'_{k,-}(b)$, then for any  $r\in (\mathscr{C}'_{k,+}(a),
\mathscr{C}'_{k,-}(b))$ there exist $e'\in (a,b)$ such that $\mathscr{C}_k(e)$
is differentiable at $e'$ and $\mathscr{C}'_{k,-}(e')=\mathscr{C}'_{k,+}(e')=r=T_k^{\max}(e')=T_k^{\min}(e')$.
\end{thm}

As a monotone function on a regular interval $[a,b]$ of $\mathscr{H}$ as above, $\mathscr{C}_k(e)$
satisfies $\mathscr{C}'_{k,-}(e)=\mathscr{C}'_{k,+}(e)$ for almost all values of $e\in [a, b]$
and thus both $T_k^{\max}$ and $T_k^{\min}$ are almost everywhere continuous.
Actually, the first claim of Theorem~\ref{th:convexDiff} and a recent result \cite[Corollary~6.4]{Bl14}
imply that both $T_k^{\max}$ and $T_k^{\min}$ have only at most countable discontinuous
points and are also Riemann integrable on $[a,b]$.

By Theorem~\ref{th:convexDiff}, for any regular interval $[a,b]\subset (0,\sup \mathscr{H})$
 of $\mathscr{H}$ with $\mathscr{C}'_{k,+}(a)\le\mathscr{C}'_{k,-}(b)$,
if $T\in [\mathscr{C}'_{k,+}(a), \mathscr{C}'_{k,-}(b)]$ then
(\ref{e:convexDiff+}) has a solution
$y:[0, T]\to \mathscr{H}^{-1}([a,b])$ such that
(\ref{e:convexDiff++}) holds with $T_x=T$.
 For example, we have

 \begin{corollary}\label{cor:convexDiff}
Suppose that a proper and strictly convex Hamiltonian $\mathscr{H}\in C^2(\mathbb{R}^{2n},\mathbb{R}_+)$
satisfies the conditions:
\begin{enumerate}
\item[(i)] $\mathscr{H}(0)=0$, $\mathscr{H}''>0$ and every $e>0$ is a regular value of $\mathscr{H}$,
\item[(ii)] there exist positive numbers $r_j, R_j$, $j=1,\cdots,n$ such that
\begin{eqnarray}\label{e:convexDiff+++}
&&\min\{r^2_1,\cdots,r^2_k, r^2_{k+1}/2,\cdots, r^2_n/2\}\nonumber\\
&&\le\min\{R^2_1,\cdots, R^2_k, R^2_{k+1}/2,\cdots, R^2_n/2\}
\end{eqnarray}
    and that for $x=(q_1,\cdots,q_n,p_1,\cdots,p_n)$
    with   small (resp. large) norm     $\mathscr{H}(x)$ is equal to
    \begin{center}
     $q(x):=\sum^n_{j=1}(q^2_j+p^2_j)/r_j^2$ {\rm (}resp.
    $Q(x):=\sum^n_{j=1}(q^2_j+p^2_j)/R_j^2${\rm )}.
 \end{center}
 \end{enumerate}
Then for every $T=T_0\pi$, where $T_0$ sits between the two numbers in (\ref{e:convexDiff+++}), the corresponding system
(\ref{e:convexDiff+}) has a solution
$y:[0, T]\to \mathbb{R}^{2n}$ such that (\ref{e:convexDiff++}) holds with $T_x=T$.
\end{corollary}

In fact, if $e>0$ is small (resp. large) enough then
$D(e)$ is equal to
\begin{eqnarray*}
&&\hbox{$D_q(e):=\{q<e\}=\sqrt{e}E(r_1,\cdots,r_n)$}\\
&&\hbox{(resp.  $D_Q(e):=\{Q< e\}=\sqrt{e}E(R_1,\cdots,R_n)$)}
\end{eqnarray*}
and so
\begin{eqnarray*}
&&\hspace{-4mm}\hbox{$c_{\rm LR}(D(e),D(e)\cap \mathbb{R}^{n,k})=e\pi\min\{r^2_1,\cdots,r^2_k, r^2_{k+1}/2,\cdots, r^2_n/2\}$}\\
&&\hspace{-4mm}\hbox{(resp. $c_{\rm LR}(D(e),D(e)\cap \mathbb{R}^{n,k})=e\pi\min\{R^2_1,\cdots, R^2_k, R^2_{k+1}/2,\cdots, R^2_n/2\}$)}.
\end{eqnarray*}
By these and  Corollary~\ref{cor:ellipsoid},
$\mathscr{C}'_k(a)=\pi\min\{r^2_1,\cdots,r^2_k, r^2_{k+1}/2,\cdots, r^2_n/2\}$ for $a>0$ small enough,
and  $\mathscr{C}'_k(b)=\pi\min\{R^2_1,\cdots, R^2_k, R^2_{k+1}/2,\cdots, R^2_n/2\}$ for $b>0$ large enough.
From Theorem~\ref{th:convexDiff} we derive Corollary~\ref{cor:convexDiff}  immediately.\\

\begin{remark}\label{rm:add}
{\rm  Pazit Haim-Kislev \cite{Haim19} gave
 a combinatorial formula for the Ekeland-Hofer-Zehnder capacity of a convex polytope in $\mathbb{R}^{2n}$,
 and used it to prove  subadditivity of this capacity for hyperplane cuts of arbitrary convex domains.
 Motivated by this work,  Kun Shi and the second named author \cite{ShiLu2021} used Theorem~\ref{th:represention}
 to derive the corresponding combinatorial formula of $c_{\rm LR}(D,D\cap \mathbb{R}^{n,k})$
 for any  convex  polytope $D\subset\mathbb{R}^{2n}$ such that $D\cap \mathbb{R}^{n,k}\ne\emptyset$.
However, contrary to Pazit Haim-Kislev’s subadditivity result for the Ekeland-Hofer-Zehnder capacities
of convex domains, we showed that the coisotropic Hofer-Zehnder capacities
satisfy the superadditivity for suitable hyperplane cuts of two-dimensional
convex domains.}
\end{remark}


\noindent{\bf Outline of the paper}.
In the next section  we give  variational settings and related preparations
on the basis  of \cite{LiRi13}. In particular, our Proposition~\ref{prop:L^2} shows that
the Hilbert subspace $L^2_{n,k}$ of $L^2([0,1],\mathbb{R}^{2n})$ in \cite[Definition~3.6]{LiRi13}
is exactly $L^2([0,1],\mathbb{R}^{2n})$.
 Section~\ref{sec:3} proves Theorem~\ref{th:represention}.
  Choices of Banach spaces $\mathscr{F}_p^{n,k}$ in (\ref{e:BanachSpace}) and
 Lemmas~\ref{lem:3.1} and \ref{lem:3.2} are keys completing proof.
Theorem~\ref{th:represention*},  Corollaries~\ref{cor:ellipsoid},\ref{cor:ellipsoid+}, \ref{cor:ellipsoid++}
and \ref{cor:polyDisk}
are proved in Section~\ref{sec:4}.
 In the setting of Section~\ref{sec:3} Theorems~\ref{th:Brun} and \ref{th:convexDiff} are easily proved
  as in \cite{JinLu1915, JinLu1916}. For completeness we outline their proofs in
   Sections~\ref{sec:Brunn} and ~\ref{sec:Neduv}, respectively.
  In Section~\ref{sec:add} we give proofs of Theorem~\ref{th:J-estimate} and Example~\ref{ex:add}.
Section~\ref{sec:viterbo} proposes an analogue of Viterbo's volume-capacity conjecture for the coisotropic capacity $c_{\rm LR}$
and gives some example to support it.

\section{Variational settings and  preparations}\label{sec:2}
\setcounter{equation}{0}

Fix a real $p\ge 1$.
 Recall that every element of Banach space $W^{1,p}([0,1],\mathbb{R}^{2n})$
is actually an equivalence class of almost equal Lebesgue-integrable $\mathbb{R}^{2n}$-valued functions on $[0, 1]$ everywhere which contains
an $\mathbb{R}^{2n}$-valued function having $L^p$-integrable weak derivative on $[0,1]$.
Moreover, every  $x\in W^{1,p}([0,1],\mathbb{R}^{2n})$ contains only
a continuous $\mathbb{R}^{2n}$-valued function representative, denoted by $\underline{x}$ for clearness,
which is also absolutely continuous on $[0, 1]$.
The embedding  $W^{1,p}([0,1],\mathbb{R}^{2n})\ni x\mapsto \underline{x}\in C^0([0,1],\mathbb{R}^{2n})$
is compact.
Without occurring of confusions, an element $x$ of $W^{1,p}([0,1],\mathbb{R}^{2n})$ is simply called a function,
and $\underline{x}(t)$ is written as $x(t)$ for any $t\in [0, 1]$ in this paper.
Consider the following Banach subspace of $W^{1,p}([0,1],\mathbb{R}^{2n})$,
\begin{equation*}
X_p:=\{x\in W^{1,p}([0,1],\mathbb{R}^{2n})\,| \,\underline{x}(0),\, \underline{x}(1)\in L^n_0\},
\end{equation*}
which  is dense in $L^p([0,1],\mathbb{R}^{2n})$.
 The complex structure $J_{2n}$  gives rise to
 a real linear, unbounded Fredholm operator
 on $L^p([0,1],\mathbb{R}^{2n})$ with domain ${\rm dom} (\Lambda_p)=X_p$,
\begin{equation*}
\Lambda_p:=-J_{2n}\frac{d}{dt}.
\end{equation*}
(For $p=2$, this is also selj-adjoint and has pure spectrum without limit point
by \cite[Lemma~3.1]{CLM94}.)
By identifying ${\bf a}\in L^n_0 $ with a constant path in $L^p([0,1],\mathbb{R}^{2n})$ given by
$\hat{\bf a}(t)={\bf a}$ for all $t\in [0,1]$, we can identify
${\rm Ker} (\Lambda_p)$ with $L^n_0$, and
write ${\rm Ker} (\Lambda_p)=L^n_0$ below (without occurring of confusions).
The following proposition
can be  easily derived by improving the proof of Proposition~2 in \cite[Chap. III, Sec. 1]{Ek90} (or \cite[Lemma~3.1]{CLM94}).

\begin{prop}\label{prop:eigenvalues}
\begin{description}
    \item[(i)] The range $R(\Lambda_p)$ of $\Lambda_p$ is closed in $L^p([0,1],\mathbb{R}^{2n})$ and there exists the following direct sum decomposition
        \begin{equation}\label{e:split}
        L^p([0,1],\mathbb{R}^{2n})= {\rm Ker} (\Lambda_p)+ R(\Lambda_p),
        \end{equation}
        such that $\int^1_0\langle x(t), y(t)\rangle dt=0$ for all $x\in R(\Lambda_p)$ and  $y\in{\rm Ker} (\Lambda_p)$.
    \item[(ii)]  The restriction ~$\Lambda_{p,0}:=\Lambda_p|_{R(\Lambda_p)\cap {\rm dom}(\Lambda_p)}$ is
    bijection onto $R(\Lambda_p)$, and $\Lambda_{p,0}^{-1}$ (as an operator from $R(\Lambda_p)$ to itself) is  compact (and
        self-adjoint for $p=2$).
\end{description}
\end{prop}

 \begin{proof}

\noindent{\bf Step 1}. The first claim  follows from \cite[Ex.~6.9]{Bre} directly.
Let us prove (\ref{e:split}).
For any fixed $x\in L^p([0,1],\mathbb{R}^{2n})$,
by the orthogonal decomposition (\ref{lagmultipier1}) (with $k=0$)
we may write
$\int_0^1 J_{2n}x(s)ds=J_{2n}{\bf a}+{\bf b}$,
where ${\bf a}, {\bf b}\in L_0^n$. For $t\in [0,1]$ let
$$
u(t):=\int_0^t J_{2n}(x(s)-{\bf a})ds.
$$
Then $u\in W^{1,p}([0,1],\mathbb{R}^{2n})$, $u(0)=0$, $u(1)={\bf b}\in L^n_0$ and $x=\Lambda_p u+ {\bf a}$.
For any ${\bf c}\in {\rm Ker} (\Lambda_p)=L^n_0$ and $y=\Lambda_p w$, where $w\in X_p$,  since $w(1)$, $w(0)\in L^n_0$ there holds
$$
\langle {\bf c},y\rangle_{L^2}=\int_0^1 \langle {\bf c}, -J\dot{w}\rangle dt
=\langle {\bf c}, -J_{2n}(w(1)-w(0))\rangle=0.
$$

\noindent{\bf Step 2}.
Clearly, $R(\Lambda_p)\cap X_p$ is a closed subspace in $X_p$
since it is equal to the inverse image of $R(\Lambda_p)$ under the continuous
map $\Lambda_p:X_p\to L^p([0,1],\mathbb{R}^{2n})$.

We claim that \textsf{$\Lambda_{p,0}:=\Lambda_p|_{R(\Lambda_p)\cap X_p}$ is a bijective continuous linear map from a Banach subspace ${\rm dom} (\Lambda_{p,0})=
R(\Lambda_p)\cap X_p$ of $X_p$ to the Banach subspace $R(\Lambda_p)$ of $L^p([0,1],\mathbb{R}^{2n})$}.
Indeed, for any $x\in R(\Lambda_p)$, there exists $u\in X_p$ such that
$\Lambda_p u=x$. By (\ref{e:split}) we have
$u=\Lambda_p w+{\bf a}$, where $w\in X_p$ and ${\bf a}\in L^n_0$. Let
 $\hat{u}:=u-{\bf a}$.  Then
 $\hat{u}\in R(\Lambda_p)\cap X_p$ and $\Lambda_{p,0}\hat{u}=\Lambda_p u=x$,
   which implies that $\Lambda_{p,0}$ is surjective. Moreover, suppose that $u_1$, $u_2\in R(\Lambda_p)\cap X_p$
satisfies $\Lambda_{p,0} u_1=\Lambda_{p,0} u_2$. Then we have ${\bf c}\in L^n_0$ and $w_1$, $w_2\in X_p$ such that
$u_1-u_2={\bf c}$, $u_1=\Lambda_p w_1$ and $u_2=\Lambda_p w_2$. These imply
${\bf c}=\Lambda_p (w_1-w_2)\in R(\Lambda_p)$, and so ${\bf c}=0$ and $u_1=u_2$.
Therefore $\Lambda_{p,0}$ is bijective.
By the Banach inverse operator theorem we get a
continuous linear operator
$\Lambda_{p,0}^{-1}$ from $R(\Lambda_p)$ to the Banach subspace $R(\Lambda_p)\cap X_p$ of $X_p$.
  Note that the inclusion map
$i_p:R(\Lambda_p)\cap X_p\hookrightarrow R(\Lambda_p)$ (as a restriction of the compact inclusion map $W^{1,p}\hookrightarrow L^p$) is compact.
  Hence
  $i_p\circ \Lambda_{p,0}^ {-1 }:R(\Lambda_p)\rightarrow R(\Lambda_p)$
   is compact.

   Finally, let us prove  that $i_2\circ \Lambda_{2,0}^{-1 }$ is also self-adjoint. In fact, for
   $x_j=\Lambda u_j\in R(\Lambda_2)$, $j=1,2$,
    where $u_1$, $u_2\in R(\Lambda_2)\cap X_2$, since
    $u_1(j),\;u_2(j)\in L^n_0$ for $j=0,1$,
        it follows that
  \begin{eqnarray*}
  &&\langle i_2\circ\Lambda_{2,0}^{-1}x_1,x_2\rangle_{L^2}=\langle u_1,\Lambda_2 u_2\rangle_{L^2}
  =\int_0^1 \langle u_1(t),-J_{2n}\dot{u}_2(t)\rangle dt\\
  &=&\int_0^1 \langle -J_{2n}\dot{u}_1(t),u_2(t)\rangle dt+\langle u_1(1),-J_{2n}u_2(1)\rangle-\langle u_1(0),-J_{2n}u_2(0)\rangle\\
  &=&\int_0^1 \langle -J_{2n}\dot{u}_1(t),u_2(t)\rangle dt
 = \langle x_1,i_2\circ\Lambda_{2,0}^{-1}x_2\rangle_{L^2}.
  \end{eqnarray*}
\end{proof}

For each $i=1,\cdots,n$, let $e_i$ be a vector in $\mathbb{R}^{2n}$ with $1$ in the $i$-th position and $0$s elsewhere.
 Clearly, $\{e_i\}_{i=1}^n$ is an orthogonal  basis for $L^n_0$.

\begin{cor}\label{cor:eigenvalues}
$L^2([0,1],\mathbb{R}^{2n})$ has an orthogonal basis
$\{e^{m\pi tJ_{2n}}e_i\}_{1\le i\le n, m\in\mathbb{Z}}$,
and every $x\in L^2([0,1],\mathbb{R}^{2n})$ can be uniquely expanded as
\begin{equation}\label{e:expand}
x=\sum_{m\in\mathbb{Z}}e^{m\pi tJ_{2n}}x_m,
\end{equation}
where $x_m\in L_0^n$ for all $m\in\mathbb{Z}$, satisfy
$\sum_{m\in\mathbb{Z}}|x_m|^2<\infty$.
\end{cor}

\begin{proof}
Since $R(\Lambda_2)$ is a closed subspace of the separable Hilbert space $L^2([0,1],\mathbb{R}^{2n})$,
by the standard linear functional analysis theory, there exists an orthogonal basis of
$R(\Lambda_2)$ which completely consists of eigenvectors of $i_2\circ\Lambda_{2,0}^{-1}$.

Suppose that $y\in R(\Lambda_2)$ is a nonzero eigenvector of $i_2\circ\Lambda_{2,0}^{-1}$ which belongs to the
eigenvalue $\lambda$. Then $\lambda\ne 0$, $y\in R(\Lambda_2)\cap X_2$ and satisfies $\Lambda_2 (\lambda y)=y$, which implies that
$-J_{2n}\dot{y}=y/{\lambda}$ and $y(i)\in L^n_0$, $i=0,1$.
It follows that $y$ is smooth and has the form
$y(t)=e^{J_{2n}t/\lambda}y(0)$.
Since
$$
y(1)=e^{J_{2n}/\lambda}y(0)=\cos(1/\lambda)y(0)+\sin(1/\lambda)J_{2n}y(0)\in L_0^n,\quad y(0)\in L_0^n\setminus\{0\},
$$
we obtain $\sin(1/\lambda)=0$, and so
$1/\lambda=k\pi$ for some $k\in\mathbb{Z}\setminus\{0\}$. Moreover,
for each $k\in\mathbb{Z}\setminus\{0\}$ it is easily checked that $\frac{1}{k\pi}$ is
an eigenvalue of $i_2\circ\Lambda_{2,0}^{-1}$, and has the eigenvector subspace
$e^{k\pi tJ_{2n}}L_0^n$. It follows that
$$
\{e^{m\pi tJ_{2n}}e_i\}_{1\le i\le n, m\in\mathbb{Z}}\setminus\{e_i\}_{i=1}^n
$$
is an orthogonal basis for the closed subspace $R(\Lambda_2)$ of $L^2([0,1],\mathbb{R}^{2n})$.
Then  this and (\ref{e:split}) lead to the desired conclusions.
\end{proof}

Consider the Hilbert space defined in \cite[Definition 3.6]{LiRi13}
\begin{eqnarray*}
L^2_{n,k}=\Big\{x\in L^2([0,1],\mathbb{R}^{2n})\,\Big|\!\!\!\! &&x\stackrel{L^2}{=}\sum_{m\in\mathbb{Z}}e^{m\pi tJ_{2n}}a_m+\sum_{m\in\mathbb{Z}}e^{2m\pi tJ_{2n}}b_m\\
&&a_m\in V^{n,k}_0,\quad b_m\in V^{n,k}_1,\\
&&\sum_{m\in\mathbb{Z}}(|a_m|^2+|b_m|^2)<\infty
\Big\}
\end{eqnarray*}
with inner product
$$
\langle\psi,\phi\rangle_{L^2_{n,k}}=\left(\int^1_0\langle\psi(t),\phi(t)\rangle dt\right)^{\frac{1}{2}}.
$$

\begin{prop}\label{prop:L^2}
The Hilbert space $L^2_{n,k}$ is exactly $L^2([0,1],\mathbb{R}^{2n})$.
\end{prop}

\begin{proof}
Consider the projections
\begin{eqnarray*}
&&\hspace{-6mm}\Pi_1:\mathbb{R}^{2n}\to \mathbb{R}^{2(n-k)},\;   (q_1,\cdots,q_n, p_1,\cdots, p_n)\mapsto (q_{k+1},\cdots,q_n, p_{k+1},\cdots, p_n),\\
&&\hspace{-6mm}\Pi_2:\mathbb{R}^{2n}\to \mathbb{R}^{2k},\;   (q_1,\cdots,q_n, p_1,\cdots, p_n)\mapsto (q_1,\cdots,q_k,p_1,\cdots,p_k).
\end{eqnarray*}
Then $\Pi_1(V_0^{n,k})=L_0^{n-k}:=\{(u_1,\cdots,u_{n-k}, 0,\cdots,0)\in \mathbb{R}^{2(n-k)}\}$,  $\Pi_2|_{V^{n,k}_1}$
is a linear isomorphism to  $\mathbb{R}^{2k}$ and
\begin{equation}\label{e:proj}
\Pi_1\circ J_{2n}=J_{2(n-k)}\circ\Pi_1,\qquad \Pi_2\circ J_{2n}=J_{2k}\circ\Pi_2.
\end{equation}

For any given $x\in L^2([0,1],\mathbb{R}^{2n})$ we write
\begin{equation}\label{e:decomp}
x(t)=(q_1(t),\cdots,q_n(t),p_1(t),\cdots,p_n(t))=\hat{x}(t)+ \check{x}(t),
\end{equation}
where
\begin{eqnarray*}
&&\hat{x}(t)=(0,\cdots,0,q_{k+1}(t),\cdots,q_n(t),0,\cdots,0, p_{k+1}(t),\cdots,p_n(t))\in \mathbb{R}^{2n},\\
&&\check{x}(t)=(q_1(t),\cdots,q_k(t),0,\cdots,0,p_1(t),\cdots,p_k(t),0,\cdots,0)\in \mathbb{R}^{2n}.
\end{eqnarray*}
Then we have elements in $L^2([0,1],\mathbb{R}^{2(n-k)})$ and $L^2([0,1],\mathbb{R}^{2k})$, respectively,
\begin{eqnarray}\label{e:proj1}
&&\hspace{-6mm}[0,1]\ni t\mapsto\Pi_1\circ\hat{x}(t)=(q_{k+1}(t),\cdots,q_n(t),p_{k+1}(t),\cdots,p_n(t))\in \mathbb{R}^{2(n-k)},\nonumber\\
&&\\
&&\hspace{-6mm}[0,1]\ni t\mapsto\Pi_2\circ\check{x}(t)=(q_{1}(t),\cdots,q_k(t),p_{1}(t),\cdots,p_k(t))\in \mathbb{R}^{2k}.\label{e:proj2}
\end{eqnarray}
Applying Corollary~\ref{cor:eigenvalues} to $\Pi_1\circ\hat{x}$ and $L_0^{n-k}$ we deduce that
$\Pi_1\circ\hat{x}$ has the following expansion in $L^2([0,1],\mathbb{R}^{2(n-k)})$
\begin{equation}\label{e:expand1}
\Pi_1\circ\hat{x}=\sum_{m\in\mathbb{Z}}e^{m\pi tJ_{2(n-k)}}\hat{a}_m,
\end{equation}
where $\hat{a}_m=(q_{k+1,m},\cdots,q_{n,m},0,\cdots,0)\in L_0^{n-k}$ for all $m\in\mathbb{Z}$, satisfy
$\sum_{m\in\mathbb{Z}}|\hat{a}_m|^2<\infty$.
Moreover, we can represent $\Pi_2\circ\check{x}$ by its Fourier expansions in $L^2([0,1],\mathbb{R}^{2k})$
\begin{equation}\label{e:expand2}
\Pi_2\circ\check{x}=\sum_{m\in\mathbb{Z}}e^{2m\pi tJ_{2k}}\check{b}_m,
\end{equation}
where $\check{b}_m=(q_{1,m},\cdots,q_{k,m}, p_{1,m},\cdots,p_{k,m})\in \mathbb{R}^{2k}$ for all $m\in\mathbb{Z}$, satisfy
$\sum_{m\in\mathbb{Z}}|\check{b}_m|^2<\infty$.
For every $m\in\mathbb{Z}$, we define $a_m, b_m\in \mathbb{R}^{2m}$ by
\begin{eqnarray*}
&&a_m=(0,\cdots,0,q_{k+1,m},\cdots,q_{n,m},0,\cdots,0),\\
&&b_m=(q_{1,m},\cdots,q_{k,m},0,\cdots,0,p_{1,m},\cdots,p_{k,m},0,\cdots,0).
\end{eqnarray*}
Then $a_m\in V_0^{n,k}, b_m\in V^{n,k}_1$ satisfy $\Pi_1({a}_m)=\hat{a}_m$, $\Pi_2({b}_m)=\check{b}_m$ and
\begin{equation}\label{e:expand3}
|a_m|=|\hat{a}_m|\quad\hbox{and}\quad |b_m|=|\check{b}_m|.
\end{equation}
By (\ref{e:proj}) and (\ref{e:expand1}) (resp. (\ref{e:expand2})) it is not hard to check that
\begin{equation}\label{e:expand4}
\hat{x}=\sum_{m\in\mathbb{Z}}e^{m\pi tJ_{2n}}{a}_m\quad\hbox{and}\quad
\check{x}=\sum_{m\in\mathbb{Z}}e^{2m\pi tJ_{2n}}{b}_m.
\end{equation}
These and (\ref{e:decomp}) lead to the desired claim.
\end{proof}

For any real $s\ge0$ we follow \cite[Definition~3.6]{LiRi13} to define
\begin{eqnarray*}
H^s_{n,k}=\Big\{x\in L^2([0,1],\mathbb{R}^{2n})\,\Big| &&x\stackrel{L^2}{=}\sum_{m\in\mathbb{Z}}e^{m\pi tJ_{2n}}a_m+\sum_{m\in\mathbb{Z}}e^{2m\pi tJ_{2n}}b_m\\
&&a_m\in V^{n,k}_0, \quad b_m\in V^{n,k}_1,\\
&&\sum_{m\in\mathbb{Z}}|m|^{2s}(|a_m|^2+|b_m|^2)<\infty.
\Big\}
\end{eqnarray*}

There are some standard results from \cite{LiRi13} (or \cite{HoZe94}).
\begin{lemma}[\hbox{\cite[Lemma 3.8 and 3.9]{LiRi13}}]
For each $s\ge 0$, $H^s_{n,k}$ is a Hilbert space with the inner product
$$
\langle\phi,\psi\rangle_{s,n,k}=\langle a_0,a_0'\rangle+\langle b_0,b_0'\rangle
+\pi\sum_{m\ne 0}(|m|^{2s}\langle a_m, a_m'\rangle+|2m|^{2s}\langle b_m, b_m'\rangle),
$$
where $\phi$, $\psi\in H^s_{n,k}$ are expanded respectively as
\begin{eqnarray*}
&&\phi\stackrel{L^2}{=}\sum_{m\in\mathbb{Z}}e^{m\pi tJ_{2n}}a_m+\sum_{m\in\mathbb{Z}}e^{2m\pi tJ_{2n}}b_m\quad\hbox{and}\quad\\
&&\psi\stackrel{L^2}{=}\sum_{m\in\mathbb{Z}}e^{m\pi tJ_{2n}}a_m'+\sum_{m\in\mathbb{Z}}e^{2m\pi tJ_{2n}}b_m'
\end{eqnarray*}
with
$$
a_m, a_m'\in V^{n,k}_0\quad\hbox{and}\quad b_m, b_m'\in V^{n,k}_1 \quad\hbox{for all}\quad m\in\mathbb{Z}.
$$
Furthermore, if $s>t$, then the inclusion $j:H^s_{n,k}\hookrightarrow H^t_{n,k}$
and its Hilbert adjoint $j^\ast:H^t_{n,k}\rightarrow H^s_{n,k}$ are compact.
\end{lemma}

Let $\|\cdot\|_{s,n,k}$ denote the norm induced by $\langle\cdot,\cdot\rangle_{s,n,k}$.
For $r\in\mathbb{N}$ or $r=\infty$ let
$$
C^{r}_{n,k}([0,1],\mathbb{R}^{2n})
$$
denote the space of $C^r$ maps $x:[0,1]\to\mathbb{R}^{2n}$ such that $x(i)\in\mathbb{R}^{n,k}$, $i=0,1$, and $x(1)\sim x(0)$,
 where $\sim$ is the leaf relation on $\mathbb{R}^{n,k}$.

\begin{lemma}[\hbox{\cite[Lemma 3.10]{LiRi13}}]
If $x\in H^s_{n,k}$ for $s>1/2+r$ where $r$ is an integer, then $x\in C^{r}_{n,k}([0,1],\mathbb{R}^{2n})$.
\end{lemma}

\begin{lemma}[\hbox{\cite[Lemma 3.11]{LiRi13}}]
$j^\ast(L^2)\subset H^1_{n,k}$ and $\|j^\ast(y)\|_{1,n,k}\le \|y\|_{L^2}$.
\end{lemma}

Let
\begin{equation}\label{e:space}
E=H^{1/2}_{n,k}\quad\hbox{and}\quad \|\cdot\|_E:=\|\cdot\|_{1/2,n,k}.
\end{equation}
There is an orthogonal decomposition $E=E^+\oplus E^0\oplus E^-$, where $E^0=\mathbb{R}^{n,k}$ and
\begin{eqnarray*}
&&E^-=\Big\{x\in H^{1/2}_{n,k}\,\Big|\,x\stackrel{L^2}{=}\sum_{m<0}e^{m\pi tJ}a_m+\sum_{m<0}e^{2m\pi tJ}b_m\Big\},\\
&&E^+=\Big\{x\in H^{1/2}_{n,k}\,\Big|\,x\stackrel{L^2}{=}\sum_{m>0}e^{m\pi tJ}a_m+\sum_{m>0}e^{2m\pi tJ}b_m\Big\}.
\end{eqnarray*}
Let $P^+$, $P^0$ and $P^-$ be the orthogonal projections to $E^+$, $E^0$ and $E^-$
respectively. For $x\in E$ we write
$x^+=P^+x$, $x^0=P^0x$ and $x^-=P^-x$.
Define functional
$$
\mathfrak{a}:E\rightarrow\mathbb{R},\quad x\mapsto \frac{1}{2}(\|x^+\|^2_E-\|x^-\|^2_E).
$$
Then there holds
$$
\mathfrak{a}(x)=\frac{1}{2}\int_0^1\langle -J_{2n}\dot{x},x\rangle dt,\quad\forall x\in C^1_{n,k}([0,1],\mathbb{R}^{2n}).
$$
(See \cite{LiRi13}.) The functional $\mathfrak{a}$ is differentiable with gradient
$\nabla \mathfrak{a}(x)=x^+-x^-$.

Suppose that $H:\mathbb{R}^{2n}\to\mathbb{R}$ is a $C^{1,1}$ Hamiltonian and that
$\nabla H$ is globally Lipschitz. (For example,
this is satisfied for $C^2$ Hamiltonian  $H:\mathbb{R}^{2n}\to\mathbb{R}$
with bounded second derivative, i.e., $|H_{zz}(z)|\leq C$ for some constant $C>0$ and for all $z\in\mathbb{R}^{2n}$.)
Then there exist positive real numbers $C_i$, $i=1,2,3,4$, such that
\begin{equation}\label{e:functiona0}
|\nabla H(z)|\le C_1|z|+C_2,\quad |H(z)|\le C_3|z|^2+C_4
\end{equation}
for all $z\in\mathbb{R}^{2n}$.  Define functionals $\mathfrak{b}_H, \Phi_H: E\rightarrow\mathbb{R}$ by
\begin{equation}\label{e:functional}
\mathfrak{b}_H(x)=\int_0^1 H(x(t))dt\quad\hbox{and}\quad\Phi_H=\mathfrak{a}-\mathfrak{b}_H.
\end{equation}

\begin{lemma}[\hbox{\cite[Section~3.3, Lemma~4]{HoZe94}}]\label{e:gradient}
The functional $\mathfrak{b}_H$ is differentiable. Its gradient $\nabla \mathfrak{b}_H$ is compact and
satisfies a global Lipschitz condition on $E$.
\end{lemma}

\begin{lemma}\label{lem:critic}
$x\in E$ is a critical point of $\Phi_H$ if and only if $x\in C^1_{n,k}([0,1],\mathbb{R}^{2n})$ and solves
\begin{equation}\label{HamEq}
\dot{x}=J\nabla H(x).
\end{equation}
Moreover, if $H$ is of class $C^l$ ($l\ge 2$) then each critical point of $\Phi_H$ on $E$ is $C^l$.
\end{lemma}

\begin{proof}
This can be proved by modifying the proof of \cite[Section~3.3, Lemma~5]{HoZe94}.
Firstly, suppose that $x\in E$ is a critical point of $\Phi_H$.
 Then there holds
\begin{equation}\label{e:criticEq}
x^+-x^-=\nabla \mathfrak{b}_H(x)=j^\ast\nabla H(x),
\end{equation}
where  $j^\ast:L^2\to E$ is the adjoint operator of the inclusion $j:E\to L^2$
defined by $(j(x),y)_{L^2}=\langle x, j^\ast(y)\rangle_E=\langle x, j^\ast(y)\rangle_{\frac{1}{2},n,k}$
for all $x\in E$ and $y\in L^2([0,1],\mathbb{R}^{2n})$.
By Corollary~\ref{cor:eigenvalues} we write
$$
x\stackrel{L^2}{=}\sum_{m\in\mathbb{Z}}e^{m\pi tJ_{2n}}z_m
\quad\hbox{and}\quad
\nabla H(x)\stackrel{L^2}{=}\sum_{m\in\mathbb{Z}}e^{m\pi tJ_{2n}}\hat{z}_m
$$
where $z_m$, $\hat{z}_m\in V_0^{n,k}$ for odd $m$,  and
$z_m$, $\hat{z}_m\in V_0^{n,k}\oplus V_1^{n,k}$ for even $m$, and
$$
\sum_{m\in\mathbb{Z}}|m||z_m|^2<\infty,\qquad \sum_{m\in\mathbb{Z}}|\hat{z}_m|^2<\infty.
$$
Hence (\ref{e:criticEq}) is equivalent to
\begin{equation}\label{e:relation}
m\pi z_m=\hat{z}_m\quad\forall m\in\mathbb{Z}.
\end{equation}
This implies that $x\in H^1_{n,k}$ and hence $x\in C^0_{n,k}$. It follows that
$\nabla H(x)\in C^0$ (but not necessarily in $C^0_{n,k}$ since it may not satisfy the
boundary condition). Let us write
$$
\int_0^1 J_{2n}\nabla H(x(t))dt=J_{2n}d_1+d_2+c,
$$
where $d_1,\, d_2\in V_0^{n,k}$ and $c\in V_1^{n,k}$.
For $0\le t\le 1$ let
\begin{equation}\label{xi}
\xi(t):=\int_0^t \big(J_{2n}\nabla H(x(s))-J_{2n}d_1-c\big)ds.
\end{equation}
Then
$\xi(0)=0\sim d_2=\xi(1)$
and hence $\xi\in C^1_{n,k}$. Let 
$\xi(t)\stackrel{L^2}{=}\sum_{m\in\mathbb{Z}}e^{m\pi tJ_{2n}}z_m'$.
Then
$$
\nabla H(x(t))-d_1+J_{2n}c=-J_{2n}\dot{\xi}(t)\stackrel{L^2}{=}\sum_{0\ne m\in\mathbb{Z}}m\pi e^{m\pi tJ_{2n}}z_m'.
$$
By (\ref{e:relation}) we get
\begin{equation}\label{relation2}
d_1=c=0\quad\hbox{and}\quad z_m=z_m'\quad\hbox{for}\quad m\ne 0.
\end{equation}
It follows that $\xi(t)=x(t)-x(0)$
and hence $x\in C^1_{n,k}$. Clearly $x$ solves (\ref{HamEq})
due to (\ref{xi}) and the first two equalities in (\ref{relation2}).
This $x$ also sits in $C^l_{n,k}([0,1],\mathbb{R}^{2n})$ if $H$ is $C^l$ with $l\ge 2$.  

Conversely, suppose that $x\in C^1_{n,k}([0,1],\mathbb{R}^{2n})$
 solves (\ref{HamEq}). Clearly,
 (\ref{e:relation}) holds and so does (\ref{e:criticEq}). It follows that $x$ is a
critical point of $\Phi_H$ on $E$.
\end{proof}

\section{Proof of Theorem \ref{th:represention}}\label{sec:3}
\setcounter{equation}{0}

Recall that $0\le k<n$.
We may also assume $0\in D$ without loss of generality. In fact, for any
$z_0\in D\cap\mathbb{R}^{n,k}$, $x$ is a (generalized) leafwise chord on $\partial D$ for
$\mathbb{R}^{n,k}$ if and only if $x-z_0$ is a (generalized) leafwise chord on $\partial (D-z_0)$ for
$\mathbb{R}^{n,k}-z_0=\mathbb{R}^{n,k}$. Moreover $x(1)-x(0)\in V_0^{n,k}$ and (\ref{lagmultipier1}) imply
\begin{eqnarray*}
A(x-z_0)=\frac{1}{2}\int_0^1 \langle -J_{2n}\dot{x},x-z_0\rangle dt
      =A(x)+\langle J_{2n}(x(1)-x(0)), z_0\rangle
      =A(x).
\end{eqnarray*}
 By the monotonicity of $c_{\rm LR}$ there also holds
 $$
 c_{\rm LR}(D,D\cap \mathbb{R}^{n,k})=c_{\rm LR}(D-z_0,(D-z_0)\cap \mathbb{R}^{n,k}).
 $$

\subsection{Proof of (\ref{minaction}) in Theorem \ref{th:represention}}\label{sec:3.1}

 Fix a real  $p>1$. In fact, to prove  (\ref{minaction}) it suffices to consider Clarke duality in the case $p=2$. Here we consider more general cases $p>1$
 since it is useful in the proof of Theorem~\ref{th:Brun}.
  With  the leaf relation $\sim$ on $\mathbb{R}^{n,k}$,
consider the Banach subspace of $W^{1,p}([0,1],\mathbb{R}^{2n})$
\begin{eqnarray}\label{e:BanachSpace}
\mathscr{F}_p^{n,k}:=\left\{x\in W^{1,p}([0,1],\mathbb{R}^{2n})\,\bigg|\hspace{-4mm}
\begin{array}{ll}
&x(0), x(1)\in\mathbb{R}^{n,k},\;x(0)\sim x(1),\\
&\int_0^1x(t)dt\in J_{2n}V_0^{n,k}
\end{array}\right\}
\end{eqnarray}
and its subset
\begin{eqnarray}\label{e:Brunn.1}
\mathcal{A}_p^{n,k} =\{x\in \mathscr{F}_p^{n,k}\,|\,A(x)=1 \}.
\end{eqnarray}
 As stated below Theorem \ref{th:represention} the key point of our proof is to work in the space $\mathscr{F}_p^{n,k}$,
 which does not appear in references such as \cite{HoZe90, HoZe94, LiRi13}.
Note that $\mathcal{A}_p^{n,k}$ is a regular submanifold of $\mathscr{F}_p^{n,k}$. In fact, for any $x\in \mathscr{F}_p^{n,k}$ and
$\zeta\in T_x\mathscr{F}_p^{n,k}=\mathscr{F}_p^{n,k}$ we have
$$
dA(x)[\zeta]= \int_0^{1}\langle -J_{2n}\dot{\zeta},x\rangle dt+\frac{1}{2}\langle -J_{2n}x,\zeta\rangle|_0^1=\int_0^{1}\langle -J_{2n}\dot{\zeta},x\rangle dt
$$
because $J_{2n}x(i)\in J_{2n}\mathbb{R}^{n,k}$ is
orthogonal to $\zeta(i)\in \mathbb{R}^{n,k}$ for each $i\in\{0,1\}$ by (\ref{lagmultipier1}).
 In particular, $dA(x)[x]=2$ for any $x\in \mathcal{A}_p^{n,k}$, and thus $dA\neq 0$ on $\mathcal{A}_p^{n,k}$.

The following two lemmas are very key  for the proof of (\ref{minaction}) in Theorem~\ref{th:represention} and the proof of Theorem~\ref{th:Brun}.
 \begin{lemma}\label{lem:3.1}
 There exists a constant $\widetilde{C}_1>0$ such that
\begin{equation}\label{e:Brun.2}
\|x\|_{L^\infty}\le\widetilde{C}_1\|\dot{x}\|_{L^p},\quad\forall x\in\mathscr{F}_p^{n,k}.
\end{equation}
\end{lemma}
 \begin{proof}
 Write
  $x(t)=(q_1(t),\cdots,q_n(t),p_1(t),\cdots,p_n(t))$. Since $x(0), x(1)\in\mathbb{R}^{n,k}$,
  \begin{equation}\label{e:Brun.2.1}
p_i(0)=p_i(1)=0,\quad i=k+1,\cdots,n.
\end{equation}
 Note that
 $$
 J_{2n}V^{n,k}_0=\{(0,\cdots,0,p_{k+1},\cdots,p_n)\in\mathbb{R}^{2n}\,|\, p_i\in\mathbb{R},\;i=k+1,\cdots, n\}.
 $$
  Then $\int_0^1x(t)dt\in J_{2n}V_0^{n,k}$ leads to
  \begin{eqnarray*}
\int^1_0q_i(t) dt=0=\int^1_0p_j(t) dt,\quad i=1,\cdots,n,
\quad j=1,\cdots,k.
\end{eqnarray*}
  Hence there exist $t_i\in [0, 1]$ for $i=1,\cdots,n$, and
  $s_j\in [0, 1]$ for $j=1,\cdots,k$, such that
    \begin{eqnarray*}
q_i(t_i)=0=p_j(s_j),\quad i=1,\cdots,n,\quad j=1,\cdots,k.
\end{eqnarray*}
As usual we derive from these and (\ref{e:Brun.2.1}) that for any $t\in [0, 1]$,
\begin{eqnarray*}
&&|q_i(t)|=|q_i(t)-q_i(t_i)|=|\int^t_{t_i}\dot q_i(\tau)d\tau|\le\|\dot q_i\|_{L^p},\quad i=1,\cdots,n,\\
&&|p_j(t)|=|p_j(t)-p_j(s_j)|=|\int^t_{s_j}\dot p_j(\tau)d\tau|\le\|\dot p_j\|_{L^p},\quad j=1,\cdots,k,\\
&&|p_j(t)|=|p_j(t)-p_j(0)|=|\int^t_{0}\dot p_j(\tau)d\tau|\le\|\dot p_j\|_{L^p},\quad j=k+1,\cdots,n.
\end{eqnarray*}
 These lead to the expected inequality  immediately.
 \end{proof}

\begin{lemma}\label{lem:3.2}
$L^p([0,1],\mathbb{R}^{2n})$ has the direct sum decomposition
\begin{equation}\label{lagmultipier}
L^p([0,1],\mathbb{R}^{2n})=\{-J_{2n}\dot\zeta\,|\,\zeta\in \mathscr{F}_p^{n,k}\}\dot{+}\mathbb{R}^{n,k}
\end{equation}
such that
\begin{equation}\label{lagmultipier+}
\int_0^1\langle {\bf a},-J_{2n}\dot{\zeta}(t)\rangle dt=0\,\quad\forall \zeta\in\mathscr{F}_p^{n,k},\;\forall {\bf a}\in\mathbb{R}^{n,k},
\end{equation}
where $\mathbb{R}^{n,k}$ may be naturally viewed as a subspace of $L^p([0,1],\mathbb{R}^{2n})$ as before.
{In particular, when $p=2$, (\ref{lagmultipier+}) means that the space decomposition in (\ref{lagmultipier}) is orthogonal.}
\end{lemma}
 \begin{proof}
Let $P_{n,k}:\mathbb{R}^{2n}\to \mathbb{R}^{n,k}$ denote  the orthogonal projection with respect to the first decomposition
in (\ref{lagmultipier1}).
For  $x\in L^p([0,1],\mathbb{R}^{2n})$ we  write
$\int_0^1 J_{2n}x(s)ds=J_{2n}{\bf a}+{\bf b}$,
where ${\bf a}\in\mathbb{R}^{n,k}$, ${\bf b}\in V^{n,k}_0$. Define
$$
\zeta(t):=\int_0^t J_{2n}(x(s)-{\bf a})ds- P_{n,k}\int^1_0\left(\int_0^\tau J_{2n}(x(s)-{\bf a})ds\right)d\tau,\quad t\in [0,1].
$$
Then $\zeta\in W^{1,p}([0,1],\mathbb{R}^{2n})$ satisfies $x=-J_{2n}\dot\zeta+{\bf a}$ and
\begin{eqnarray*}
&&\zeta(0)\in \mathbb{R}^{n,k},\quad \zeta(1)={\bf b}-P_{n,k}\int^1_0\left(\int_0^t J_{2n}(x(s)-{\bf a})ds\right)dt\in \mathbb{R}^{n,k},\\
&&\zeta(1)-\zeta(0)=\int^1_0\dot \zeta(t)dt=\int^1_0J_{2n}(x(t)-{\bf a})dt={\bf b}\in V^{n,k}_0,\\
&&\int^1_0\zeta(t)dt=\int^1_0\left(\int_0^t J_{2n}(x(s)-{\bf a})ds\right)dt\\
&&\hspace{20mm}-P_{n,k}\int^1_0\left(\int_0^t J_{2n}(x(s)-{\bf a})ds\right)dt\in J_{2n}V^{n,k}_0.
\end{eqnarray*}
{The last three lines show that} $\zeta\in \mathscr{F}_p^{n,k}$. Hence
$$
L^p([0,1],\mathbb{R}^{2n})=\{-J_{2n}\dot\zeta\,|\,\zeta\in \mathscr{F}_p^{n,k}\}+\mathbb{R}^{n,k}.
$$
It is easily seen that the orthogonal decompositions in (\ref{lagmultipier1}) imply
(\ref{lagmultipier+}).
\end{proof}

Having these two lemmas, the proof of equality (\ref{minaction}) can be completed by slightly generalizing the Clarke duality proving the existence of a closed characteristic on  a compact, $C^2$-smooth and strictly convex hypersurface in $\mathbb{R}^{2n}$ (see \cite[Sec. 1.5]{HoZe94}). Similar methods are used in \cite{JinLu1916} to prove the representation formula and a Brunn-Minkowski type inequality for generalized Hofer-Zehnder capacity.

Let $j_D$ be  the Minkowski functional of $D$,  $H_D=(j_D)^2$ and let $H_D^{\ast}$ be the Legendre transformation of $H_D$. Then
   $\partial D=H_D^{-1}(1)$, and
 there exist constants $R_1, R_2\geq 1$ such that
\begin{eqnarray}\label{e:dualestimate}
 \frac{|z|^2}{R_1}\leq H_D(z)\leq R_1|z|^2,\;\forall z\in\mathbb{R}^{2n},\;\,
\frac{|z|^2}{R_2}\leq H_D^{\ast}(z)\leq R_2|z|^2,\;\forall z\in\mathbb{R}^{2n}.
\end{eqnarray}
Recall that for  $p>1$ and $q=\frac{p}{p-1}$ there holds
\begin{eqnarray}\label{e:Brunn.0}
  \left(\frac{1}{p}j_D^p\right)^{\ast}(w)=\frac{1}{q}(h_D(w))^q
\end{eqnarray}
by \cite[(9.1)]{JinLu1916}. Then $H^\ast_D=(h_D/4)^2$, and therefore the following functional
\begin{eqnarray}\label{e:convexfunctional}
I_p:\mathscr{F}_p^{n,k}\to\mathbb{R},\;x\mapsto \int_0^1(H_D^{\ast}(-J_{2n}\dot{x}(t)))^{\frac{p}{2}}dt
\end{eqnarray}
is convex, and $C^1$ if $D$ is strictly convex and has $C^1$-smooth boundary.
The equality (\ref{minaction}) can be obtained from the case $p=2$ in the following proposition.

\begin{prop}\label{prop:Brun.0}
  For real $p>1$, $I_p|_{\mathcal{A}_p^{n,k}}$
  attains its minimum and
  $$
  (\min_{x\in\mathcal{A}_p^{n,k}}I_p)^{\frac{2}{p}}
  $$
  is the minimizer of
    $$
\{A(x)>0\,|\,x\;\hbox{is a generalized leafwise chord on $\partial D$ for $\mathbb{R}^{n,k}$}\}.
$$
\end{prop}
\begin{proof}

\noindent{\bf Step 1}.\quad\textsf {$\mu_p:=\inf_{x\in\mathcal{A}_p^{n,k}}I_p(x)$ is positive. }
Let $q=\frac{p}{p-1}$. By Lemma~\ref{lem:3.1},  we deduce
$$
2=2A(x)\le \|x\|_{L^q}\|\dot{x}\|_{L^p}\le \|x\|_{L^\infty}\|\dot{x}\|_{L^p}\le  \widetilde{C}_1\|\dot{x}\|_{L^p}^2\quad
\forall x\in\mathcal{A}_p^{n,k},
$$
which leads to for $R_2$  in (\ref{e:dualestimate})
$$
I_p(x)\ge\left(\frac{1}{R_2}\right)^{p/2}\|\dot{x}\|_{L^p}^p\ge \widetilde{C}_2:=
\left(\frac{2}{R_2\widetilde{C}_1}\right)^{\frac{p}{2}}>0.
$$
\noindent{\bf Step 2}.\quad\textsf{ There exists $u\in  \mathcal{A}_p^{n,k}$ such  that
$I_p(u)=\mu_p$, and thus $u$ is a critical point of  $I_p|_{\mathcal{A}_p^{n,k}}$.}
Let $(x_m)$ be a sequence in ${\mathcal{A}_p^{n,k}}$ such that
$\lim_{m\to\infty}I_p(x_m)=\mu_p$.
This and (\ref{e:dualestimate}) lead to
\begin{eqnarray*}
\frac{1}{R_2^{p/2}}\|\dot{x}_m\|_{L_p}^p&=&\int_0^1\left(\frac{|-J\dot{x}_m|^2}{R_2}\right)^{p/2}dt\\
&\le&\int_0^1\left(H_D^{\ast}(-J\dot{x}_m)\right)^{p/2} dt=I_p(x_m)\le C
\end{eqnarray*}
for some constant $C>0$. Combining with (\ref{e:Brun.2}) we get
$$
\|x_m\|_{L_p}\le\|x_m\|_{L^\infty}\le\widetilde{C}_1\|\dot{x}_m\|_{L^p}\le\widetilde{C}_1C^{1/p}R_2^{1/2}.
$$
 Therefore $(x_m)$ is a bounded sequence in $W^{1,p}([0,1],\mathbb{R}^{2n})$.
 Since $p>1$,  $W^{1,p}([0,1])$ is reflexive and the embedding
 $W^{1,p}([0,1],\mathbb{R}^{2n})\hookrightarrow C^0([0,1],\mathbb{R}^{2n})$ is compact.
 Passing to a subsequence if necessary we can assume  that
 \begin{eqnarray}\label{e:weakC}
&&\hbox{$x_m$  converges weakly to $u$ in $W^{1,p}([0,1],\mathbb{R}^{2n})$},\\
&&\hbox{$\underline{x_m}$  converges  to $\underline{u}$ in  $C^0([0,1],\mathbb{R}^{2n})$.} \label{e:uniformC}
\end{eqnarray}
  (See the beginning of Section~\ref{sec:2} for related notations.)
  It follows that
$$
\underline{u}(0), \underline{u}(1)\in\mathbb{R}^{n,k},\quad \underline{u}(0)\sim \underline{u}(1),\quad
\int_0^1 u(t)dt=\int_0^1 \underline{u}(t)dt\in J_{2n}V_0^{n,k}.
$$
Moreover,
\begin{eqnarray*}
2A(u)&=&\int_0^1\langle-J_{2n}\dot{u},u\rangle dt\\
&=&\int_0^1 \langle-J_{2n}(\dot{u}-\dot{x}_m),u\rangle dt
+\int_0^1\langle-J_{2n}\dot{x}_m,x_m\rangle dt\\
&&+\int_0^1\langle-J_{2n}\dot{x}_m,u-x_m\rangle dt
\rightarrow 2
\end{eqnarray*}
because $\int_0^1 \langle-J_{2n}(\dot{u}-\dot{x}_m),u\rangle dt\to 0$ by (\ref{e:weakC}) and
$$
\int_0^1\langle-J_{2n}\dot{x}_m,u-x_m\rangle dt=\int_0^1\langle-J_{2n}\dot{x}_m,\underline{u}-\underline{x_m}\rangle dt\rightarrow 0
$$
 by (\ref{e:uniformC}). Therefore we deduce that
$u\in\mathcal{A}_p^{n,k}$.
Consider the functional
$$
\hat{I}_p: L^p([0,1],\mathbb{R}^{2n})\to\mathbb{R}, \;x\mapsto \int_0^1(H_D^{\ast}(u(t))^{\frac{p}{2}}dt.
$$
 Then (\ref{e:convexfunctional}) implies $I_p(x)=\hat{I}_p(-J\dot{x})$ for any $x\in\mathscr{F}_p^{n,k}$.
Since $(H^\ast_D)^{\frac{p}{2}}=(h_D/2)^p$ by (\ref{e:Brunn.0}),
 $\hat{I}_p$ is convex, and continuous by (\ref{e:dualestimate}).
It follows from Corollary~3 in \cite[Chapter II, \S3]{Ek90} that $\hat{I}_p$
 has nonempty subdifferential $\partial\hat{I}_p(v)$ at each point $v\in L^p([0,1],\mathbb{R}^{2n})$,
$$
\partial\hat{I}_p(v)=\{w\in L^q([0,1],\mathbb{R}^{2n})\,|\, w(t)\in \frac{p}{2}(H_D^{\ast}(v(t))^{\frac{p}{2}-1}\partial H^\ast_D(v(t))\;\hbox{a.e. on}\;[0,1]\}.
$$
Hence there exists $w\in \partial\hat{I}_p(-J\dot{u})\subset L^q([0,1],\mathbb{R}^{2n})$ such that
\begin{eqnarray*}
I_p(u)-I_p(x_m)&=&\hat{I}_p(-J\dot{u})-\hat{I}_p(-J\dot{x}_m)\\
&\le& \int_0^1\langle w(t),-J(\dot{u}(t)-\dot{x}_m(t))\rangle dt\to 0
\end{eqnarray*}
by  (\ref{e:weakC}). Therefore
\begin{equation}\label{e:Ipu}
I_p(u)=\lim_{m\to\infty}I_p(x_m)=\mu_p.
\end{equation}

\noindent{\bf Step 3}.\quad\textsf{
There exists a generalized leafwise chord on $\partial D$ for $\mathbb{R}^{n,k}$,
${x}^\ast:[0, 1]\rightarrow \partial D$,  such that $A({x}^\ast)=(\mu_p)^{\frac{2}{p}}$.}

 Since $u$ is the minimizer  of  $I_p|_{\mathcal{A}_p^{n,k}}$,
  applying Lagrangian multiplier theorem (\cite[Theorem~6.1.1]{Cl83}) we get some $\lambda_p\in\mathbb{R}$ such that
$0\in\partial (I_p+\lambda_p A)(u)=\partial I_p(u)+\lambda_p A'(u)$.
This means that there exists some $\rho\in L^q([0,1],\mathbb{R}^{2n})$ satisfying
\begin{equation}\label{e:Brun.3-}
\rho(t)\in\partial (H_D^{\ast})^{\frac{p}{2}}(-J_{2n}\dot{u}(t))\quad\hbox{a.e.}
\end{equation}
and
\begin{equation}\label{e:Brun.3-1}
\int_0^1\langle\rho(t)+\lambda_p u(t),-J_{2n}\dot{\zeta}(t)\rangle dt=0\quad\forall \zeta\in\mathscr{F}_p^{n,k}.
\end{equation}
Using Lemma~\ref{lem:3.2}, from (\ref{e:Brun.3-1})
 we derive that for some ${\bf a}_0\in \mathbb{R}^{n,k}$,
\begin{equation}\label{e:Brun.3}
\rho(t)+\lambda_p u(t)={\bf a}_0,\quad\hbox{a.e.}\quad
\end{equation}
Since $(H_D^{\ast})^{p/2}$ is $p$-homogeneous, by \cite[Theorem~3.1]{YangWei08} and (\ref{e:Brun.3-}) we get
$$
\langle\rho(t),-J_{2n}\dot{u}(t)\rangle=p(H_D^{\ast})^{p/2}(-J_{2n}\dot{u}(t)),\quad\hbox{a.e.}.
$$
Inserting (\ref{e:Brun.3}) into the above equality and integrating over interval $[0,1]$
we obtain
\begin{equation}\label{e:intEul}
\int_0^1\langle {\bf a}_0-\lambda_p u(t),-J_{2n}\dot{u}(t)\rangle dt=p\int_0^1(H_D^{\ast})^{p/2}(-J_{2n}\dot{u}(t))dt.
\end{equation}
Because ${\bf a}_0\in\mathbb{R}^{n,k}$ and $u\in \mathcal{A}_p^{n,k}$, we derive
\begin{eqnarray*}
&&\int_0^1\langle {\bf a}_0,-J_{2n}\dot{u}(t)\rangle dt=\langle {\bf a}_0,-J_{2n}(u(1)-u(0))\rangle=0, \\
&&\int_0^1 \langle -\lambda_p u(t),-J_{2n}\dot{u}(t)\rangle dt=-2\lambda_p A(u)=-2\lambda_p.
\end{eqnarray*}
 Hence by (\ref{e:Ipu}), (\ref{e:intEul}) becomes
\begin{equation}\label{e:Brun.3.1}
-2\lambda_p=p\mu_p.
\end{equation}

By the definition of Legendre transformation,
  for every convex function $\Psi$ we have
$$
(c\Psi)^{\ast}(x)=c\Psi^{\ast}(\frac{x}{c}),\quad\forall c>0.
$$
 This and (\ref{e:Brunn.0}) imply that
 $(H_D^{\ast})^{\frac{p}{2}}=(\frac{h_D}{2})^p$ has Legendre transformation
\begin{eqnarray}
\left(\frac{h_D^p}{2^p}\right)^{\ast}(x)
=\frac{2^p}{p}\left(\frac{h_D^p}{p}\right)^{\ast}(\frac{2^p}{p}x)
=\frac{2^p}{p}\cdot\frac{1}{q}j_D^q(\frac{2^p}{p}x)
=\frac{2^q}{qp^{q-1}}j_D^q(x).
\end{eqnarray}
Combining this with (\ref{e:Brun.3-}) and (\ref{e:Brun.3}), we get that
$$
-J_{2n}\dot{u}(t)\in\frac{2^q}{qp^{q-1}}\partial j_D^q(-\lambda_p u(t)+{\bf a}_0),\quad\hbox{a.e.}.
$$
Let $v(t):=-\lambda_p u(t)+{\bf a}_0$. Then
$v(j)\in\mathbb{R}^{n,k}$ for $j=0,1$, $v(1)\sim v(0)$ and
\begin{eqnarray}\label{e:Clarke}
-J_{2n}\dot{v}(t)\in-\lambda_p\frac{2^q}{qp^{q-1}}\partial j_D^q(v(t)),\;\hbox{a.e.}
\end{eqnarray}
{Since $j_D^q$ is  convex (so regular by \cite[Proposition 2.3.6]{Cl83})
it follows from a result by Clarke (\cite[Proposition 7.7.1]{Cl83})  and Rockafellar (\cite[Theorem~2]{Ku96}) that
(\ref{e:Clarke}) implies that $j_D^q(v(t))$ is a constant}
and
\begin{eqnarray*}
\frac{-2^{q-1}\lambda_p}{p^{q-1}}j_D^q(v(t))&=&\int^1_0\frac{-2^{q-1}\lambda_p}{p^{q-1}}j_D^q(v(t))dt\\
&=&\frac{1}{2}\int^1_0\langle-J_{2n}\dot{v}(t),
v(t)\rangle dt=\lambda_p^2A(u)=\left(\frac{p\mu_p}{2}\right)^2
\end{eqnarray*}
by the Euler formula \cite[Theorem~3.1]{YangWei08} and (\ref{e:Brun.3.1}).
Therefore
\begin{eqnarray*}
&&j_D^q(v(t))=\left(\frac{p}{2}\right)^q\mu_p\quad\hbox{and}\\
&&A(v)=\frac{1}{2}\int^1_0\langle-J_{2n}\dot{v}(t), v(t)\rangle dt=\lambda_p^2=\left(\frac{p\mu_p}{2}\right)^2.
\end{eqnarray*}
It follows that  $x^{\ast}(t):=\frac{v(t)}{j_D(v(t))}$ satisfies $j_D(x^{\ast}(t))\equiv1$,
$x^\ast(j)\in\mathbb{R}^{n,k}$, $j=0,1$, $x^\ast(1)\sim x^\ast(0)$ and
$$
-J_{2n}\dot{x}^\ast(t)\in \frac{2}{q}\mu_p^{2/p}\partial j_D^q(x^\ast(t)),\;\hbox{a.e.}
\quad\hbox{and}\quad
A(x^{\ast})=\frac{1}{j_D^2(v(t))}A(v)=\mu_p^{\frac{2}{p}}.
$$
Hence $x^\ast$ is a generalized leafwise chord  on $\partial D$ for $\mathbb{R}^{n,k}$  with action
$\mu_p^{\frac{2}{p}}$.

\noindent{\bf Step 4}.
 \quad\textsf{ For any generalized leafwise chord on $\partial D$ for $\mathbb{R}^{n,k}$ with positive action, $y:[0,T]\rightarrow \mathcal{S}$, there holds $A(y)\ge \mu_p^{\frac{2}{p}}$.}

 Since  \cite[Theorem~2.3.9]{Cl83} implies
$\partial j_D^q(x)=q(j_D(x))^{q-1}\partial j_D(x)$,
by Lemma~2 in \cite[Chap.V,\S1]{Ek90}, after reparameterizing it we may assume that
$y\in W^{1,\infty}([0,T],\mathbb{R}^{2n})$ and satisfies
\begin{equation}\label{e:step4.1}
\left.\begin{array}{ll}
&j_D(y(t))\equiv 1, \quad-J_{2n}\dot{y}(t)\in\partial j_D^q(y(t))\quad\hbox{a.e. on}\;[0, T],\\
&y(0), y(T)\in \mathbb{R}^{n,k},\quad y(0)\sim y(T).
\end{array}\right\}
\end{equation}
As above, using the Euler formula \cite[Theorem~3.1]{YangWei08} we deduce
\begin{equation}\label{e:Brun.4}
A(y)=\frac{1}{2}\int^T_0\langle-J_{2n}\dot{y}(t), y(t)\rangle dt=\frac{1}{2}\int^T_0
qj_D^q(y(t))dt=\frac{qT}{2}.
\end{equation}
Define $a:=\sqrt\frac{2}{qT}$ and
\begin{equation}\label{e:defyast}
y^{\ast}:[0,1]\rightarrow \mathbb{R}^{2n},\;  t\mapsto y^{\ast}(t)=a y(tT)- P_{n,k}\int^1_0a y(sT)ds,
\end{equation}
where  $P_{n,k}:\mathbb{R}^{2n}=J_{2n}V^{n,k}_0\oplus \mathbb{R}^{n,k}\to\mathbb{R}^{n,k}$
is the orthogonal projection as above. Then
\begin{eqnarray*}
&&\int^1_0y^{\ast}(t)dt=\int^1_0a y(tT)dt- P_{n,k}\int^1_0a y(tT)dt\in J_{2n}V^{n,k}_0,\\
&&y^{\ast}(0)=a y(0)- P_{n,k}\int^1_0a y(tT)dt\in \mathbb{R}^{n,k},\\
&&y^{\ast}(1)=a y(T)- P_{n,k}\int^1_0a y(tT)dt\in \mathbb{R}^{n,k},\\
&&y^{\ast}(1)-y^{\ast}(0)=a y(T)-a y(0)\in V^{n,k}_0.
\end{eqnarray*}
That is, $y^\ast\in\mathscr{F}_p^{n,k}$. Moreover, a direct computation shows that
\begin{eqnarray*}
A(y^\ast)&=&\frac{1}{2}\int^1_0\langle-J_{2n}\dot y^\ast(t), y^\ast(t)\rangle dt\\
&=&\frac{1}{2}\int^1_0\langle-J_{2n}\dot y^\ast(t), a y(tT)\rangle dt\\
&&\qquad\qquad-\frac{1}{2}\int^1_0\langle-J_{2n}\dot y^\ast(t),P_{n,k}\int^1_0a y(sT)ds\rangle dt\\
&=&\frac{1}{2}\int^1_0\langle-aTJ_{2n}\dot y(Tt), a y(tT)\rangle dt\\
&&\qquad\qquad-\frac{1}{2}\langle J_{2n}(y^{\ast}(1)-y^{\ast}(0)),P_{n,k}\int^1_0 a y(sT)ds\rangle\\
&=&\frac{1}{2}\int^1_0\langle-aTJ_{2n}\dot y(Tt), a y(tT)\rangle dt=1.
\end{eqnarray*}
Hence  $y^\ast$ belongs to $\mathcal{A}_p^{n,k}$.

 In the following we compute $I_p(y^\ast)$. By the second term in (\ref{e:step4.1}) and the definition of $y^\ast$ in (\ref{e:defyast}) we have
$$
-J_{2n}\frac{\dot{y}^\ast (t)}{aT}=-J_{2n}\dot{y}(tT)\in \partial j_D^q(y(tT)), \quad 0\le t\le 1,
$$
i.e.,
$$
-J_{2n}\frac{\dot{y}^\ast (t)}{qaT}\in \partial (\frac{j_D^q}{q})(y(tT)), \quad 0\le t\le 1.
$$
Using this, (\ref{e:Brunn.0}) and  the Legendre reciprocity formula (cf. \cite[Proposition~II.1.15]{Ek90})
we derive
$$
\frac{h_D^p}{p}(-J_{2n}\frac{\dot{y}^\ast (t)}{qaT})+(\frac{j_D^q}{q})(y(tT))=
\langle-J_{2n}\frac{\dot{y}^\ast (t)}{qaT},y(tT)\rangle=1
$$
as in (\ref{e:Brun.4}).
Then (\ref{e:Brunn.0}) and this equality lead to
\begin{eqnarray*}
I_p(y^\ast)&=&\int_0^1 (H_D^\ast(-J_{2n}\dot{y}^\ast))^{\frac{p}{2}}dt\\
           &=&\int_0^1(\frac{h_D^p}{2^p})(-J_{2n}\dot{y}^\ast)dt\\
           &=&\frac{p(qaT)^p}{2^p}\int_0^1 (\frac{h_D^p}{p})(-J_{2n}\frac{\dot{y}^\ast (t)}{qaT})dt\\
           &=&\frac{p(qaT)^p}{2^p}\int_0^1\left(1-\frac{1}{q}\right)dt
           =\frac{(qaT)^p}{2^p}
           \ge\mu_p.
\end{eqnarray*}
Inserting $a=\sqrt{\frac{2}{qT}}$ into the last inequality we deduce that
$A(y)={qT}/{2}\ge\mu_p^{\frac{2}{p}}$.
This and Step 3 prove Proposition~\ref{prop:Brun.0}.
\end{proof}

\begin{remark}\label{rem:Brun.1.5}
{\rm  By the final two steps in the proof of Proposition~\ref{prop:Brun.0},
for a minimizer $u$ of  $I_p|_{\mathcal{A}_p^{n,k}}$ there exists ${\bf a}_0\in \mathbb{R}^{n,k}$
such that
$$
x^\ast(t)=\left(\min I_p|_{\mathcal{A}_p^{n,k}}\right)^{1/p}u(t)+ \frac{2}{p}\left(\min I_p|_{\mathcal{A}_p^{n,k}}\right)^{(1-p)/p}{\bf a}_0
$$
gives a generalized leafwise chord on $\partial D$ for $\mathbb{R}^{n,k}$
with action
$$
\left(\min I_p|_{\mathcal{A}_p^{n,k}}\right)^{2/p}=
 \min\left\{A(x)>0\,\bigg|\,\hspace{-3mm}\begin{array}{ll}
 &\hbox{$x$ is a generalized leafwise chord}\\
 &\hbox{on $\partial D$ for $\mathbb{R}^{n,k}$}
 \end{array}\right\}.
$$
Thus this and (\ref{capacity}) proved below show that
$$
x^\ast(t)=\left(c_{\rm LR}(D, D\cap\mathbb{R}^{n,k})\right)^{1/2}u(t)+ \frac{2}{p}\left(c_{\rm LR}(D, D\cap\mathbb{R}^{n,k})\right)^{(1-p)/2}{\bf a}_0
$$
 is a $c_{\rm LR}(D, D\cap\mathbb{R}^{n,k})$ carrier.
}
\end{remark}

\subsection{Proof of (\ref{capacity}) in Theorem \ref{th:represention}}\label{sec:3.2}

    The following lemma, which is a generalization of the fact that the set of actions of closed
    characteristics on a smooth convex energy surface has no interior point in $\mathbb{R}$ (cf. \cite[\S~7.4]{Sik90}),
    plays a key role in the proof of (\ref{capacity}) in Theorem \ref{th:represention}.
     We need to improve the methods of \cite[\S~7.4]{Sik90} to prove it.

\begin{lemma}\label{nointerior}
  Let $D\subset \mathbb{R}^{2n}$ be a bounded convex domain with
boundary $\mathcal{S}=\partial D$ and containing $0$ in its interior. If $\mathcal{S}$ is of class $C^{2n+2}$, then the set
   \begin{equation}\label{eq:setsigma}
    \Sigma_\mathcal{S}:=\{ A(x)\,|\,\hbox{is a leafwise chord on \;$\mathcal{S}$\;for \;$\mathbb{R}^{n,k}$ and}\;  A(x)>0\;\}
    \end{equation}
   has no interior point  in $\mathbb{R}$.
 \end{lemma}

\begin{proof}
 Since $\mathcal{S}$ is of class $C^{2n+2}$,
 the Minkowski functional $j_D: \mathbb{R}^{2n}\rightarrow\mathbb{R}$
 is $C^{2n+2}$ in $\mathbb{R}^{2n}\setminus\{0\}$. Fix $1<\alpha<2$. Define
 $F: \mathbb{R}^{2n}\rightarrow \mathbb{R}$ by $F(z)=(j_D(z))^{\alpha}$. It
is  $C^1$ in $\mathbb{R}^{2n}$ and $C^{2n+2}$ in $\mathbb{R}^{2n}\setminus\{0\}$.

Fix an arbitrarily $\sigma\in\Sigma_{\mathcal{S}}$.
 It suffices to prove
that $\Sigma_{\mathcal{S}}\cap(\sigma-\epsilon,\sigma+\epsilon)$ has no interior point
for some sufficiently small $\epsilon>0$. Since $F(0)=0$ we may choose $0<\varepsilon_1<\varepsilon_2$ such that
    the closure of $B^{2n}({\varepsilon_2})$ is contained in $D$
      and that
       \begin{equation}\label{period0}
     \max_{z\in B^{2n}({\varepsilon_2})}F(z)< \left(\frac{2(\sigma+\epsilon)}{\alpha}\right)^{\frac{\alpha}{\alpha-2}}.
    \end{equation}
       Take a smooth function $f:\mathbb{R}^{2n}\rightarrow [0,1]$ such that
   $f|_{B_{\varepsilon_1}}=0$ and $f|_{B_{\varepsilon_2}^{c}}=1$.
      Define  $\overline{F}:\mathbb{R}^{2n}\rightarrow \mathbb{R}$ by
   $\overline{F}(z)=f(z)F(z)$ for $z\in \mathbb{R}^{2n}$.
      Then $\overline{F}\in C^{2n+2}(\mathbb{R}^{2n},\mathbb{R})$.

    Let $x:[0,T]\rightarrow \mathcal{S}$ be a leafwise chord
for $\mathbb{R}^{n,k}$ and with action $A(x)\in (\sigma-\epsilon,\sigma+\epsilon)$.
We may assume that it satisfies
\begin{equation}\label{period1}
\dot{x}=J_{2n}\nabla F(x).
\end{equation}
Then $A(x)=\frac{\alpha T}{2}$.
Since $\nabla F(\lambda z)=\lambda^{\alpha-1}\nabla F(z)\;\forall\lambda\ge 0$ and $F(x(t))\equiv 1$, we deduce that
 \begin{equation}\label{period3}
 y:[0,1]\rightarrow \mathbb{R}^{2n},\;
t\mapsto y(t)=T^{\frac{1}{\alpha-2}}x(tT)
\end{equation}
 satisfies
\begin{eqnarray*}
&&\dot{y}(t)=J_{2n}\nabla F(y(t)),\quad  y(j)\in \mathbb{R}^{n,k}, \;j=0,1,\quad \hbox{and}\\
&&  y(1)\sim y(0),\quad  F(y(t))\equiv T^{\frac{\alpha}{\alpha-2}}
\geq  \left(\frac{2(\sigma+\epsilon)}{\alpha} \right)^{\frac{\alpha}{\alpha-2}}.
\end{eqnarray*}
The last inequality and (\ref{period0}) imply  $y([0, 1])\subset (B_{\varepsilon_2})^c$. Hence $\dot{y}=J\nabla\overline{F}(y)$ (since
$\overline{F}=F$ on $(B_{\varepsilon_2})^c$) and so
 $y$ is a  critical point of the functional
$$
\Phi_{\overline{F}}:E\rightarrow \mathbb{R},\;x\mapsto \frac{1}{2}\|x^+\|_E^2-\frac{1}{2}\|x^-\|_E^2-\int_0^{1}\overline{F}(x(t)) dt.
$$
{Here we recall that $E=H^{1/2}_{n,k}$   and $\|\cdot\|_E:=\|\cdot\|_{1/2,n,k}$  (see (\ref{e:space})) .}
Moreover, a direct computation gives rise to
 \begin{eqnarray*}
 \Phi_{\overline{F}}(y)&=&\frac{1}{2}\int_0^{1}\langle-J\dot{y},y\rangle dt-\int_0^1F(y(t))dt\\
                       &=&\left(\frac{\alpha}{2}-1\right)F(y(t))
                       =\left(\frac{\alpha}{2}-1\right)T^{\frac{\alpha}{\alpha-2}}.
 \end{eqnarray*}
By Lemma~\ref{lem:critic} all critical points of $\Phi_{\overline{F}}$ sit in $C^{2n+2}_{n,k}$
and hence $\Phi_{\overline{F}}$ and $\Phi_{\overline{F}}|_{C^1_{n,k}}$
have the same critical value sets. Since $\overline{F}$ is $C^{2n+2}$,
 as in the proof of \cite[Claim~4.4]{JinLu1916} we can deduce that $\Phi_{\overline{F}}|_{C^1_{n,k}}$ is of class $C^{2n+1}$.

Let $P_0$ and $P_1$ be the orthogonal projections of $\mathbb{R}^{2n}$ to the spaces $V_0^{n,k}$
and $V_1^{2k}$ in (\ref{e:V0}) and (\ref{e:V1}), respectively. Take a smooth
$g:[0,1]\rightarrow [0,1]$
 such that $g$ equals $1$ (resp. $0$)
near $0$ (resp. $1$). Denote by $\phi^t$ the flow of $X_{\overline{F}}$.
 Since $X_{\overline{F}}$ is $C^{2n+1}$, we have a $C^{2n+1}$ map
$$
\psi:[0,1]\times\mathbb{R}^{n,k}\rightarrow \mathbb{R}^{2n},\;(t,z)\mapsto g(t)\phi^t(z)+(1-g(t))\phi^{t-1}(P_0\phi^1(z)+P_1z).
$$
For any $z\in\mathbb{R}^{n,k}$, since $\psi(0,z)=\phi^0(z)=z$ and $\psi(1,z)=P_0\phi^1(z)+P_1z$, we have
$$
\psi(1,z),\;\psi(0,z)\in\mathbb{R}^{n,k}\quad\hbox{and}\quad \psi(1,z)\sim\psi(0,z).
$$
These and \cite[Corollary~B.2]{JinLu1916} show that $\psi$ gives rise to a $C^{2n}$ map
$$
\Omega:\mathbb{R}^{n,k}\to C^1_{n,k}([0,1],\mathbb{R}^{2n}),\quad z\mapsto \psi(\cdot,z).
$$
Hence
$\Phi_{\overline{F}}\circ\Omega: \mathbb{R}^{n,k}\to\mathbb{R}$
is of class $C^{2n}$. It follows from Sard's Theorem that the critical value sets of $\Phi_{\overline{F}}\circ\Omega$ is nowhere dense
(since $\dim \mathbb{R}^{n,k}<2n$).

Let $z\in\mathbb{R}^{n,k}$ be such that  $\phi^1(z)\in\mathbb{R}^{n,k}$ and $\phi^1(z)\sim z$.
Then $P_0\phi^1(z)-P_0z=\phi^1(z)-z$ and therefore
$P_0\phi^1(z)+P_1z=\phi^1(z)$,
which implies
$\psi(t,z)=\phi^t(z)\;\forall t\in [0,1]$.

For  $y$ in (\ref{period3}) we have $z_y\in \mathbb{R}^{n,k}$ such that $y(t)=\phi^t(z_y)\;\forall t\in [0,1]$.
This implies $\phi^1(z_y)\in\mathbb{R}^{n,k}$, $\phi^1(z_y)\sim z_y$ and therefore
 $y=\psi(\cdot,z_y)=\Omega(z_y)$. Hence $z_y$ is a critical point of $\Phi_{\overline{F}}\circ\Omega$
 and $\Phi_{\overline{F}}\circ\Omega(z_y)=\Phi_{\overline{F}}(y)$. It follows that
 $$
 \Xi:=\left\{\left(\frac{\alpha}{2}-1\right)T^{\frac{\alpha}{\alpha-2}}\,\Bigg|\,\begin{array}{ll}
 & \hbox{$x:[0,T]\rightarrow \mathcal{S}$ is a leafwise chord
for $\mathbb{R}^{n,k}$ }\\
&\hbox{that satisfies (\ref{period1}) and has action}\\
&\hbox{$A(x)\in (\sigma-\epsilon,\sigma+\epsilon)$}
\end{array}\right\}.
 $$
is nowhere dense in $\mathbb{R}$ as a subset of  the critical value sets of $\Phi_{\overline{F}}\circ\Omega$.
This implies that $\{\frac{\alpha T}{2}>0\,|\, \left(\frac{\alpha}{2}-1\right)T^{\frac{\alpha}{\alpha-2}}\in\Xi\}$
 is nowhere dense in $\mathbb{R}$.
\end{proof}

Having Lemma~\ref{nointerior},
as in \cite[\S4.4]{JinLu1916}, by approximating arguments \textsf{it suffices
to prove (\ref{capacity}) for the smooth and strictly convex $D$}.
By Proposition~\ref{prop:Brun.0} there exists a leafwise chord on $\partial D$ for $\mathbb{R}^{n,k}$,
$x^\ast:[0,1]\to \partial D$, such that
\begin{equation}\label{e:leaf-chord}
   \left.   \begin{array}{ll}
&\dot{x}^\ast=J_{2n} A(x^\ast)\nabla H_D (x^\ast)\quad\hbox{and}\\
&A(x^\ast)=\min\{A(x)>0\;|\;x\;\hbox{is a leafwise chord on\;$\partial D$\;for \;$\mathbb{R}^{n,k}$}\}>0.
\end{array}\right\}
     \end{equation}
We shall prove $c_{\rm LR}(D,D\cap \mathbb{R}^{n,k})=A(x^\ast)$ in two steps below.

\noindent{\bf Step 1.}\quad \textsf{Prove $c_{\rm LR}(D,D\cap \mathbb{R}^{n,k})\ge A(x^\ast)$.}
For small $0<\epsilon,\, \delta<1/2$, pick  a smooth
 function $f: [0,1]\rightarrow\mathbb{R}$ such that
 $f(t)=0$ for $t\le\delta$, $f(t)=A(x^{\ast})-\epsilon$ for $t\ge 1-\delta$, and
    $0\leq f'(t)<A(x^{\ast})$ for $\delta<t<1-\delta$.
     Then $D\ni x\mapsto H(x):=f(H_D(x))$ belongs to $\mathcal{H}(D,D\cap\mathbb{R}^{n,k})$ and $\max H=A(x^{\ast})-\epsilon$.
      We claim that every solution
   $x:[0, T]\to D$ of the boundary value problem
   $$
   \left.
   \begin{array}{ll}
        &\dot{x}=J_{2n}\nabla H(x)=f'(H_D(x))J_{2n}\nabla H_D(x)\quad\hbox{and}\\
        &x(0), x(T)\in \mathbb{R}^{n,k},\quad  x(T)\sim x(0)
        \end{array}\right\}
     $$
with $0<T\le 1$ is constant. By contradiction let $x=x(t)$
be a nonconstant solution of it. Then
$j_D(x(t))\equiv c\in (0, 1)$  and
$f'(H_D(x(t)))\equiv a\in (0, A(x^\ast))$.
Since $\nabla H_D(\lambda z)=\lambda\nabla H_D(z)$ for all $(\lambda,z)\in\mathbb{R}_+\times\mathbb{R}^{2n}$, $y(t):=\frac{1}{c}x(t)$
sits in $\partial D$ and satisfies
\begin{equation}\label{e:action-capacity7}
        \dot{y}=aJ_{2n}\nabla H_D(y),\quad y(0), y(T)\in \mathbb{R}^{n,k},\quad  y(T)\sim y(0).
     \end{equation}
Since $H_D$ is homogeneous of degree $2$,
there holds $\langle \nabla H_D(z), z\rangle=2H_D(z)=2$ for any $z\in\partial D$.
 (\ref{e:action-capacity7}) leads to
$0<A(y)=aT\le a<A(x^\ast)$,
which contradicts (\ref{minaction}). Hence
$H$ is admissible and so $c_{\rm LR}(D,D\cap \mathbb{R}^{n,k})\ge A(x^\ast)-\epsilon$.
Letting $\epsilon\to 0$ the expected inequality is proved.

\noindent{\bf Step 2.}\quad \textsf{Prove}
$c_{\rm LR}(D,D\cap\mathbb{R}^{n,k})\le A(x^{\ast})$.
Let $H\in \mathcal{H}(D, D\cap\mathbb{R}^{n,k})$ with $m(H)>A(x^\ast)$.
We want to prove that the boundary value problem
   \begin{equation}
        \dot{x}(t)=J_{2n}\nabla H(x(t))\;\forall t\in [0,1],\quad x(0), x(1)\in \mathbb{R}^{n,k},\quad
     x(1)\sim x(0)
     \end{equation}
   possesses a nonconstant solution $x:[0,1]\to  D$.
(Then $0\le T_x\le 1$ and hence $H$ is not admissible.)
By Lemma~\ref{nointerior} there exists  $\epsilon>0$ such that $m(H)>A(x^\ast)+\epsilon$ and
$A(x^\ast)+\epsilon\notin \Sigma_\mathcal{S}$.
Hence the boundary value problem for $x:[0,1]\to\mathbb{R}^{2n}$,
   \begin{equation}
        \dot{x}=(A(x^\ast)+\epsilon)J_{2n}\nabla H_D(x),\quad x(0), x(1)\in \mathbb{R}^{n,k},\quad
     x(1)\sim x(0),
     \end{equation}
{admits only the trivial solution $x\equiv 0$.  Otherwise, we have
$x(t)\ne 0\;\forall t\in [0, 1]$ as above. Thus after
 multiplying $x(t)$ by a suitable positive number we may assume that
  $x([0,1])\subset\mathcal{S}=\partial D$, and so $\Sigma_{\mathcal{S}}\ni A(x)=A(x^\ast)+\epsilon$,
  which is a contradiction.}

{Fix a number $\delta>0$ and choose a}  $f\in C^\infty(\mathbb{R}, \mathbb{R})$ such that
   \begin{eqnarray*}
     &f(t)\ge (A(x^\ast)+\epsilon)t,& \,\,\,   t\ge 1, \\
     &f(t)=(A(x^\ast)+\epsilon)t,& \,\,\,   t\;\hbox{large}, \\
     & f(t)=m(H),& \,\,\, 1\le t\le 1+\delta, \\
    & 0\leq f'(t)\le A(x^{\ast})+\epsilon,& \,\,\, t>1+\delta.
   \end{eqnarray*}
 Define $\overline{H}:\mathbb{R}^{2n}\to\mathbb{R}$ by
$\overline{H}(z)=f(H_D(z))$, and $\overline{H}|_D=H$.
Let $\Phi_{\overline{H}}$ be the functional associated to $\overline{H}$ as in (\ref{e:functional}).
Repeating the proof of \cite[Lemma 4]{HoZe90} yields

\begin{lemma}\label{lem:positive}
   Assume that $x:[0,1]\to\mathbb{R}^{2n}$ is a solution of
  $$
  \dot{x}(t)=X_{\overline{H}}(x(t)),\quad x(0), x(1)\in \mathbb{R}^{n,k},\quad
     x(1)\sim x(0)
     $$
  with  $\Phi_{\overline{H}}(x)>0$.
   Then it is nonconstant, sits in $D$ completely, and thus
   is a solution of  $\dot{x}=X_H(x)$ on $D$.
\end{lemma}

It remains to prove that there exists a critical point $x$ of
$\Phi_{\overline{H}}$ on $E$ with $\Phi_{\overline{H}}(x)>0$.

\begin{lemma}
 If  a sequence $(x_m)\subset E$ satisfies that $\nabla \Phi_{\overline{H}}(x_m)\rightarrow 0$ in $E$,
 then it has a convergent subsequence in $E$.
   In particular, $\Phi_{\overline{H}}$ satisfies the Palais-Smale condition.
\end{lemma}

\begin{proof}
Note that $\nabla \Phi_{\overline{H}}(x_m)=x_m^+-x_m^--\nabla\mathfrak{b}_{\overline{H}}(x_m)$
and that $\nabla\mathfrak{b}_{\overline{H}}$ is compact by Lemma~\ref{e:gradient}.
By these and the compactness of the orthogonal projection $P^0:E\to E^0=\mathbb{R}^{n,k}$
we only need to prove that  $(x_m)$  has a bounded subsequence in $E$.
 Otherwise,  after passing to a subsequence (if necessary), we may assume that
$\lim_{m\to\infty}\|x_m\|_E=\infty$.
Let $y_m=x_m/\|x_m\|_E$. Then $\|y_m\|_E=1$ and for $j^\ast$ in (\ref{e:criticEq}) it holds that
  \begin{equation}\label{e:limit-y}
 y_m^+-y_m^--j^{\ast}\left(\frac{ \nabla \overline{H}(x_m)}{\|x_m\|_E}\right)\rightarrow 0\quad\hbox{in}\quad E.
  \end{equation}
  By the construction of $\overline{H}$, $\overline{H}_{zz}$ is bounded
  and so $\overline{H}$ satisfies (\ref{e:functiona0}). Then
  $\left(\nabla \overline{H}(x_m)/\|x_m\|_E\right)$
  is bounded in $L^2$ and hence $(y_m)$ has a convergent subsequence in $E$.
  Without loss of generality, we assume $y_m\to y$ in $E$. Then $\|y\|_E=1$. Since
  $\overline{H}(z)=Q(z):=(A(x^\ast)+\epsilon)H_D(z)$
   for $|z|$  sufficiently large, which implies that
   $|\nabla\overline{H}(z)-\nabla Q(z)|$ is bounded on $\mathbb{R}^{2n}$,
   as in the proof of Lemma~6 of \cite[page 89]{HoZe94}  we deduce
 $$
 \left\|\frac{\nabla \overline{H}(x_m)}{\|x_m\|_{E}}-\nabla Q(y)\right\|_{L^2}\rightarrow 0.
 $$
 This and (\ref{e:limit-y}) yield
 $y^+-y^--j^{\ast}\nabla Q(y)=0$.
  By Lemma \ref{lem:critic}  $y\in C^\infty_{n,k}([0,1],\mathbb{R}^{2n})$ and solves
 $
 \dot{y}=J(A(x^\ast)+\epsilon)\nabla H_D(y).
 $
 If $j_D(y(t))\ne 0$ then by multiplying a constant we get a leafwise chord
 on $\partial D$ for $\mathbb{R}^{n,k}$ and
 with action $A(y)=A(x^\ast)+\epsilon$. However, we have assumed $A(x^\ast)+\epsilon\notin\Sigma_{\mathcal{S}}$.
 Hence $y(t)=0$ for all $t$ and we get a contradiction since $\|y\|_E=1$. Therefore $(x_m)$ has a bounded subsequence.
\end{proof}

Since $A(x^\ast)>0$,  the projection $x^{\ast+}$ of $x^\ast$ to $E^+$,
does not vanish.   Following \cite{HoZe90}
 we define for $s>0$ and $\tau>0$,
 \begin{eqnarray*}
&& W_s:= E^-\oplus E^0\oplus sx^{\ast+},\\
 &&\Sigma_\tau:=\{x^-+x^0+sx^{\ast+}\,|\,0\le s\le\tau,\;\|x^-+x^0\|\le\tau\}.
 \end{eqnarray*}
Let $\partial\Sigma_\tau$ denote the boundary of $\Sigma_\tau$ in $E^-\oplus E^0\oplus\mathbb{R}x^{\ast+}$. Then
 $$
  \partial\Sigma_\tau=\{x=x^-+x^0 +sx^{\ast+}\in\Sigma_\tau\,|\,\|x^-+x^0\|_{E}=\tau \;\text{or}\; s=0 \;\text{or} \; s=\tau \}.
$$
Repeating the proofs of Lemmas~5 and 6 in \cite{HoZe90} lead to

\begin{lemma}\label{lem:HZ5}
  There exists a constant $C>0$ such that for any $s\ge 0$,
  $$
  \Phi_{\overline{H}}(x)\le-\epsilon\int^1_0H_D(x(t))dt+C,\quad\forall x\in W_s.
  $$
\end{lemma}

\begin{lemma}\label{lem:HZ6}
   $\Phi_{\overline{H}}|_{\partial\Sigma_\tau}\le 0$ if $\tau>0$ is sufficiently large.
\end{lemma}

As in the proof of Lemma~9 of \cite[\S3.4]{HoZe94}, we can obtain

\begin{lemma}\label{lem:HZ7}
 For $z_0\in \mathbb{R}^{n,k}\cap H^{-1}(0)$ and $\Gamma_\alpha:=\{z_0+x\,|\,x\in E^+\;\hbox{and}\;\|x\|_E=\alpha\}$
there exist constants $\alpha>0$ and $\beta>0$ such that
$\Phi_{\overline{H}}|_{\Gamma_\alpha}\geq\beta>0$.
  \end{lemma}

Let $\phi^t$ be the negative gradient flow of $\Phi_{\overline{H}}$.
As in the proofs of Lemma~7 of \cite[\S3.3]{HoZe94} and Lemma~10 of \cite[\S3.4]{HoZe94} respectively, we have also
the following two lemmas.
\begin{lemma}
$\phi^t$ has the form
$\phi^t(x)=e^t x^-+x^0+e^{-t}x^++K(t,x)$,
where $K:\mathbb{R}\times E\to E$ is continuous and compact.
\end{lemma}

 \begin{lemma}\label{positive+}
$\phi^t(\Sigma_\tau)\cap\Gamma_\alpha\neq \emptyset, \,\forall t\geq 0$.
 \end{lemma}

Define $\mathcal{F}:=\{\phi^t(\Sigma_\tau)\,|\, t\geq 0\}$ and
$$
c(\Phi_{\overline{H}}, \mathcal{F}):=\inf_{t\geq 0}\sup_{x\in \phi^t(\Sigma_\tau)}\Phi_{\overline{H}}(x).
$$
It follows from Lemmas~\ref{lem:HZ7} and \ref{positive+} that
 $$
 0<\beta\leq \inf _{x\in\Gamma_\alpha}\Phi_{\overline{H}}(x)\leq \sup_{x\in \phi^t(\Sigma_\tau)}\Phi_{\overline{H}}(x)\;\;
\forall t\geq 0,
$$
and hence $c(\Phi_{\overline{H}}, \mathcal{F})\geq\beta>0$.
On the other hand,  since $\Sigma_\tau$ is bounded, using the fact that $\Phi_{\overline{H}}$ maps  bounded sets into bounded sets we arrive at
$$
c(\Phi_{\overline{H}}, \mathcal{F})\leq \sup_{x\in\Sigma_\tau}\Phi_{\overline{H}}(x)<\infty.
$$
Thus the Minimax Lemma on \cite[page 79]{HoZe94} yields a critical point $x$ of $\Phi_{\overline{H}}$
with $\Phi_{\overline{H}}(x)=c(\Phi_{\overline{H}}, \mathcal{F})>0$.

Now Lemmas~\ref{lem:positive} and  \ref{lem:critic} together give the proof of (\ref{capacity}) in the case
that $D$ contains $0$ and is bounded,  smooth and strictly convex.

\section{Proofs of Theorem~\ref{th:represention*}
and Corollaries~\ref{cor:ellipsoid}--\ref{cor:polyDisk}}\label{sec:4}
\setcounter{equation}{0}

\subsection{Proof of Theorem~\ref{th:represention*}}\label{sec:4.1}

By the arguments above Section~\ref{sec:3.1} we may assume $0\in D$ in the following.

We first prove  (\ref{e:add1}) of Theorem~\ref{th:represention*}.
Since $D\cap \mathbb{R}^{n,k}\ne\emptyset$ implies  $D\cap \mathbb{R}^{n,k+1}\ne\emptyset$,
{by Theorem~\ref{th:represention}} and Proposition~\ref{prop:Brun.0} we obtain
\begin{equation}\label{e:ad3.2-}
c_{\rm LR}(D,D\cap \mathbb{R}^{n,k+1})=\min_{x\in\mathcal{A}_2^{n,k+1}}I_2
=\min_{x\in \mathcal{A}_2^{n,k+1}}\int_0^1H_D^\ast(-J_{2n}\dot{x})dt.
\end{equation}
{Let  $x\in\mathcal{A}^{n,k+1}_2$  such that}
\begin{equation}\label{e:ad3.2}
c_{\rm LR}(D,D\cap \mathbb{R}^{n,k+1})=\int_0^1H_D^\ast(-J_{2n}\dot{x})dt.
\end{equation}
Write
$x(t)=(q_1(t),\cdots,q_n(t),p_1(t),\cdots,p_n(t))$.
Then $x(0)$, $x(1)\in\mathbb{R}^{n,k+1}$ and $x(0)\sim x(1)$ imply that
\begin{eqnarray}\label{e:ad3.3}
&&\hspace{-4mm}p_j(0)=p_j(1)=0,\quad\forall j>k+1,\\
&&\hspace{-4mm}q_i(0)=q_i(1)\; \forall 1\le i\le k+1,\qquad p_j(0)=p_j(1)\;\forall 1\le j\le k+1.\label{e:ad3.4}
\end{eqnarray}
Define
\begin{equation}\label{e:ad3.5}
c:=(0,\cdots,0,p_{k+1}(0),0,\cdots,0)\in\mathbb{R}^{2n}
\end{equation}
 with the $(k+n+1)$-th coordinate equal to $p_{k+1}(0)$ and others $0$, and
 \begin{equation}\label{e:ad3.6}
 y(t):=x(t)-c,\quad\forall 0\le t\le 1.
 \end{equation}
From this, equalities (\ref{e:ad3.3}) and (\ref{e:ad3.4})-(\ref{e:ad3.5}) we get
\begin{equation}\label{e:ad3.7}
y(0),y(1)\in\mathbb{R}^{n,k} \quad\hbox{and}\quad y(0)\sim y(1).
\end{equation}
Moreover, since $\int_0^1x(t)dt\in J_{2n}V_0^{n,k+1}\subset J_{2n}V_0^{n,k}$ and $c\in J_{2n}V_0^{n,k}$, it holds that
$$
\int_0^1 y(t)dt=\int_0^1x(t)dt-c\in J_{2n}V_0^{n,k}.
$$
This and  (\ref{e:ad3.7}) show that $y\in\mathscr{F}_2^{n,k}$.

Note that $J_{2n}(x(1)-x(0))\in J_{2n}V_0^{n,k+1}$ and $c\in\mathbb{R}^{n,k+1}$. We have $\langle J_{2n}(x(1)-x(0)),c\rangle=0$ and hence
\begin{eqnarray*}
A(y)=\frac{1}{2}\int_0^1\langle-J_{2n}\dot{y},y\rangle dt
    &=&\frac{1}{2}\int_0^1\langle-J_{2n}\dot{x},x-c\rangle dt\\
    &=&A(x)+\frac{1}{2}\langle J_{2n}(x(1)-x(0)),c\rangle =1.
\end{eqnarray*}
Then $y\in\mathcal{A}^{n,k}_2$. Using 
Theorem \ref{th:represention} and Proposition~\ref{prop:Brun.0} again, we arrive at
\begin{eqnarray*}
c_{\rm LR}(D,D\cap \mathbb{R}^{n,k})&=&\min_{z\in \mathcal{A}_2^{n,k}}\int_0^1H_D^\ast(-J_{2n}\dot{z})dt\\
&\le& \int_0^1H_D^\ast(-J_{2n}\dot{y})dt\\
&=&\int_0^1H_D^\ast(-J_{2n}\dot{x})dt
=c_{\rm LR}(D,D\cap \mathbb{R}^{n,k+1}).
\end{eqnarray*}
Here the second and third equalities come from (\ref{e:ad3.6}) and (\ref{e:ad3.2}), respectively.
(\ref{e:add1}) is proved.

Next let us prove (\ref{e:add2}) of Theorem~\ref{th:represention*}.
Let
$$
\mathcal{E}=\{x\in W^{1,2}([0,1],\mathbb{R}^{2n})\,\Big|\,x(0)=x(1),\; \int_{0}^1x(t)dt=0,\; A(x)=1\}.
$$
By \cite[Proposition~2.1]{AAO08} (see also \cite[\S2.12]{MoZe} or \cite[\S1.5]{HoZe94}),  we have
$$
c_{\rm HZ}(D)=\min_{{z}\in \mathcal{E}}\int_0^1H_D^\ast(-J_{2n}\dot{{z}})dt.
$$
Suppose that $x\in \mathcal{E}$ satisfies
$$
c_{\rm HZ}(D)=\int_0^1H_D^\ast(-J_{2n}\dot{x})dt.
$$
Write $x(t)=(q_1(t),\cdots,q_n(t),p_1(t),\cdots,p_n(t))$. Then $x\in \mathcal{E}$ implies that $x(0)=x(1)$.
In particular, $p_n(0)=p_n(1)$. Let
$c:=(0,\cdots,0,p_n(0))\in\mathbb{R}^{2n}$
with the $2n$th coordinate equal to $p_n(0)$ and others $0$. It sits in
$J_{2n}V^{n,n-1}_0$ by (\ref{lagmultipier1}). Define
$y(t):=x(t)-c$.
Clearly, $y(0)$, $y(1)\in\mathbb{R}^{n,n-1}$, and $y(0)\sim y(1)$ since $y(1)-y(0)=0$. Moreover,
$$
\int_0^1y(t)dt=\int_{0}^1x(t)dt-c=-c\in J_{2n}V^{n,n-1}_0,
$$
and that $x(1)-x(0)=0$ implies
\begin{eqnarray*}
A(y)&=&\frac{1}{2}\int_0^1\langle-J_{2n}\dot{y},y\rangle dt\\
    &=&\frac{1}{2}\int_0^1\langle-J_{2n}\dot{x},x-c\rangle dt\\
    &=&A(x)+\langle J_{2n}(x(1)-x(0)),c\rangle
    =A(x)=1.
\end{eqnarray*}
Hence $y\in \mathcal{A}^{n,n-1}_2$. But it is obvious that
$$
\int_0^1H_D^\ast(-J_{2n}\dot{y})dt=\int_0^1H_D^\ast(-J_{2n}\dot{x})dt.
$$
From (\ref{e:ad3.2-}) and these we can derive (\ref{e:add2}) as follows:
\begin{eqnarray*}
c_{\rm LR}(D,D\cap \mathbb{R}^{n,n-1})&=&\min_{{z}\in\mathcal{A}^{n,n-1}_2}\int_0^1H_D^\ast(-J_{2n}\dot{{z}})dt\\
            &\le& H_D^\ast(-J_{2n}\dot{y})dt\\
            &=&\int_0^1H_D^\ast(-J_{2n}\dot{x})dt
            =c_{\rm HZ}(D).
\end{eqnarray*}

\subsection{Proof of Corollary \ref{cor:ellipsoid}}\label{sec:4.2}

For the sake of convenience we understand elements of $\mathbb{R}^{2n}$ as
column vectors. Note that $E(r_1,\cdots,r_n)=\{q<1\}$, where
$q(z)=\frac{1}{2}\langle Sz, z\rangle$ with $S={\rm Diag}(2/r^2_1,\cdots, 2/r^2_n, 2/r^2_1,\cdots, 2/r^2_n)$.
Since the Hamiltonian vector field of the quadratic form $q(z)$ on $\mathbb{R}^{2n}$
is $X_q(z)=J_{2n}Sz$, every characteristic  on $\partial E(r_1,\cdots,r_n)$ may be parameterized as the form
$$
[0, T]\ni \tau\mapsto \gamma_z(\tau):=\exp(\tau J_{2n}S)z\in \partial E(r_1,\cdots,r_n)
$$
 where $z\in\partial E(r_1,\cdots,r_n)$. Noting that $S$
 commutes with $J_{2n}$, it is easily computed that
  \begin{eqnarray*}
 &\exp(\tau J_{2n}S)=\\
 &\hspace{-3mm}\left(
           \begin{array}{cc}
             {\rm Diag}(\cos(2\tau/r^2_1),\cdots, \cos(2\tau/r^2_n))& -{\rm Diag}(\sin(2\tau/r^2_1),\cdots, \sin(2\tau/r^2_n)) \\
            {\rm Diag}(\sin(2\tau/r^2_1),\cdots, \sin(2\tau/r^2_n)) & {\rm Diag}(\cos(2\tau/r^2_1),\cdots, \cos(2\tau/r^2_n)) \\
           \end{array}
         \right)
\end{eqnarray*}
Hence for $z=(q_1,\cdots,q_n,p_1,\cdots,p_k,0,\cdots,0)^t\in\partial E(r_1,\cdots,r_n)\cap\mathbb{R}^{n,k}$, i.e.,
\begin{equation}\label{e:4.1}
 \sum^k_{j=1}\frac{q_j^2+p_j^2}{r^2_j}+ \sum^n_{j=k+1}\frac{q_j^2}{r^2_j}=1,
  \end{equation}
 we have $\exp(\tau J_{2n}S)z=(X_1(\tau),\cdots, X_n(\tau), Y_1(\tau),\cdots,Y_n(\tau))^t$,
  where
 \begin{eqnarray*}
 &&X_j(\tau)=q_j\cos(2\tau/r^2_j)-p_j\sin(2\tau/r^2_j),\quad j=1,\cdots,k,\\
 &&X_j(\tau)=q_j\cos(2\tau/r^2_j),\quad j=k+1,\cdots,n,\\
&&Y_j(\tau)=q_j\sin(2\tau/r^2_j)+p_j\cos(2\tau/r^2_j),\quad j=1,\cdots,k,\\
 &&Y_j(\tau)=q_j\sin(2\tau/r^2_j),\quad j=k+1,\cdots,n.
 \end{eqnarray*}
The condition that $\exp(TJ_{2n}S)z\in\mathbb{R}^{n,k}$ for  $T>0$ is equivalent to the following
 \begin{equation}\label{e:4.2}
 Y_j(T)=q_j\sin(2T/r^2_j)=0,\quad j=k+1,\cdots,n.
  \end{equation}
 In this case the requirement that $\exp(TJ_{2n}S)z-z\in V_0^{n,k}$
 is equivalent to the following
 \begin{equation}\label{e:4.3}
 \left.\begin{array}{ll}
 &q_j\left(\cos(2T/r^2_j)-1\right)-p_j\sin(2T/r^2_j)=0,\\
 &q_j\sin(2T/r^2_j)+p_j\left(\cos(2T/r^2_j)-1\right)=0,\\
 &\quad j=1,\cdots, k.
 \end{array}\right\}
   \end{equation}
 A direct computation also shows
 \begin{equation}\label{e:4.4}
 \langle-J_{2n}\dot\gamma_z(\tau), \gamma_z(\tau)\rangle=2\sum^k_{j=1}\frac{q_j^2+p_j^2}{r^2_j}+ 2\sum^n_{j=k+1}\frac{q_j^2}{r^2_j}=2
  \end{equation}
  and hence
  \begin{eqnarray}\label{e:4.5}
 A(\gamma_z)&=&\frac{1}{2}\int^T_0\langle-J_{2n}\dot\gamma_z(\tau), \gamma_z(\tau)\rangle d\tau\nonumber\\
 &=& T\sum^k_{j=1}\frac{q_j^2+p_j^2}{r^2_j}+ T\sum^n_{j=k+1}\frac{q_j^2}{r^2_j}=T.
  \end{eqnarray}

Our aim is  to find the smallest $T>0$ for all
$$
z=(q_1,\cdots,q_n,p_1,\cdots,p_k,0,\cdots,0)^t\in\partial E(r_1,\cdots,r_n)\cap\mathbb{R}^{n,k}
$$
satisfying (\ref{e:4.2}) and (\ref{e:4.3}).

\noindent{\bf Case 1}. $q_j=0$, $j=k+1,\cdots, n$. In this situation (\ref{e:4.1}) implies that for some $j\in\{1,\cdots,k\}$,
  \begin{equation}\label{e:4.6}
 \left.\begin{array}{ll}
 &q_j\left(\cos(2T/r^2_j)-1\right)-p_j\sin(2T/r^2_j)=0,\\
 &q_j\sin(2T/r^2_j)+p_j\left(\cos(2T/r^2_j)-1\right)=0
 \end{array}\right\}
   \end{equation}
has nonzero solutions and thus
$
\left(\cos(2T/r^2_j)-1\right)^2+\left( \sin(2T/r^2_j)\right)^2=0
$
or equivalently $T=m\pi r_j$ for some $m\in\N$.  It follows that $A(\gamma_z)=m\pi r^2_j$.
Let $r_1=\min\{r_i\,|\, 1\le i\le k\}$ without loss of generality.
Choose
$$
z=(q_1,\cdots,q_n,p_1,\cdots,p_n)^t\in\partial E(r_1,\cdots,r_n)\cap\mathbb{R}^{n,k}
$$
such that
$q_1^2+ p_1^2=r^2_1$ and $(q_j,p_j)=0$ for $j\ne 1$. Then $z$ and $T=\pi r^2_1$
satisfy (\ref{e:4.1})-(\ref{e:4.3}), and so $A(\gamma_z)=\pi\min\{r^2_j\,|\, 1\le j\le k\}$.

\noindent{\bf Case 2}. $q_j\ne 0$ for some $j\in\{k+1,\cdots, n\}$. By (\ref{e:4.2}) we have $T=\frac{m\pi}{2}r^2_j$
with $m\in\N$, and hence $A(\gamma_z)=\frac{m\pi}{2}r^2_j$.
On the other hand, let $r_n=\min\{r_i\,|\, k+1\le i\le n\}$ without loss of generality.
Choose $z=(q_1,\cdots,q_n,p_1,\cdots,p_n)^t\in\partial E(r_1,\cdots,r_n)\cap\mathbb{R}^{n,k}$ such that
$q_n^2=r^2_n$ and $p_1=\cdots=p_n=0$ and $q_j=0$ for $j\ne n$. Then $z$ and $T=\frac{\pi}{2}r^2_n$
satisfy (\ref{e:4.1})-(\ref{e:4.3}), and so $A(\gamma_z)=\frac{\pi}{2}\min\{r^2_j\,|\, k+1\le j\le n\}$.

Summarizing up the above two cases, (\ref{e:ellipsoid}) is derived from Theorem~\ref{th:represention}.

\subsection{Proof of Corollary \ref{cor:ellipsoid+}}\label{sec:4.3}

Let $H(x)=|x-{\bf a}|^2$ for $x\in\mathbb{R}^{2n}$. Then $\partial B^{2n}({\bf a},1)=H^{-1}(1)$. Let
$x:[0,T]\to H^{-1}(1)$ for $T>0$ satisfy
\begin{equation}\label{e:chord2}
\dot{x}=J\nabla H(x),\quad
x(0),x(T)\in\mathbb{R}^{n,k} \;\hbox{and}\; x(0)\sim x(T).
\end{equation}
Then it has action
\begin{eqnarray*}
A(x)&=&\frac{1}{2}\int_0^T\langle -J\dot{x},x\rangle dt\\
&=&\frac{1}{2}\int_0^T\langle -J\dot{x},x-{\bf a}\rangle dt+\frac{1}{2}\langle -J(x(T)-x(0)),{\bf a }\rangle\\
&=&T+\frac{1}{2}\langle -J(x(T)-x(0)),{\bf a }\rangle.
\end{eqnarray*}
Let
$x(0)=(q_1,\cdots,q_n,p_1,\cdots,p_k,0,\cdots,0)$
and $x(T)=(\hat{q}_1,\cdots,\hat{q}_n,\hat{p}_1,\cdots,\hat{p}_n)$.
We have
\begin{equation}\label{e:chord}
x(T)=e^{2TJ}(x(0)-{\bf a })+{\bf a },
\end{equation}
which leads to
\begin{eqnarray*}
&&\hat{q}_i=q_i\cos(2T)-p_i\sin(2T),\;\hbox{for}\; 1\le i\le k;\\
&&\hat{q}_i=q_i\cos(2T),\;\hbox{for}\;k+1\le i\le n-1;\\
&&\hat{q}_n=q_n\cos(2T)+a\sin(2T);\\
&&\hat{p}_i=p_i\cos(2T)+q_i\sin(2T),\;\hbox{for}\; 1\le i\le k;\\
&&\hat{p}_i=q_i\sin(2T),\;\hbox{for}\;k+1\le i\le n-1\\
&&\hat{p}_n=a(1-\cos(2T))+q_n\sin (2T).
\end{eqnarray*}
Then second line in (\ref{e:chord2}) implies
\begin{eqnarray}
&&\hat{p}_i=0,\;\hbox{for}\; k+1\le i\le n;\label{e:chord3}\\
&&\hat{q}_i=q_i \;\hbox{and}\; \hat{p}_i=p_i, \;\hbox{for}\; 1\le i\le k\label{e:chord4}
\end{eqnarray}

\noindent{\bf Case 1}.
$(q_1,\cdots,q_k,p_1,\cdots,p_k)\ne (0,\cdots,0)$. (\ref{e:chord4}) implies that $T=\pi$,
$x(T)=x(0)$ and $A(x)=T=\pi$.

\noindent{\bf Case 2}.
$(q_1,\cdots,q_k,p_1,\cdots,p_k)= (0,\cdots,0)$.
\begin{itemize}
\item[(i)]
If $q_{i_0}\ne 0$ for some $k+1\le i_0\le n-1$ and $a\ne 0$, then by (\ref{e:chord3})
we have $\sin(2T)=0$, $\cos(2T)=1$ and hence
$T=\pi$ and $A(x)=\pi$.

\item[(ii)]If $q_{i_0}\ne 0$ for some $k+1\le i_0\le n-1$ and $a=0$ then by (\ref{e:chord3}) it is enough to require
$\sin(2T)=0$
and hence
$T=\frac{\pi}{2}$ and $A(x)=\frac{\pi}{2}$.

\item[(iii)] If $q_i=0$ for all $k+1\le i\le n-1$ then $q_n^2+a^2=1$. The problem
is reduced to 2-dimensional. Take $q_n=r=\sqrt{1-a^2}$ when $a\le0$, and $q_n=-r=-\sqrt{1-a^2}$ when $a>0$. Then
\begin{equation}\label{e:chord5}
A(x)=\arcsin (r)-r\sqrt{1-r^2}\le\frac{\pi}{2}.
\end{equation}
\end{itemize}

Summarizing the above arguments, the possibly least action attained by leafwise chord is
given by (\ref{e:chord5}). (\ref{e:1ellipsoid+}) is proved by Theorem~\ref{th:represention}.

\subsection{Proof of Corollary \ref{cor:ellipsoid++}}\label{sec:4.4}

For $R>0$ denote by $\Delta_R$ the convex (open) domain in $\mathbb{R}^2$ surrounded by the curves
$$
\{q_n^2+p_n^2=1\;|\;p_n\ge 0\},\quad \{q_n=1\},\quad \{q_n=-1\},\quad \{p_n=-R\}.
$$
Let $j_R$ be the Minkowski functional of $\Delta_R$. Consider the convex domain in $\mathbb{R}^{2n}$ given by
$$
\widetilde{\Delta}_R:=\{z\in \mathbb{R}^{2n}\,|\,\tilde{j}_{R}^2(z):=|z_1|^2/R^2+j_R^2(z_2)<1\},
$$
where
$z_1=(q_1,\cdots,q_{n-1},p_1,\cdots,p_{n-1})$ and $z_2=(q_n,p_n)$ and
 for
 $$
 z=(q_1,\cdots,q_n,p_1,\cdots,p_n)\in\mathbb{R}^{2n}.
 $$
  Then for $0<R_1\le R_2$,
$\widetilde{\Delta}_{R_1}\subset\widetilde{\Delta}_{R_2}$,
$\cup_{R>1} \widetilde{\Delta}_R =U^{2n}(1)$
and hence
$$
c_{\rm LR}\left(U^{2n}(1), U^{n,k}(1)\right) = \sup_{R>0}c_{\rm LR}\left(\widetilde{\Delta}_R,\widetilde{\Delta}_R\cap\mathbb{R}^{n,k}\right).
$$
We claim that
$$
c_{\rm LR}\left(\widetilde{\Delta}_R,\widetilde{\Delta}_R\cap\mathbb{R}^{n,k}\right) =\frac{\pi}{2},\quad\forall R>1.
$$
To this end, let an absolutely continuous curve
$x:[0,T] \to \partial \widetilde{\Delta}_R$
satisfy
\begin{equation}\label{e:chord6}
\dot{x}=J_{2n}\partial \tilde{j}_R^2(x),\;{\rm a.\, e.},\quad
x(0), x(T)\in\mathbb{R}^{n,k},\; x(0)\sim x(T).
\end{equation}
Then the action of $x$ is
$A(x)=T$.
Since both $z_1\mapsto f(z_1):=|z_1|^2/R^2$ and $z_2\mapsto j_R^2(z_2)$ are convex and continuous,
we have
$$
\partial\tilde{j}_R^2(z_1,z_2)=\partial f(z_1)\times\partial j_R^2(z_2)=\{(2z_1/R^2,u)\,|\, u\in \partial j_R^2(z_2)\}.
$$
Write $x(t)=(z_1(t),z_2(t))$ then the  problem  (\ref{e:chord6}) is equivalent to the system
\begin{eqnarray}
&&\hspace{-4mm}\dot{z}_1=2J_{2(n-1)} z_1/R^2, \quad z_1(0), z_1(T)\in\mathbb{R}^{n-1,k}, \quad z_1(0)\sim z_1(T),\label{e:chord6.1}\\
&&\hspace{-4mm}\dot{z}_2=J_2\partial j_R^2(z_2), \quad z_2(0), z_2(T)\in\mathbb{R}^{1,0}, \quad z_2(0)\sim z_2(T)\label{e:chord6.2}.
\end{eqnarray}

\textsf{ We want to find the smallest $T>0$ such that (\ref{e:chord6}) holds, i.e., (\ref{e:chord6.1}) and (\ref{e:chord6.2}) hold at the same time.}

If $z_1(0)\ne 0$, (\ref{e:chord6.2}) implies that $T\ge \pi R^2/2>\pi/2$.

If $z_1(0)=0$, then $z_1=0$, $z_2(0)\in\partial\Delta_R\cap\mathbb{R}^{1,0}$ and thus $z_2(t)\in\partial\Delta_R$ for all $t\in [0, T]$.
In this situation $T$ is equal to the smaller of areas of the
semi-disk
$\{q_n^2+p_n^2\le 1\;|\;p_n\ge 0\}$
and the rectangle
$\{(q_n,p_n)\in\mathbb{R}^2\,|\, -1\le q_n\le 1,\; -R\le p_n\le 0\}$.
Thus $T=\pi/2$.

It follows that the smallest $T$ is $\pi/2$.

\subsection{Proof of Corollary \ref{cor:polyDisk}}\label{sec:4.5}

(\ref{e:polyD.2}) comes from the fact that $P^{2n}(r_1,\cdots,r_n)$ contain the ellipsoid given by
(\ref{e:ellipsoid-}). Let $(q_j, p_j)$ denote coordinates in $B^2(r_j)$ for each $j=1,\cdots,n$.
For each $k+1\le j<n$, the linear isomorphism  $\phi_j$ on $(\mathbb{R}^{2})^n$
only permuting coordinates $(q_j,p_j)$ and $(q_n,p_n)$  is a symplectomorphism fixing
$(\mathbb{R}^2)^k\times(\mathbb{R}\times\{0\})^{n-k}$ and $(\{0\}\times\{0\})^k\times(\mathbb{R}\times\{0\})^{n-k}$.
By Definition~\ref{def:coCap}(i),
\begin{eqnarray*}
&&c_{\rm LR}\left(P^{2n}(r_1,\cdots,r_n), P^{2n}(r_1,\cdots,r_n)\cap\mathbb{R}^{n,k}\right)\\
&=&c_{\rm LR}\left(P^{2n}(r'_1,\cdots, r'_n), P^{2n}(r'_1,\cdots,r'_n)\cap\mathbb{R}^{n,k}\right),
\end{eqnarray*}
where $r'_i=r_i$ for $i\ne j, n$, and $r'_n=r_j$, $r'_j=r_n$. Hence we can assume
$r_n=\min\{r_{k+1},\cdots, r_n\}$ for proving (\ref{e:polyD.1}).
Observe that $P^{2n}(r_1,\cdots,r_n)\subset r_nU^{2n}(1)=\{r_nz\,|\, z\in U^{2n}(1)\}$.
(\ref{e:polyD.1})  follows from Definition~\ref{def:coCap}(i)-(ii) and
(\ref{e:ellipsoid4+}) immediately.

\section{Proofs of Theorem~\ref{th:Brun} and Corollary~\ref{cor:Brun.2}  }\label{sec:Brunn}
\setcounter{equation}{0}

\subsection{Proof of Theorem~\ref{th:Brun}}\label{sec:Brunn.1}

Recall that for a convex domain $D\subset\mathbb{R}^{2n}$ containing $0$, $H_D=j_D^2$ and $H_D^\ast$ is the Legendre transformation
of $H_D$ given by $H_D^\ast(z)=\max _{\xi\in\mathbb{R}^{2n}}(\langle\xi,z\rangle-H_D(\xi))$. The support function of $D$ is
$$h_D(w)=\sup\{\langle x,w\rangle\,|\,x\in D\},\quad\forall w\in\mathbb{R}^{2n}$$
and there holds $(h_D)^2= 4H^\ast_D$.

For a real $p>1$, Proposition~\ref{prop:Brun.0} and (\ref{capacity}) in Theorem \ref{th:represention} show
\begin{equation}\label{e:Brun.1}
\left(c_{\rm LR}(D, D\cap\mathbb{R}^{n,k})\right)^{\frac{p}{2}}=\min_{x\in\mathcal{A}_p^{n,k}}
\int_0^1 \left(H_D^{\ast}(-J_{2n}\dot{x}(t))\right)^{\frac{p}{2}}dt.
\end{equation}
Corresponding to \cite[Proposition~2.1]{AAO08} we have

\begin{prop}\label{prop:Brun.2}
  For  real numbers $p_1>1$ and $p_2\ge 1$, there holds
  \begin{eqnarray*}
  (c_{\rm LR}(D, D\cap\mathbb{R}^{n,k}))^{\frac{p_2}{2}}&=&\min_{x\in\mathcal{A}_{p_1}^{n,k}}
\int_0^1(H_D^{\ast}(-J\dot{x}(t)))^{\frac{p_2}{2}}dt\\
&=&\min_{x\in\mathcal{A}_{p_1}^{n,k}}\frac{1}{2^{p_2}}\int_0^1
  (h_{D}(-J\dot{x}))^{p_2}dt.
  \end{eqnarray*}
\end{prop}
This can be proved as in the proof of \cite[Proposition~2.1]{AAO08};
that is, we may obtain it by replacing  $c_{\rm EHZ}^\Psi$ by $c_{\rm LR}(D, D\cap\mathbb{R}^{n,k})$
in the proof of \cite[Proposition~9.3]{JinLu1916}.\\

\noindent{\bf Proof of Theorem~\ref{th:Brun}}.\quad
Choose a real $p_1>1$.
  Proposition~\ref{prop:Brun.2} implies that {for any $p\ge 1$}
\begin{eqnarray}\label{e:Brun.8}
&&\left(c_{\rm LR}(D+_pK, (D+_pK)\cap\mathbb{R}^{n,k})\right)^{\frac{p}{2}}\nonumber\\
&=&\min_{x\in\mathcal{A}_{p_1}^{n,k}}\int_0^1
  \left(\frac{(h_{D+_pK}(-J\dot{x}))^2}{4}\right)^{\frac{p}{2}} dt\nonumber\\
    &=&\min_{x\in\mathcal{A}_{p_1}^{n,k}}\frac{1}{2^p}\int_0^1
  ((h_{D}(-J\dot{x}))^{p}+(h_{K}(-J\dot{x}))^{p}) dt\nonumber\\
  &\ge& \min_{x\in\mathcal{A}_{p_1}^{n,k}}\frac{1}{2^p}\int_0^1
  (h_{D}(-J\dot{x}))^{p}dt+\min_{x\in\mathcal{A}_{p_1}^{n,k}}\frac{1}{2^p}\int_0^1
  (h_{K}(-J\dot{x}))^{p}dt\nonumber\\
  &=&\left(c_{\rm LR}(D, D\cap\mathbb{R}^{n,k})\right)^{\frac{p}{2}}+ \left(c_{\rm LR}(K, K\cap\mathbb{R}^{n,k})\right)^{\frac{p}{2}}.
\end{eqnarray}
{Suppose that  there exists a $c_{\rm LR}(D, D\cap\mathbb{R}^{n,k})$ carrier,
$\gamma_D:[0, T]\to\partial D$,  and
a $c_{\rm LR}(K, K\cap\mathbb{R}^{n,k})$ carrier,
 $\gamma_K:[0, T]\to\partial K$
that satisfy $\gamma_D=\alpha\gamma_K+{\bf b}$
for some $\alpha\in\mathbb{R}\setminus\{0\}$ and
some ${\bf b}\in \mathbb{R}^{n,k}$.
Then the latter  implies $A(\gamma_D)=\alpha^2 A(\gamma_K)$.
Moreover by Step 4 in the proof of Proposition~\ref{prop:Brun.0} we can
construct $z_D$ and $z_K$ in $\mathcal{A}_{p_1}^{n,k}$ such that
\begin{eqnarray}\label{e:Brun.8.1}
\left(c_{\rm LR}(D, D\cap\mathbb{R}^{n,k})\right)^{\frac{p}{2}}&=&\min_{x\in\mathcal{A}_{p_1}^{n,k}}\frac{1}{2^p}\int_0^1
  (h_{D}(-J\dot{x}))^{p}dt\nonumber\\
  &=&\frac{1}{2^p}\int_0^1
  (h_{D}(-J\dot{z}_D))^{p}dt,\\
  \left(c_{\rm LR}(K, K\cap\mathbb{R}^{n,k})\right)^{\frac{p}{2}}&=&\min_{x\in\mathcal{A}_{p_1}^{n,k}}\frac{1}{2^p}\int_0^1
  (h_{K}(-J\dot{x}))^{p}dt\nonumber\\
  &=&\frac{1}{2^p}\int_0^1
  (h_{K}(-J\dot{z}_K))^{p}dt.\label{e:Brun.8.2}
\end{eqnarray}
In particular, for suitable vectors ${\bf b}_D, {\bf b}_K\in \mathbb{R}^{n,k}$ it holds that
\begin{eqnarray*}
z_D(t)=\frac{1}{\sqrt{A(\gamma_D)}}\gamma_D(Tt)+{\bf b}_D\quad\hbox{and}\quad
z_K(t)=\frac{1}{\sqrt{A(\gamma_K)}}\gamma_K(Tt)+{\bf b}_K.
\end{eqnarray*}
It follows from these that $\dot{z}_D(t)=\alpha\left(\frac{A(\gamma_K)}{A(\gamma_D)}\right)^{1/2}\dot{z}_K=\dot{z}_K$
because $A(\gamma_D)=\alpha^2 A(\gamma_K)$.
This, (\ref{e:Brun.8.1})-(\ref{e:Brun.8.2}) and (\ref{e:Brun.8}) lead to
\begin{eqnarray*}
&&\left(c_{\rm LR}(D+_pK, (D+_pK)\cap\mathbb{R}^{n,k})\right)^{\frac{p}{2}}\\
&=&\left(c_{\rm LR}(D, D\cap\mathbb{R}^{n,k})\right)^{\frac{p}{2}}+\left(c_{\rm LR}(K, K\cap\mathbb{R}^{n,k})\right)^{\frac{p}{2}}.
\end{eqnarray*}
}

Now suppose that {\bf $p>1$} and the equality in (\ref{e:BrunA}) holds.
We may require that the above $p_1$ satisfies $p_1<p$.
Since  there exists $u\in\mathcal{A}_{p_1}^{n,k}$ such that
\begin{eqnarray*}
 \left(c_{\rm LR}(D+_pK, (D+_pK)\cap\mathbb{R}^{n,k})\right)^{\frac{p}{2}} =\int_0^1
  \left(\frac{(h_{D+_pK}(-J\dot{u}))^2}{4}\right)^{\frac{p}{2}}dt,
  \end{eqnarray*}
 it follows from the above computation that
 \begin{eqnarray*}
  &&\frac{1}{2^p}\int_0^1
  ((h_{D}(-J\dot{u}))^{p}+(h_{K}(-J\dot{u}))^{p})dt\\
  &=& \min_{x\in\mathcal{A}_{p_1}^{n,k}}\frac{1}{2^p}\int_0^1
  (h_{D}(-J\dot{x}))^{p}dt
  +\min_{x\in\mathcal{A}_{p_1}^{n,k}}\frac{1}{2^p}\int_0^1
  (h_{K}(-J\dot{x}))^{p}dt
\end{eqnarray*}
 and thus
\begin{eqnarray*}
\left(c_{\rm LR}(D, D\cap\mathbb{R}^{n,k})\right)^{\frac{p}{2}}&=&\min_{x\in\mathcal{A}_{p_1}^{n,k}}\frac{1}{2^p}\int_0^1
  (h_{D}(-J\dot{x}))^{p}dt\\
  &=&\frac{1}{2^p}\int_0^1
  (h_{D}(-J\dot{u}))^{p}dt\quad\hbox{and}\\
 \left(c_{\rm LR}(K, K\cap\mathbb{R}^{n,k})\right)^{\frac{p}{2}}&=&\min_{x\in\mathcal{A}_{p_1}^{n,k}}\frac{1}{2^p}\int_0^1
  (h_{K}(-J\dot{x}))^{p}dt\\
  &=&\frac{1}{2^p}\int_0^1
  (h_{K}(-J\dot{u}))^{p}dt.
\end{eqnarray*}
As in the proof \cite[Theorem~1.22]{JinLu1916},
we may derive from these, (\ref{e:Brun.1}), Proposition~\ref{e:Brun.1}   and H\"older's inequality that
 \begin{eqnarray*}
 2(c_{\rm LR}(D, D\cap\mathbb{R}^{n,k}))^{\frac{1}{2}}&=&
  \left(\int_0^1(h_{D}(-J\dot{u}))^{p}dt\right)^{\frac{1}{p}}\\
  &=&  \left(\int_0^1(h_{D}(-J\dot{u}))^{p_1}dt\right)^{\frac{1}{p_1}},\\
 2(c_{\rm LR}(K, K\cap\mathbb{R}^{n,k}))^{\frac{1}{2}}&=&
  \left(\int_0^1(h_{K}(-J\dot{u}))^{p} dt\right)^{\frac{1}{p}}\\
 &=&\left(\int_0^1(h_{K}(-J\dot{u}))^{p_1} dt\right)^{\frac{1}{p_1}}.
 \end{eqnarray*}
By Remark~\ref{rem:Brun.1.5} there are
 ${\bf a}_D, {\bf a}_K\in \mathbb{R}^{n,k}$ such that
\begin{eqnarray*}
&&\hspace{-4mm}\gamma_D(t)=\left(c_{\rm LR}(D, D\cap\mathbb{R}^{n,k})\right)^{1/2}u(t)+ \frac{2}{p_1}\left(c_{\rm LR}(D, D\cap\mathbb{R}^{n,k})\right)^{(1-p_1)/2}{\bf a}_D,\\
&&\hspace{-4mm}\gamma_K(t)=\left(c(K, K\cap\mathbb{R}^{n,k})\right)^{1/2}u(t)+ \frac{2}{p_1}\left(c(K, K\cap\mathbb{R}^{n,k})\right)^{(1-p_1)/2}{\bf a}_K
\end{eqnarray*}
are a $c_{\rm LR}(D, D\cap\mathbb{R}^{n,k})$ carrier and a
$c_{\rm LR}(K, K\cap\mathbb{R}^{n,k})$ carrier, respectively.
Clearly, they coincide up to  dilations and translations in
$\mathbb{R}^{n,k}$.
\hfill$\Box$\vspace{2mm}

\subsection{Proof of  Corollary~\ref{cor:Brun.2}  }\label{sec:Brunn.2}

{\bf (i)} By the assumptions, $0$ belongs to  $D\cap (x+K)$, $D\cap (y+K)$ and $D\cap(\lambda x+(1-\lambda)y+K)$.
The first inequality follows from  Corollary~\ref{cor:Brun.1} directly.

Suppose further that $D$ and $K$ are centrally symmetric, i.e., $-D=D$ and $-K=K$.
Then for any  $x\in K$, $D\cap (x+K)\ne\emptyset$ and   $D\cap(-x+K)=-(D\cap(x+K))$,  hence
taking $y=-x$ and $\lambda=1/2$ in the first inequality we get (\ref{e:BrunB}).

\noindent{\bf (ii)} Since $0\in D\cap K$,
 $K\subset RD$ for some $R>0$, and thus
\begin{eqnarray*}
&&\varepsilon\left(c_{\rm LR}(K, K\cap \mathbb{R}^{n,k})\right)^{1/2}\\
&\le&\left(c_{\rm LR}(D+\varepsilon K, (D+\varepsilon K)\cap\mathbb{R}^{n,k})\right)^{1/2}-\left(c_{\rm LR}(D, D\cap \mathbb{R}^{2n})\right)^{1/2}\\
&\le&\left(c_{\rm LR}((1+\varepsilon R)D, (1+\varepsilon R)D\cap\mathbb{R}^{n,k})\right)^{1/2}-\left(c_{\rm LR}(D, D\cap \mathbb{R}^{2n})\right)^{1/2}\\
&=&\varepsilon R\left(c_{\rm LR}(D, D\cap \mathbb{R}^{2n})\right)^{1/2}.
\end{eqnarray*}
This implies that the function
\begin{equation}\label{e:5.7}
(0,\infty)\ni\varepsilon\mapsto \frac{\left(c_{\rm LR}(D+ \varepsilon K, (D+\varepsilon K)\cap\mathbb{R}^{n,k} )\right)^{\frac{1}{2}}- \left(c_{\rm LR}(D, D\cap\mathbb{R}^{n,k})\right)^{\frac{1}{2}} }{\varepsilon}
\end{equation}
is bounded, and
\begin{equation}\label{e:5.8}
\lim_{\varepsilon\to 0+}\left(c_{\rm LR}(D+\varepsilon K, (D+\varepsilon K)\cap\mathbb{R}^{n,k})\right)^{1/2}=\left(c_{\rm LR}(D, D\cap \mathbb{R}^{2n})\right)^{1/2}.
\end{equation}
 Moreover, for $t>s>0$, as in \cite[pages 21-22]{AAO08} it follows from (\ref{e:BrunA}) that
\begin{eqnarray*}
&&\hspace{-4mm}(1-s/t)\left(c_{\rm LR}(D, D\cap\mathbb{R}^{n,k})\right)^{\frac{1}{2}}+ (s/t)\left(c_{\rm LR}(D+tK, (D+tK)\cap\mathbb{R}^{n,k})\right)^{\frac{1}{2}}\\
 &&\hspace{-4mm}\le \left(c_{\rm LR}(D+ sK, (D+sK)\cap\mathbb{R}^{n,k} )\right)^{\frac{1}{2}}
 \end{eqnarray*}
and so
\begin{eqnarray*}
&&\frac{\left(c_{\rm LR}(D+ tK, (D+tK)\cap\mathbb{R}^{n,k} )\right)^{\frac{1}{2}}- \left(c_{\rm LR}(D, D\cap\mathbb{R}^{n,k})\right)^{\frac{1}{2}} }{t}\\
&\le& \frac{\left(c_{\rm LR}(D+ sK, (D+sK)\cap\mathbb{R}^{n,k} )\right)^{\frac{1}{2}}- \left(c_{\rm LR}(D, D\cap\mathbb{R}^{n,k})\right)^{\frac{1}{2}} }{s}
 \end{eqnarray*}
These imply that the function in (\ref{e:5.7}) has a limit
$\Xi\ge \left(c(K, K\cap \mathbb{R}^{2n})\right)^{1/2}$ as $\varepsilon\to 0+$. So
\begin{eqnarray*}
&&\frac{c_{\rm LR}(D+\varepsilon K, (D+\varepsilon K)\cap\mathbb{R}^{n,k})-c_{\rm LR}(D, D\cap \mathbb{R}^{2n})}{\varepsilon}\\
&=&\frac{\left(c_{\rm LR}(D+ \varepsilon K, (D+\varepsilon K)\cap\mathbb{R}^{n,k} )\right)^{\frac{1}{2}}- \left(c_{\rm LR}(D, D\cap\mathbb{R}^{n,k})\right)^{\frac{1}{2}} }{\varepsilon}\times\\
&\times&\left(\left(c_{\rm LR}(D+\varepsilon K, (D+\varepsilon K)\cap\mathbb{R}^{n,k})\right)^{1/2}+\left(c_{\rm LR}(D, D\cap \mathbb{R}^{2n})\right)^{1/2}\right)
\end{eqnarray*}
has a limit $d_K(D, D\cap\mathbb{R}^{n,k})$ satisfying
\begin{eqnarray*}
d_K(D, D\cap\mathbb{R}^{n,k})&=&2\Xi \left(c_{\rm LR}(D, D\cap \mathbb{R}^{2n})\right)^{1/2}\\
&\ge& 2\left(c_{\rm LR}(D, D\cap \mathbb{R}^{2n})\right)^{1/2}\left(c_{\rm LR}(K, K\cap \mathbb{R}^{2n})\right)^{1/2}
\end{eqnarray*}
because of (\ref{e:5.8}).
The first inequality in (\ref{e:BrunD}) is proved.

In order to prove the second inequality we
fix a real $p_1>1$. Using Proposition~\ref{prop:Brun.2} we have $u\in \mathcal{A}_{p_1}^{n,k}$
such that
  \begin{eqnarray}\label{e:Brun.9}
  (c_{\rm LR}(D, D\cap\mathbb{R}^{n,k}))^{\frac{1}{2}}=\frac{1}{2}\int_0^1h_{D}(-J\dot{u}))dt
  \end{eqnarray}
By Remark~\ref{rem:Brun.1.5} there exists ${\bf a}_0\in \mathbb{R}^{n,k}$ such that
\begin{eqnarray}\label{e:Brun.10}
x^\ast(t)=\left(c_{\rm LR}(D, D\cap\mathbb{R}^{n,k})\right)^{1/2}u(t)+ \frac{2}{p_1}\left(c_{\rm LR}(D, D\cap\mathbb{R}^{n,k})\right)^{(1-p_1)/2}{\bf a}_0
 \end{eqnarray}
is  a $c_{\rm LR}(D, D\cap\mathbb{R}^{n,k})$ carrier.
Proposition~\ref{prop:Brun.2} also leads to
\begin{eqnarray}\label{e:Brun.11}
&&\left(c_{\rm LR}(D+\varepsilon K, (D+\varepsilon K)\cap\mathbb{R}^{n,k})\right)^{\frac{1}{2}}\nonumber\\
  &=&\min_{x\in\mathcal{A}_{p_1}^{n,k}}\frac{1}{2}\int_0^1
  (h_{D}(-J\dot{x})+ \varepsilon h_{K}(-J\dot{x}))dt\nonumber\\
  &\le& \frac{1}{2}\int_0^1
  h_{D}(-J\dot{u})dt+\frac{\varepsilon}{2}\int_0^1
  h_{K}(-J\dot{u})dt\nonumber\\
  &=&(c_{\rm LR}(D, D\cap\mathbb{R}^{n,k}))^{\frac{1}{2}}+\frac{\varepsilon}{2}\int_0^1
  h_{K}(-J\dot{u})dt
\end{eqnarray}
because of (\ref{e:Brun.9}). Let $z_D(t)=x^\ast(t)$ for $0\le t\le 1$.  From (\ref{e:Brun.11}) it follows that
\begin{eqnarray}\label{e:Brun.12}
&&\frac{\left(c_{\rm LR}(D+\varepsilon K, (D+\varepsilon K)\cap\mathbb{R}^{n,k})\right)^{\frac{1}{2}}-
(c_{\rm LR}(D, D\cap\mathbb{R}^{n,k}))^{\frac{1}{2}}}{\varepsilon}\nonumber\\
&\le&\frac{1}{2}\left(c_{\rm LR}(D, D\cap\mathbb{R}^{n,k})\right)^{-\frac{1}{2}}
\int_0^1h_{K}(-J\dot{z}_D)dt.
\end{eqnarray}
Letting $\varepsilon\to 0+$ in (\ref{e:Brun.12})
we arrive at the second inequality in (\ref{e:BrunD}).
\hfill$\Box$\vspace{2mm}

\section{Proof of Theorem~\ref{th:convexDiff}}\label{sec:Neduv}
\setcounter{equation}{0}

{Recall that $\mathcal{H}$ is a strictly convex function in $\mathbb{R}^{2n}$
with nonnegative values such that $\mathcal{H}(0)=0$ and $D(e)=\{\mathcal{H}<e\}$ for
$e$ near a regular value $e_0$ of $\mathcal{H}$ satisfying $\mathcal{H}^{-1}(e_0)\neq\emptyset$. }

For each number $\epsilon$ with
$|\epsilon|$ small enough the set $D_\epsilon:=D(e_0+\epsilon)$
is a bounded strictly convex  domain in $\mathbb{R}^{2n}$ containing $0$
and with $C^2$-boundary $\mathcal{S}_\epsilon=\mathcal{S}(e_0+\epsilon)$.
Let $H_\epsilon=(j_{D_\epsilon})^2$ and $H^\ast_\epsilon$ denote the
Legendre transform of $H_\epsilon$. Both are $C^{1,1}$ on $\mathbb{R}^{2n}$, $C^2$ on $\mathbb{R}^{2n}\setminus\{0\}$,
and have positive Hessian matrices at every point on $\mathbb{R}^{2n}\setminus\{0\}$.
For every fixed $x\in\mathbb{R}^{2n}\setminus\{0\}$, it was proved in \cite[Lemma~3.2]{Ned01} that
 $\epsilon\mapsto H_\epsilon(x)$ is of class $C^2$.

Let $x^\ast$ be a given $c_{\rm LR}(D_0, D_0\cap \mathbb{R}^{n,k})$-carrier.
{As stated in Remark~\ref{rem:CLRcarrier}}, with
$\mu=A(x^\ast)=c_{\rm LR}(D_0,D_0\cap \mathbb{R}^{n,k})$  we may assume that $x^\ast:[0,\mu]\rightarrow \mathcal{S}_0$
satisfies
$$
-J_{2n}\dot{x^\ast}(t)\in  \partial H_0(x^\ast(t)),\quad x^\ast(0), x^\ast(\mu)\in \mathbb{R}^{n,k},\quad x^\ast(0)-x^\ast(\mu)\in V_0^{n,k}.
$$
By Step 3 in the proof  of Proposition~\ref{prop:Brun.0} (taking $p=q=2$), for some ${\bf a}_0\in \mathbb{R}^{n,k}$,
$$
u:[0, 1]\to \mathbb{R}^{2n},\;t\mapsto \frac{1}{\sqrt{\mu}}x^\ast(\mu t)-\frac{{\bf a}_0}{\mu}
$$
belongs to $\mathcal{A}_2$ in (\ref{e:Brunn.1}) and satisfies
  \begin{equation}\label{e:Neduv.1}
-J_{2n}\dot{u}(t)=\nabla H_0(\mu u(t)+ {\bf a}_0)\quad\forall t\in [0,1].
\end{equation}
 Moreover, $\mu=I_2(u)=\int^1_0H^\ast_0(-J_{2n}\dot{u})$ and
\begin{equation}\label{e:Neduv.2}
\mathfrak{C}(\epsilon):=\mathscr{C}(e_0+\epsilon)=c_{\rm LR}(D_\epsilon, D_\epsilon\cap\mathbb{R}^{n,k})\le \int^1_0H^\ast_\epsilon(-J_{2n}\dot{u})dt.
\end{equation}

Clearly, $p\mapsto\langle\nabla\mathscr{H}(p),p\rangle/2$ restricts to a $C^1$
positive function $\alpha:\mathcal{S}_0\to\mathbb{R}$
 such that $\nabla\mathscr{H}(p)=\alpha(p)\nabla H_0(p)$ for any $p\in\mathcal{S}_0$.
Let
$$
h(s)=\int^s_0\frac{d\tau}{\alpha(x^\ast(\tau))}=2\int^s_0\frac{dt}{\langle\nabla\mathscr{H}(x^\ast(t)), x^\ast(t)\rangle}\quad\forall s\in [0,\mu].
$$
Then $h(\mu)=T_{x^\ast}$ and $h:[0,\mu]\to [0, T_{x^\ast}]$ has an inverse $g:[0, T_{x^\ast}]\to [0,\mu]$.
Since (\ref{e:Neduv.1}) implies  $-J\dot{x}^\ast(t)=\nabla H_0(x^\ast(t))$,
it is easily checked that $[0, T_{x^\ast}]\ni t\mapsto y(t):=x^\ast(g(t))\in\mathcal{S}_0$  satisfies
(\ref{e:convexDiff+}) and (\ref{e:convexDiff++}) with $T_x=T_{x^\ast}$.
In particular, $T_{x^\ast}$ is the return time of  $y$ for $\mathbb{R}^{n,k}$.
Almost repeating the arguments below \cite[Theorem~8.2]{JinLu1916} we can get
the corresponding result of \cite[Theorem~3.1]{Ned01} as follows.

\begin{thm}\label{th:Neduv}
{$\mathfrak{C}(\epsilon)\le c_{\rm LR}(D_0, D_0\cap\mathbb{R}^{n,k})+ T_{x^\ast}\epsilon+ K\epsilon^2$}
for some constant $K$ only depending on $\mathcal{S}_0$ and $H_\epsilon$ with $\epsilon$ near $0$.
\end{thm}

Since $T_k^{\max}(e_0+\epsilon)$ and $T_k^{\min}(e_0+\epsilon)$ are the largest and smallest numbers in
the compact set $\mathscr{I}(e_0+\epsilon)$ defined by (\ref{e:convexDiff}),
by \cite[Lemma~4.1]{Ned01} and \cite[Corollary~4.2]{Ned01},
both  are functions of bounded variation in $\epsilon$
(and thus bounded near $\epsilon=0$), and $\epsilon\mapsto \mathscr{C}_k(e_0+\epsilon)$ is continuous.
As in the proof of \cite[Theorem~4.4]{Ned01}, it follows from  these and Theorem~\ref{th:Neduv}
 that $\mathfrak{C}_k(\epsilon)$ has, respectively, the left and right derivatives
at $\epsilon=0$,
\begin{eqnarray*}
&&\mathfrak{C}'_{k,-}(0)=\lim_{\epsilon\to0-}T_k^{\max}(e_0+\epsilon)=T_k^{\max}(e_0)\quad\hbox{and}\quad\\
&&\mathfrak{C}'_{k,+}(0)=\lim_{\epsilon\to0+}T_k^{\min}(e_0+\epsilon)=T_k^{\min}(e_0).
\end{eqnarray*}
 The first part of Theorem~\ref{th:convexDiff} is proved.
The final part is a direct consequence of the first one and
a modified version of the intermediate value theorem (cf. \cite[Theorem~5.1]{Ned01}).

 \section{Proofs of Theorem~\ref{th:J-estimate} and Example~\ref{ex:add}}\label{sec:add}
\setcounter{equation}{0}

In this section we always assume that $\langle\cdot,\cdot\rangle$ and $\|\cdot\|$ denote the standard inner product and norm in
$\mathbb{R}^{2n}$, respectively. Then $\langle u, v\rangle=\omega_0(u,J_{2n}v)$.
Let ${\rm Sp}(\mathbb{R}^{2n},\omega_0)$ be the group of linear symplectic
maps on $(\mathbb{R}^{2n},\omega_0)$.
Let $S^T$ be the adjoint of $S\in {\rm Sp}(\mathbb{R}^{2n},\omega_0)$ with respect to the inner product $\langle\cdot,\cdot\rangle$.
Then the matrix of $S^T$ under the standard symplectic basis $\{e_j\}^{2n}_{j=1}$ of $(\mathbb{R}^{2n},\omega_0)$ is
the transpose of the matrix of $S$ under the same basis $\{e_j\}^{2n}_{j=1}$.
Hence $S^T$ also belongs to ${\rm Sp}(\mathbb{R}^{2n},\omega_0)$.
Let $\mathbb{R}^{n,0}=\mathbb{R}^{n}\times 0\subset(\mathbb{R}^{2n},\omega_0)$ (resp.
$\mathbb{R}^{0,n}:=0\times\mathbb{R}^{n}\subset(\mathbb{R}^{2n},\omega_0)$)
be the horizontal (resp. vertical) Lagrangian subspace in $(\mathbb{R}^{2n},\omega_0)$.
Consider the stabilizers (or isotropy subgroups) of $\mathbb{R}^{n,0}$ and $\mathbb{R}^{0,n}$,
\begin{eqnarray*}
&&{\rm Sp}(\mathbb{R}^{2n},\omega_0;\mathbb{R}^{n,0}):=\{S\in {\rm Sp}(\mathbb{R}^{2n},\omega_0)\;|\;S(\mathbb{R}^{n,0})=\mathbb{R}^{n,0}\},\\
&&{\rm Sp}(\mathbb{R}^{2n},\omega_0;\mathbb{R}^{0,n}):=\{S\in {\rm Sp}(\mathbb{R}^{2n},\omega_0)\;|\;S(\mathbb{R}^{0,n})=\mathbb{R}^{0,n}\}.
\end{eqnarray*}
It is easy to check that  $S\in {\rm Sp}(\mathbb{R}^{2n},\omega_0;\mathbb{R}^{n,0})$
if and only if $S^T$ sits in ${\rm Sp}(\mathbb{R}^{2n},\omega_0;\mathbb{R}^{0,n})$.

\begin{lemma}\label{lem:Lagr}
Let $\{e_j\}^{2n}_{j=1}$ be the standard symplectic basis of $(\mathbb{R}^{2n},\omega_0)$,
($e_{n+j}=J_{2n}e_j$ for $j=1,\cdots,n$), and let $u\in \mathbb{R}^{n,0}$ and $v\in \mathbb{R}^{2n}$ satisfy $\omega_0(u,v)=1$.
Then for each $e_j$, $j=1,\cdots,n$, there exists a $S\in {\rm Sp}(\mathbb{R}^{2n},\omega_0;\mathbb{R}^{n,0})$ such that
$u=Se_j$ and $v=SJ_{2n}e_j$. Moreover, if $u\in \mathbb{R}^{2n}$ and $v\in \mathbb{R}^{0,n}$ satisfy $\omega_0(u,v)=1$
then for each $e_{j}$, $j=1,\cdots,n$, there exists a $S\in {\rm Sp}(\mathbb{R}^{2n},\omega_0;\mathbb{R}^{n,0})$ such that
$v=S^Te_j$ and $u=-S^TJ_{2n}e_j$.
\end{lemma}
\begin{proof}
It suffices to prove the case $j=1$. If $n=1$ the conclusion is obvious. We assume $n>1$ below.

Let $u\in \mathbb{R}^{n,0}$ and $v\in \mathbb{R}^{2n}$ satisfy $\omega_0(u,v)=1$.
Then $u_1:=u$ and $v_1:=v$ satisfy
$\omega_0(u_1,v_1)=1$.
 Let $\{v_1\}^{\omega_0}$ be the symplectic orthogonal complement of $\mathbb{R}v_1$ in $(\mathbb{R}^{2n},\omega_0)$.
Then $\dim \{v_1\}^{\omega_0}=2n-1$ and so $\dim\mathbb{R}^{n,0}\cap\{v_1\}^{\omega_0}=n-1$.
Take a basis of $\mathbb{R}^{n,0}\cap\{v_1\}^{\omega_0}$, $\{u_j\}^{n}_{j=2}$.
Then $\{u_j\}^{n}_{j=1}$ is a basis of $\mathbb{R}^{n,0}$ and $\omega_0(u_j,v_1)=\delta_{j1}$ for
$j=1,\cdots,n$. By \cite[Theorem~1.15]{deG06} we can extend $\{u_j\}^{n}_{j=1}\cup\{v_1\}$  into a
symplectic basis of $(\mathbb{R}^{2n},\omega_0)$, $\{u_j\}^{n}_{j=1}\cup\{v_j\}^n_{j=1}$, i.e.,
$$
\omega_0(u_i, u_j)=0,\quad \omega_0(v_i, v_j)=0,\quad \omega_0(u_i, v_j)=\delta_{ij}.
$$
Define a linear map $S_1:\mathbb{R}^{2n}\to \mathbb{R}^{2n}$ by
$$
S_1e_j=u_j,\quad S_1e_{n+j}=v_j,\quad j=1,\cdots,n.
$$
Then $S_1$ is symplectic and satisfies: $S_1(\mathbb{R}^{n,0})=\mathbb{R}^{n,0}$, $S_1e_1=u_1$ and $S_1(J_{2n}e_1)\\=
S_1e_{n+1}=v_1$.

Let $M$ be the subspace of $\mathbb{R}^{2n}$ spanned by $\{u_1,v_1\}$
and let $M^{\omega_0}$ be the symplectic orthogonal complement of $M$ in $(\mathbb{R}^{2n},\omega_0)$.
Both $M$ and $M^{\omega_0}$ are symplectic subspaces in $(\mathbb{R}^{2n},\omega_0)$ and
$\mathbb{R}^{2n}=M\oplus  M^{\omega_0}$.
Define a linear map $S_2:\mathbb{R}^{2n}\to \mathbb{R}^{2n}$ by
$$
S_2|_{M}u_1=u,\quad S_2|_{M}v_1=v,\quad S_2|_{M^{\omega_0}}=id_{M^{\omega_0}}.
$$
Then $S:=S_2S_1$ is the desired linear symplectic map.

It remains to prove the second claim.
Let $u\in \mathbb{R}^{2n}$ and $v\in \mathbb{R}^{0,n}$ satisfy $\omega_0(u,v)=1$.
Then $v':=J_{2n}v\in \mathbb{R}^{n,0}$ and $u':=J_{2n}u\in \mathbb{R}^{2n}$
satisfy $\omega_0(u', v')=1$. By the first claim there exists a
$S\in {\rm Sp}(\mathbb{R}^{2n},\omega_0;\mathbb{R}^{n,0})$ (and so
$S^{-1}\in {\rm Sp}(\mathbb{R}^{2n},\omega_0;\mathbb{R}^{n,0})$)
such that $u'=S^{-1}e_j$ and $v'=S^{-1}J_{2n}e_j=S^{-1}e_{n+j}$. Since $(S^{-1})^TJ_{2n}=J_{2n}S$ and $(S^{-1})^T=(S^T)^{-1}$ we have
\begin{eqnarray*}
&&v=J_{2n}^{-1}v'=J_{2n}^{-1}S^{-1}J_{2n}e_j= -J_{2n}S^{-1}J_{2n}e_j=  -(S^TJ_{2n})J_{2n}e_j=S^Te_j,\\
&&u=J_{2n}^{-1}u'=J_{2n}^{-1}S^{-1}e_j=-J_{2n}S^{-1}e_j=-S^TJ_{2n}e_j.
\end{eqnarray*}
\end{proof}

\noindent{\bf Proof of Theorem~\ref{th:J-estimate}}.
By (\ref{e:V0}) and (\ref{e:V1}) we have
 the symplectic decomposition of $\mathbb{R}^{2n}$,
$\mathbb{R}^{2n}=V^{n,k}_1+ (V^{n,k}_0+ J_{2n}V^{n,k}_0)$.
Let $D_1$ and $D_0$ be the projections of $D$ onto subspaces $V^{n,k}_1$ and $(V^{n,k}_0+ J_{2n}V^{n,k}_0)$, respectively.
Both are also  centrally symmetric.
After $V^{n,k}_1+ (V^{n,k}_0+ J_{2n}V^{n,k}_0)$ is identified with $\mathbb{R}^{2k}\times\mathbb{R}^{2(n-k)}$ there holds
$D\subset D_1\times D_0$. Thus
$$
c_{\rm LR}(D,D\cap \mathbb{R}^{n,k})\le c_{\rm LR}(D_1\times D_0, (D_1\times D_0)\cap \mathbb{R}^{n,k})
$$
if $k>0$.
In this situation, by (\ref{e:product1}) we obtain
 \begin{eqnarray}\label{e:J-1}
&& c_{\rm LR}(D_1\times D_0, (D_1\times D_0)\cap \mathbb{R}^{n,k})\nonumber\\
&&=\min\{{c_{\rm LR}(D_1, D_1\cap\mathbb{R}^{k,k})},   c_{\rm LR}(D_0, D_0\cap\mathbb{R}^{n-k,0})\},
 \end{eqnarray}
where  {$c_{\rm LR}(D_1, D_1\cap\mathbb{R}^{k,k})$} is equal to the Ekeland-Hofer-Zehnder capacity
$c_{\rm EHZ}(D_1)$ of $D_1$ (\cite[\S1.2]{JinLu1918}).
Let us estimate
 \begin{eqnarray}\label{e:J-2}
  c_{\rm LR}(D_0, D_0\cap\mathbb{R}^{n-k,0}).
 \end{eqnarray}
 We understand $D_0=D$ if $k=0$.
Let $\Pi_j$ denote the natural orthogonal projections from $\mathbb{R}^{2{n-k}}$
onto the subspace $E_j=\{(q_1,\cdots,q_{n-k},p_1,\cdots,p_{n-k})\in\mathbb{R}^{2(n-k)}\;|\; q_s=p_s=0\;\forall s\ne j\}$,
$j=1,\cdots,n-k$. For $S\in {\rm Sp}(\mathbb{R}^{2(n-k)},\omega_0;\mathbb{R}^{n-k,0})$, we have $S(D_0)\subset \times^{n-k}_{j=1}\Pi_j(S(D_0))$
after $\underbrace{\mathbb{R}^2\times\cdots\times\mathbb{R}^2}_{n-k\;\mbox{\scriptsize factors}}$
is identified with $\mathbb{R}^{2(n-k)}$ as below (\ref{e:product1}).
Using (\ref{e:product1}) again we arrive at
\begin{eqnarray}\label{e:J-3}
 c_{\rm LR}(D_0, D_0\cap\mathbb{R}^{n-k,0})
 &\le&
 c_{\rm LR}(\times^{n-k}_{j=1}\Pi_j(S(D_0)),  (\times^{n-k}_{j=1}\Pi_j(S(D_0)))\cap \mathbb{R}^{n-k,0} )\nonumber\\
 &=&\min_jc_{\rm LR}(\Pi_j(S(D_0))), \Pi_j(S(D_0)))\cap\mathbb{R}^{1,0})
 \end{eqnarray}
for any $S\in{\rm Sp}(\mathbb{R}^{2(n-k)},\omega_0;\mathbb{R}^{n-k,0})$ because
$$
c_{\rm LR}(D_0, D_0\cap\mathbb{R}^{n-k,0})= c_{\rm LR}(S(D_0), S(D_0)\cap\mathbb{R}^{n-k,0}).
$$
By the arguments in Remark~\ref{rem:R10case} we see that
$c_{\rm LR}(\Pi_j(S(D_0)), \Pi_j(S(D_0))\cap\mathbb{R}^{n-k,0})$ is equal to the smallest one of
\begin{eqnarray*}
&&{\rm Area}\left(\Pi_j(S(D_0))\cap\{(q_1,p_1)\in\mathbb{R}^{2}\;|\;q_1\ge 0\}\right)\;\hbox{and}\;\\
&&{\rm Area}\left(\Pi_j(S(D_0))\cap\{(q_1,p_1)\in\mathbb{R}^{2}\;|\;q_1\le 0\}\right).
\end{eqnarray*}
It follows from (\ref{e:J-3}) that
 \begin{eqnarray}\label{e:J-4}
 c_{\rm LR}(D_0, D_0\cap\mathbb{R}^{n-k,0})
 \le\frac{1}{2}\inf_{S\in {\rm Sp}(\mathbb{R}^{2(n-k)},\omega_0;\mathbb{R}^{n-k,0})}\min_j {\rm Area}\left(\Pi_j(S(D_0))\right).
 \end{eqnarray}

 Let us fix a $j\in\{1,\cdots,n-k\}$ and estimate
 $$
 \inf_{S\in {\rm Sp}(\mathbb{R}^{2(n-k)},\omega_0;\mathbb{R}^{n-k,0})}\min_j {\rm Area}\left(\Pi_j(S(D_0))\right).
 $$
 By \cite[Lemma~2.6]{GlOs}, for each $S\in {\rm Sp}(\mathbb{R}^{2(n-k)},\omega_0;\mathbb{R}^{n-k,0})$ we have
  \begin{eqnarray}\label{e:J-5}
{\rm Area}\left(\Pi_j(S(D_0))\right)\le 4\|S^Te_j\|_{\overline{D_0}^\circ}\|S^TJ_{2(n-k)}e_j\|_{(\overline{D_0})^\circ_v}
 \end{eqnarray}
 with $v=S^Te_j$, where $\overline{D_0}^\circ=\{y\in\mathbb{R}^{2(n-k)}\;|\;\langle x, y\rangle\le 1\;\forall x\in\overline{D_0}\}$ and
\begin{eqnarray}\label{e:J-5.1}
(\overline{D_0})_v=\overline{D_0}\cap\{v\}^\bot\quad\hbox{and}\quad\|w\|_{(\overline{D_0})^\circ_v}=\sup\{\langle w,y\rangle\;|\;y\in (\overline{D_0})_v\}.
 \end{eqnarray}
 where $\{v\}^\bot$ is the orthogonal complement of $\mathbb{R}v$ in $(\mathbb{R}^{2n},\langle\cdot,\cdot\rangle)$.
 By Lemma~\ref{lem:Lagr}
 if $u\in \mathbb{R}^{2(n-k)}$ and $v\in \mathbb{R}^{0,n-k}$ satisfy $\omega_0(u,v)=1$
then for each $e_{j}$, $j=1,\cdots,n-k$, there exists a $S\in {\rm Sp}(\mathbb{R}^{2(n-k)},\omega_0;\mathbb{R}^{n-k,0})$ such that
$v=S^Te_j$ and $u=-S^TJ_{2n}e_j$. Denote by $\mathbb{S}(\mathbb{R}^{0,n-k})$ the unit sphere in $\mathbb{R}^{0,n-k}$. Then we have
  \begin{equation}\label{e:J-6}
\inf_{S\in {\rm Sp}(\mathbb{R}^{2(n-k)},\omega_0;\mathbb{R}^{n-k,0})}{\rm Area}\left(\Pi_j(S(D_0))\right)
\le 4\inf_{v\in \mathbb{S}(\mathbb{R}^{0,n-k})}\inf_{u\in\Sigma_v}\|v\|_{\overline{D_0}^\circ}\|u\|_{(\overline{D_0})^\circ_v}
 \end{equation}
where $\Sigma_v=\{u\in \mathbb{R}^{2(n-k)}\:|\;\langle J_{2(n-k)}v,u\rangle=1\}$.
For each fixed $v\in \mathbb{S}(\mathbb{R}^{0,n-k})$, as in the proof of {\cite[Equation (17) in the proof of Proposition 2.5]{GlOs}}, we can obtain
 \begin{eqnarray}\label{e:J-7}
&&\inf_{u\in\Sigma_v}\|u\|_{(\overline{D_0})^\circ_v}=\sup\{\langle w, J_{2(n-k)}v\rangle\;|\;w\in\mathbb{R}J_{2(n-k)}v\;\&\;
\|w\|_{(\overline{D_0})_v}\le 1\}\nonumber\\
&=&\sup\left\{\langle aJ_{2(n-k)}v, J_{2(n-k)}v\rangle\;\Big|\;a\in\mathbb{R}\;\&\; |a|\le
\left(\|J_{2(n-k)}v\|_{(\overline{D_0})_v}\right)^{-1}\right\}\nonumber\\
&=&\frac{1}{\|J_{2(n-k)}v\|_{(\overline{D_0})_v}}
 \end{eqnarray}
because $\langle J_{2(n-k)}v, J_{2(n-k)}v\rangle=\langle v, v\rangle=1$. Note that $J_{2(n-k)}v\in\{v\}^\bot$.
We deduce
\begin{eqnarray*}
\|J_{2(n-k)}v\|_{(\overline{D_0})_v}&=&\inf\{r>0\,|\, J_{2(n-k)}v\in r(\overline{D_0})_v\}\\
&=&\inf\{r>0\,|\, J_{2(n-k)}v\in r\overline{D_0}\}=\|J_{2(n-k)}v\|_{\overline{D_0}}.
\end{eqnarray*}
 It follows from this equality, (\ref{e:J-4}), (\ref{e:J-5}) and (\ref{e:J-6}), (\ref{e:J-7}) that
\begin{eqnarray}\label{e:J-8}
 c_{\rm LR}(D_0, D_0\cap\mathbb{R}^{n-k,0})
 &\le& 2\inf_{v\in \mathbb{S}(\mathbb{R}^{0,n-k})}\frac{\|v\|_{\overline{D_0}^\circ}}{\|J_{2(n-k)}v\|_{\overline{D_0}}}\nonumber\\
 &=& \frac{2}{\|J_{2(n-k)}\|_{\overline{D_0}^\circ\cap\mathbb{R}^{0,n-k}\to \overline{D_0}}},
  \end{eqnarray}
where $\|J_{2(n-k)}\|_{\overline{D_0}^\circ\cap\mathbb{R}^{0,n-k}\to \overline{D_0}}$ is the norm of
$J_{2(n-k)}$ as a linear map between the normed spaces $(\mathbb{R}^{0,n-k}, \|\cdot\|_{\overline{D_0}^\circ\cap\mathbb{R}^{0,n-k}})$
and $(\mathbb{R}^{2(n-k)}, \|\cdot\|_{\overline{D_0}})$.

Recall that $\mathbb{R}^{2(n-k)}$ is identified with the symplectic subspace
$V^{n,k}_0+ J_{2n}V^{n,k}_0\subset\mathbb{R}^{2n}$ and that
$D_0$ is the image of $D$ under the natural orthogonal projection $\Pi_{n,k}$ from $\mathbb{R}^{2n}$ onto the subspace $(V^{n,k}_0+ J_{2n}V^{n,k}_0)$. We can understand $\overline{D_0}^\circ$ as
\begin{eqnarray*}
&&\{y\in V^{n,k}_0+ J_{2n}V^{n,k}_0\;|\;\langle \Pi_{n,k}x, y\rangle\le 1\;\forall x\in\overline{D}\}\\
&=&\{y\in V^{n,k}_0+ J_{2n}V^{n,k}_0\;|\;\langle x, y\rangle\le 1\;\forall x\in\overline{D}\}\\
&=&(V^{n,k}_0+ J_{2n}V^{n,k}_0)\cap(\overline{D})^\circ
\end{eqnarray*}
and thus $\overline{D_0}^\circ\cap\mathbb{R}^{0,n-k}$ as $(J_{2n}V^{n,k}_0)\cap(\overline{D})^\circ$. Hence
\begin{eqnarray}\label{e:J-8.1}
\|J_{2(n-k)}\|_{\overline{D_0}^\circ\cap\mathbb{R}^{0,n-k}\to \overline{D_0}}=
 \|J_{2n}|_{J_{2n}V^{n,k}_0}\|_{\overline{D}^\circ\to\overline{D}},
 \end{eqnarray}
where the right side is the norm  of $J_{2n}|_{J_{2n}V^{n,k}_0}$ as a linear map between the normed spaces
$$
(J_{2n}V^{n,k}_0, \|\cdot\|_{(J_{2n}V^{n,k}_0)\cap(\overline{D})^\circ})=
(J_{2n}V^{n,k}_0, \|\cdot\|_{(\overline{D})^\circ}|_{J_{2n}V^{n,k}_0})
\quad\hbox{and}\quad (\mathbb{R}^{2n}, \|\cdot\|_{\overline{D}}).
$$
 Hence (\ref{e:J-estimate3}) follows from (\ref{e:J-8}).

Finally, we prove equality (\ref{e:J-estimate4}).
By \cite[Remark~1.7.8]{Sch93},
for a norm $\|\cdot\|_K$ on $\mathbb{R}^m$ defined by a centrally symmetric convex body $K\subset\mathbb{R}^m$,
after the dual space of $\mathbb{R}^m$ is identified with $\mathbb{R}^m$ itself via
inner product $\langle\cdot,\cdot\rangle$, the dual Banach space $(\mathbb{R}^m, \|\cdot\|_K^\ast)$
of $(\mathbb{R}^m, \|\cdot\|_K)$ has its norm given by
$\|y\|_K^\ast=\sup\{\langle y,x\rangle\,|\,x\in\mathbb{R}^m,\;\|x\|_K\le 1\}=\|y\|_{K^\circ}$.
 Then we have
$$
\|J_{2n}v\|_{\overline{D}}=\|J_{2n}v\|_{(\overline{D}^\circ)^\circ}=\sup\{\langle J_{2n}v, u\rangle\,|\,u\in\mathbb{R}^{2n}\;\&\;
\|u\|_{\overline{D}^\circ}\le 1\}.
$$
But $\overline{D}^\circ=\{u\in\mathbb{R}^{2n}\;|\;
\|u\|_{\overline{D}^\circ}\le 1\}$ and
$$
(J_{2n}V^{n,k}_0)\cap\overline{D}^\circ=\{v\in J_{2n}V^{n,k}_0\;|\;\|v\|_{\overline{D}^\circ}\le 1\}.
$$
Hence
\begin{eqnarray*}
&&\|J_{2n}|_{J_{2n}V^{n,k}_0}\|_{\overline{D}^\circ\to\overline{D}}=
\sup\{\|J_{2n}v\|_{\overline{D}}\;|\;v\in J_{2n}V^{n,k}_0\setminus\{0\},\;\|v\|_{\overline{D}^\circ}\le 1\}\\
&&=\sup\{\langle J_{2n}v, u\rangle\,|\,u\in\mathbb{R}^{2n}\;\&\;
\|u\|_{\overline{D}^\circ}\le 1,\;v\in J_{2n}V^{n,k}_0\setminus\{0\},\;\|v\|_{\overline{D}^\circ}\le 1\}\\
&&=\sup\{\langle J_{2n}v,u\rangle\;|\;  u\in \overline{D}^\circ,\;
  v\in (J_{2n}V^{n,k}_0)\cap\overline{D}^\circ\}.
\end{eqnarray*}
The equality (\ref{e:J-estimate4}) is proved.
\hfill$\Box$\vspace{2mm}

\noindent{\bf Proof of Example~\ref{ex:add}}.
By \cite{AAO14} the classical billiard trajectory given by the line segment between
$(-1,0)$ and $(0,1)$ is the projection of a leafwise chord $\gamma$, with respect to the Lagrangian submanifold $(E^2(1,2)\times D^2(1))\cap\mathbb{R}^{2,0}$, on the boundary of $E^2(1,2)\times D^2(1)$ with action $A(\gamma)=2$. We obtain
$$
c_{\rm LR}(E^2(1,2)\times D^2(1),(E^2(1,2)\times D^2(1))\cap\mathbb{R}^{2,0})\le 2
$$
and hence equality (\ref{e:LargP1-Add2}) because
\begin{eqnarray*}
2&=&c_{\rm LR}(D^2(1)\times D^2(1),(E^2(1,2)\times D^2(1))\cap\mathbb{R}^{2,0})\\
&\le& c_{\rm LR}(E^2(1,2)\times D^2(1),(E^2(1,2)\times D^2(1))\cap\mathbb{R}^{2,0}).
\end{eqnarray*}
Similar arguments yield equality (\ref{e:LargP1-Add3}).

In order to prove inequality (\ref{e:LargP1-Add4}) consider the following convex domain in $\mathbb{R}^4$,
$$
C:=\{(x_1,x_2,y_1,y_2)\,|\,x_1^2+\frac{x_2^2}{4}+y_1^2+y_2^2\le 1\}\subset E^2(1,2)\times D^2(1),
$$
 which intersects with $\mathbb{R}^{2,1}$. By the monotonicity of $c_{\rm LR}$ there holds
\begin{equation}\label{e:7.1}
c_{\rm LR}(C,C\cap\mathbb{R}^{2,1})\le c_{\rm LR}(E^2(1,2)\times D^2(1),E^2(1,2)\times D^2(1)\cap\mathbb{R}^{2,1}).
\end{equation}
Let $H_C:=j_C^2=x_1^2+\frac{x_2^2}{4}+y_1^2+y_2^2$. Consider a leafwise chord  on $\partial C$ with respect to $\mathbb{R}^{2,1}$, $z:[0,T]\to\partial C$ satisfying
\begin{equation}\label{e:7.2}
\dot{z}=J_4\nabla H_C(z),\quad z(0), z(T)\in\mathbb{R}^{2,1},\quad z(0)\sim z(T).
\end{equation}
Then $A(z)=T$. Write $z(t)=(x_1(t),x_2(t),y_1(t),y_2(t))$.
Note that the first equation in (\ref{e:7.2}) is equivalent to
$$
\dot{x}_1=-2y_1,\quad\dot{y}_1=2x_1,\quad\hbox{and}\quad\dot{x}_2=-2y_2,\quad
\dot{y}_2=\frac{x_2}{2}.
$$
It follows that $x_1(t)^2+y_1(t)^2\equiv c_1$ and $\frac{x_2(t)^2}{4}+y_2(t)^2\equiv c_2$,
 where $c_1$, $c_2\ge 0$ are constants and $c_1+c_2=1$ since $z(t)$ is on the boundary of $C$.

The second and third conditions in (\ref{e:7.2}) are equivalent to
$$
\left\{
\begin{array}{c}
x_1(0)=x_1(T)\\
y_1(0)=y_1(T),
\end{array}
\right.
\quad\hbox{and}\quad
y_2(0)=y_2(T)=0.
$$

If $c_1\ne 0$ then $T\ge\pi$, i.e. $A(z)\ge\pi$.

If $c_1= 0$ then $x_1(t)\equiv 0$, $x_1(t)\equiv 0$ and $c_2=1$. Hence $z_2(t)=(x_2(t),y_2(t))$ is a leafwise chord on the boundary of ellipsoid
$\hat{E}^2(2,1)=\{(x_2,y_2)\,|\,{x_2^2}/{4}+y_2^2\le 1\}$
in $\mathbb{R}^2=\{(x_2,y_2)\}$, with respect to the Lagrangian submanifold $\mathbb{R}^{1,0}=\mathbb{R}\times\{0\}$. Hence
$$
\int_0^T \langle -J_2\dot{z}_2,z_2\rangle dt\ge\pi,
$$
half of the area of the ellipsoid $\hat{E}^2(2,1)$. It follows that
$A(z)=\pi$.

In summary,  $A(z)\ge\pi$ for any leafwise chord $z$ on $C$ with respect to $\mathbb{R}^{2,1}$. That is,
$c_{\rm LR}(C,C\cap \mathbb{R}^{2,1})\ge\pi$.
This and inequality (\ref{e:7.1}) lead to (\ref{e:LargP1-Add4}).
\hfill$\Box$\vspace{2mm}

 \section{An analogue of Viterbo's  conjecture}\label{sec:viterbo}
\setcounter{equation}{0}

\begin{guess}[Viterbo \cite{Vit00}]\label{conj:viterbo}
{\rm For any symplectic capacity $c$ and any convex body $D\subset\mathbb{R}^{2n}$ there holds
\begin{equation}\label{e:viterbo}
\frac{c(D)}{c(B^{2n}(1))}\le \left(\frac{{\rm Vol}(D)}{{\rm Vol}(B^{2n}(1))}\right)^{1/n}
\end{equation}
(or equivalently $(c(D))^n\le{\rm Vol}(D,\omega_0^n)=n!{\rm Vol}(D)$),
with equality if and only if $D$ is symplectomorphic to a ball, where ${\rm Vol}(D)$ denotes the Euclidean volume of $D$.
In other words, \textsf{among all convex bodies in $\mathbb{R}^{2n}$ with
a given volume, the symplectic capacity is maximal for symplectic images of the Euclidean ball.}}
\end{guess}

Clearly, this is true if $c$ is equal to the Gromov symplectic width $w_G$.
 Hermann \cite{Her98} proved (\ref{e:viterbo}) for convex Reinhardt domains $D$.
As an improvement of a result of Viterbo \cite{Vit00} it was proved in
  \cite{AAMO08} that  this conjecture holds true up to a multiplicative constant that is independent of the dimension.
So far this conjecture is only proved for some cases, see \cite{Ba20, GuRa20, KaRo19, ShiLu2008} and the references therein.
Surprisingly, by proving $c_{\rm HZ}(\Delta\times \Delta^\circ)=4$ for every centrally symmetric convex body $\Delta\subset\mathbb{R}^n_q$,
Artstein-Avidan,  Karasev, and Ostrover \cite{AAKO14} observed  that
Conjecture~\ref{conj:viterbo} for $c=c_{\rm HZ}$ and $D=\Delta\times \Delta^\circ$  implies the
the famous symmetric Mahler conjecture \cite{Mah39} in convex geometry:
\textsf{For any centrally symmetric convex body $\Delta\subset\mathbb{R}^n_q$,
there holds ${\rm Vol}(\Delta\times\Delta^\circ)\ge 4^n/n!$.}
It is pleasant that  Iriyeh and Shibata \cite{IrSh20}
have very recently proved the latter in the case $n=3$, and hence
(\ref{e:viterbo}) for $c=c_{\rm HZ}$ and $D=\Delta\times \Delta^\circ\subset\mathbb{R}^6$
with every centrally symmetric convex body $\Delta\subset\mathbb{R}^3_q$.

What are analogues of  the Viterbo conjecture for the coisotropic  capacities?
For simplicity we only consider the case $k=n-1$.
 As arguments above (\ref{e:BDY.1}), if $n=1$ and $k=0$
then  $c_{\rm LR}(D, D\cap\mathbb{R}^{1,0})$ is equal to the smaller symplectic area of
$D$ above and below the the line segment $D\cap\mathbb{R}^{1,0}$.
Motivated by this and the above Viterbo conjecture,
 for any convex body $D\subset\mathbb{R}^{2n}$
whose interior has nonempty intersection with $\mathbb{R}^{n,n-1}$, we conjecture:
\begin{equation}\label{e:viterbo2}
\frac{c_{\rm LR}(D, D\cap\mathbb{R}^{n,n-1})}{c_{\rm LR}(B^{2n}(1),  B^{2n}(1)\cap\mathbb{R}^{n,n-1})}\le 2\left(\frac{\min\{{\rm Vol}(D^+),
{\rm Vol}(D^-)\}}{{\rm Vol}(B^{2n}(1))}\right)^{1/n}
\end{equation}
or equivalently
\begin{equation}\label{e:viterbo2*}
c_{\rm LR}(D, D\cap\mathbb{R}^{n,n-1})\le \left(n!{\min\{{\rm Vol}(D^+),
{\rm Vol}(D^-)\}}\right)^{1/n},
\end{equation}
where $D^\pm:=\{x=(q_1,\cdots,q_n,p_1,\cdots,p_n)\in D\,|\,\pm p_n\ge 0\}$, i.e., two parts
of $D$ separated by hyperplane $\mathbb{R}^{n,n-1}\subset\mathbb{R}^{2n}$.

If $D$ is either centrally symmetric or  $\tau_0$-invariant, then the volume-preserving diffeomorphism
$(q,p)\mapsto (-q,-p)$ or $\tau_0:\mathbb{R}^{2n}\to \mathbb{R}^{2n}$ maps $D^+$ onto $D^-$, and hence
(\ref{e:viterbo2*}) becomes
\begin{equation}\label{e:viterbo2**}
c_{\rm LR}(D, D\cap\mathbb{R}^{n,n-1})\le \left(\frac{n!}{2}{\rm Vol}(D)\right)^{1/n}.
\end{equation}

In order to support our conjecture let us check examples we have computed.

\begin{example}\label{ex:Vit2}
{\rm Let  $D:=\prod^n_{i=1}(a_i, b_i)\times (-c_i, d_i)$ be as in (\ref{e:product7}). It is easy to prove that
\begin{eqnarray*}
c_{\rm LR}(D, D\cap\mathbb{R}^{n,n-1})&\le&(n!)^{\frac{1}{n}}c_{\rm LR}(D, D\cap\mathbb{R}^{n,n-1})\\
&\le&\left(n!{\min\{{\rm Vol}(D^+),{\rm Vol}(D^-)\}}\right)^{1/n}.
 \end{eqnarray*}
In particular, for any $c>0, d>0$ and reals $a<b$, if $D=(a, b)\times (-c, d)$ then (\ref{e:BDY.2}) leads to
\begin{eqnarray}\label{e:Vit2}
 c_{\rm LR}(D, D\cap\mathbb{R}^{1,0})=(b-a)\min\{c,d\}=\left(n!{\min\{{\rm Vol}(D^+),
{\rm Vol}(D^-)\}}\right)^{1/n}
  \end{eqnarray}
since $n=1$. Hence equality cases in (\ref{e:viterbo2*}) are possible.
}
\end{example}

\begin{example}\label{ex:Vit2.1}
{\rm Let $D=P^{2n}(r_1,\cdots,r_n)$ be as in
Corollary~\ref{cor:polyDisk}. Then
$$
{\rm Vol}(D^+)={\rm Vol}(D^-)=\frac{\pi^n}{2}r_1^2\cdots r_n^2,
$$
and
 \begin{eqnarray}\label{e:Vit3}
 c_{\rm LR}(D, D\cap\mathbb{R}^{n,n-1})=\frac{\pi}{2}\min\{2\min_{i\le n-1}r_i^2, r^2_n\}
 \end{eqnarray}
by (\ref{e:product4}). If $n=1$  then
$c_{\rm LR}(D, D\cap\mathbb{R}^{1,0})=\frac{1}{2}{\rm Vol}(D)$ by (\ref{e:1ellipsoid+}) or the arguments above (\ref{e:BDY.1}).
Suppose that $n>1$ and  $2\min_{i\le n-1}r_i^2=2r_j^2\ge r^2_n$ for some $0<j<n$, then
$$
{\rm Vol}(D^+)=\frac{\pi^n}{2}r_1^2\cdots r_n^2\ge \frac{\pi^n}{2^n}(r_n^2)^n= \left(c_{\rm LR}(D, D\cap\mathbb{R}^{n,n-1})\right)^n.
$$
Similarity, if $2\min_{i\le n-1}r_i^2=2r_j^2<r^2_n$ for some $0<j<n$, then
$$
{\rm Vol}(D^+)=\frac{\pi^n}{2}r_1^2\cdots r_n^2\ge (\pi r_j^2)^n=\left(c_{\rm LR}(D, D\cap\mathbb{R}^{n,n-1})\right)^n.
$$
In summary, we have
\begin{equation}\label{e:Vit4}
c(D, D\cap\mathbb{R}^{n,n-1})
\le \left(\frac{1}{2}{\rm Vol}(D)\right)^{1/n}\le \left(\frac{n!}{2}{\rm Vol}(D)\right)^{1/n}.
\end{equation}
That is, (\ref{e:viterbo2**}) holds in this situation.}
\end{example}

\begin{example}\label{ex:Vit3}
{\rm
For $D=E(r_1,\cdots,r_n)$ in  (\ref{e:ellipsoid-})
  we have ${\rm Vol}(D)=\frac{\pi^n}{n!}r_1^2\cdots r_n^2$ and
\begin{eqnarray*}
c_{\rm LR}\left(D, D\cap\mathbb{R}^{n,n-1}\right)=\frac{\pi}{2}\min\{2\min_{i\le n-1}r_i^2, r^2_n\}
   \end{eqnarray*}
 by (\ref{e:ellipsoid}).
If $2\min_{i\le n-1}r_i^2\ge r^2_n$, then $2r_i^2\ge r_n^2$ for $i=1,\cdots,n-1$, and so
$$
(2^{n-1}r_1^2\cdots r_n^2)^{1/n}\ge r_n^2=\min\{2\min_{i\le n-1}r_i^2, r^2_n\}.
$$
If $2\min_{i\le n-1}r_i^2<r^2_n$, we may assume $2\min_{i\le n-1}r_i^2=2r_1^2$. Then
$$
(2^{n-1}r_1^2\cdots r_n^2)^{1/n}> ((2r_1^2)^{n-1}2r_1^2)^{1/n}=2r_1^2=\min\{2\min_{i\le n-1}r_i^2, r^2_n\}.
$$
It follows from these that
\begin{eqnarray}\label{e:Vit5}
&&\frac{c_{\rm LR}(E(r_1,\cdots,r_n), E(r_1,\cdots,r_n)\cap\mathbb{R}^{n,n-1})}{c_{\rm LR}(B^{2n}(1),  B^{2n}(1)\cap\mathbb{R}^{n,n-1})}\nonumber\\
&\le& 2\left(\frac{\min\{{\rm Vol}(E(r_1,\cdots,r_n)^+),
{\rm Vol}(E(r_1,\cdots,r_n)^-)\}}{{\rm Vol}(B^{2n}(1))}\right)^{1/n},
\end{eqnarray}
with equality if and only if $r_1=\cdots=r_{n-1}=r_n/\sqrt{2}$.
}
\end{example}

\begin{example}\label{ex:Vit4}
{\rm Let $D_a=B^{2n}({\bf a}, 1)$  be as in (\ref{e:Ball1}) with $a\in [0,1)$, $r=\sqrt{1-a^2}$
and $\theta(r)=\arcsin(r)$. A long calculus exercise shows
\begin{eqnarray}\label{e:Vit6}
{\rm Vol}(D_a^-)=\frac{\pi^{n-1}}{n!}\left(\theta(r)-\cos\theta(r)\sum^{n-1}_{j=0}
\frac{(2n-2j-2)!!}{(2n-2j-1)!!}\sin^{2n-2j-1}\theta(r)\right)
\end{eqnarray}
and ${\rm Vol}(D_a^-)\le {\rm Vol}(D_a^+)$. Here $(2n+1)!!=1\cdot 3\cdot 5\cdots (2n+1)$,
$(2n)!!=2\cdot 4\cdot 6\cdots (2n)$, $0!!=0$ and $(-1)!!=0$.
 By (\ref{e:1ellipsoid+}) we have
$$
c_{\rm LR}\left(D_a, D_a\cap\mathbb{R}^{n,n-1}\right) =\theta(r)-\frac{1}{2}\sin(2\theta(r))=\frac{2}{3}(\theta(r))^3+
o((\theta(r))^4)
$$
as $\theta(r)\to 0$. Note that $\theta(r)\to 0$ as $a\to 1$. It follows that
$$
\lim_{a\to 1}\frac{n!\min\{{\rm Vol}(D_a^+),
{\rm Vol}(D_a^-)\}}{(c_{\rm LR}\left(D_a, D_a\cap\mathbb{R}^{n,n-1}\right))^n}=+\infty
$$
If $a\in (-1,0]$ then $D_a^-=D_{-a}^+$, $D_a^+=D_{-a}^-$ and so
$$
\lim_{a\to -1}\frac{n!\min\{{\rm Vol}(D_a^+),
{\rm Vol}(D_a^-)\}}{(c_{\rm LR}\left(D_a, D_a\cap\mathbb{R}^{n,n-1}\right))^n}=+\infty.
$$
These show that (\ref{e:viterbo2*}) holds for $D=D_a$ with $a\in (-1,0]\cup[0,1)$ and $|a\pm 1|$ small.
}
\end{example}

\begin{example}\label{ex:Vit5}
{\rm Let
$\phi:(B^{2n}(r), B^{2n}(r)\cap\mathbb{R}^{n,n-1})\to (D, D\cap\mathbb{R}^{n,n-1})$ be
a relative symplectic embedding respecting  the leaf relations on $B^{n,n-1}(r)$ and $D\cap\mathbb{R}^{n,n-1}$.
By (\ref{e:Rwidth-c}) we get
\begin{eqnarray*}
&&\frac{\pi r^2}{{\it w}_G(B^{2n}(1)\cap\mathbb{R}^{n,n-1};B^{2n}(1), \omega_0)}=
\left(\frac{{\rm Vol}(B^{2n}(r))}{{\rm Vol}(B^{2n}(1))}\right)^{1/n}\\
&&\qquad= \left(\frac{2\min\{{\rm Vol}(B^{2n}(r)^+),
{\rm Vol}(B^{2n}(r)^-)\}}{{\rm Vol}(B^{2n}(1))}\right)^{1/n}\\
&&\qquad= \left(\frac{2\min\{{\rm Vol}(\phi(B^{2n}(r))^+),
{\rm Vol}(\phi(B^{2n}(r))^-)\}}{{\rm Vol}(B^{2n}(1))}\right)^{1/n}\\
&&\qquad\le \left(\frac{2\min\{{\rm Vol}(D^+),
{\rm Vol}(D^-)\}}{{\rm Vol}(B^{2n}(1))}\right)^{1/n},
\end{eqnarray*}
and hence
{\small\begin{eqnarray}\label{e:Vit7}
\frac{{\it w}_G(D\cap\mathbb{R}^{n,n-1};D, \omega_0)}{{\it w}_G(B^{2n}(1)\cap\mathbb{R}^{n,n-1};B^{2n}(1), \omega_0)}
\le \left(\frac{2\min\{{\rm Vol}(D^+),
{\rm Vol}(D^-)\}}{{\rm Vol}(B^{2n}(1))}\right)^{1/n},
\end{eqnarray}}
which shows, in particular,  (\ref{e:viterbo2}) holding if $c_{\rm LR}$ is replaced by
the coisotropic  capacity ${\it w}_G/2$.}
\end{example}

\vspace{5mm}
%
%


\noindent{\bf Data availability statement}  Data sharing not applicable to this article as no
datasets were generated or analyzed during the current study.\\

\noindent{\bf Declarations}\\

\noindent{\bf Conflict of interest}  The authors declare no conflict of interest.\\

\noindent{\bf Consent for publication}  Not applicable.\\

\noindent{\bf Human and animal ethics}  Not applicable.\\

\noindent{\bf Consent to participate} Not applicable.


\subsection*{Acknowledgment}
We are also deeply grateful to the anonymous referees for giving very helpful
comments and suggestions to improve the exposition.

\end{document}